\documentclass[a4paper, DIV=13, abstracton]{scrartcl}
\usepackage{etex}
\usepackage{amsmath,amsfonts,amssymb,amsthm}
\usepackage[titles]{tocloft}
\usepackage[T1]{fontenc}            
\usepackage[british]{babel}

\usepackage{tikz}
\usetikzlibrary{patterns}
\usetikzlibrary{shapes}
\usetikzlibrary{decorations.fractals}

\usepackage{subfig}
\usepackage{color}
\usepackage[all]{xy}
\usepackage{url}
\usepackage{array}
\usepackage{booktabs}
\usepackage{multirow}
\usepackage{multicol}
\usepackage{mathtools}

\usepackage{xspace}
\newcommand\polymake{\texttt{polymake}\xspace}

\definecolor{skin}{HTML}{FFECC9}
\definecolor{pumpkin}{HTML}{FEDFA9}
\definecolor{piggy}{HTML}{FFB99D}
\definecolor{fiolet}{HTML}{CD8F9C}
\definecolor{granat}{HTML}{677081}
\definecolor{ciemnyblekit}{HTML}{91A1B8}

\definecolor{oliwkowy}{HTML}{627037}
\definecolor{ciemnazielen}{HTML}{394D2E}
\definecolor{ciemnyfiolet}{HTML}{424444}
\definecolor{mocnyfiolet}{HTML}{717299}
\definecolor{jasnyfiolet}{HTML}{B0ABCC}
\definecolor{bladyfiolet}{HTML}{C9C7DB}

\definecolor{quec}{rgb}{0.8,.000,.278}      
\definecolor{ansc}{rgb}{.200,0.7,.400}     
\definecolor{idec}{rgb}{.400,.000,.700}     
\definecolor{todo}{rgb}{.800,.100,.300}

\newtheoremstyle{question}
  {2pt}   
  {2pt}   
  {\color{quec}\normalfont}  
  {0pt}       
  {\bfseries} 
  {.}         
  {5pt plus 1pt minus 1pt} 
  {}          
\newtheoremstyle{idea}
  {2pt}   
  {2pt}   
  {\color{idec}\normalfont}  
  {0pt}       
  {\bfseries} 
  {.}         
  {5pt plus 1pt minus 1pt} 
  {}          
\newtheoremstyle{answer} 
  {2pt}   
  {2pt}   
  {\color{ansc}\normalfont}  
  {0pt}       
  {\bfseries} 
  {.}         
  {5pt plus 1pt minus 1pt} 
  {}          
\newtheoremstyle{todo}
  {2pt}   
  {2pt}   
  {\color{todo}\normalfont}  
  {0pt}       
  {\bfseries} 
  {.}         
  {5pt plus 1pt minus 1pt} 
  {}

\numberwithin{equation}{section}

\newtheorem{theorem}[equation]{Theorem}
\newtheorem{lemma}[equation]{Lemma}
\newtheorem{proposition}[equation]{Proposition}
\newtheorem{corollary}[equation]{Corollary}

\theoremstyle{idea}

\theoremstyle{question}

\theoremstyle{answer}

\theoremstyle{todo}

\theoremstyle{remark}
\newtheorem{remark}[equation]{Remark}

\theoremstyle{definition}
\newtheorem{definition}[equation]{Definition}
\newtheorem{example}[equation]{Example}

\definecolor{ctDD}{rgb}{0.00,0.10,0.00}
\definecolor{cDtD}{rgb}{0.00,0.00,0.20}
\definecolor{cGitter}{rgb}{0.60,0.60,0.60}
\definecolor{cKlaus}{rgb}{0.15,0.40,0.03}
\definecolor{cAnnaLena}{rgb}{0.70,0.20,0.23}
\definecolor{cJarek}{rgb}{0.10,0.70,0.43}
\definecolor{loesung}{rgb}{0.6,0.10,0.33}
\definecolor{cKlausOK}{rgb}{0.3,0.40,0.33}
\definecolor{intOrange}{rgb}{1.0,.310,.0} 
                   
\newcommand{\threedim}[3]{(#1,\; #2,\; #3)}
\newcommand{\twodim}[2]{(#1,\; #2)}

\definecolor{intOrange}{rgb}{1.0,.310,.0}

\newcommand{\cR}{\mathcal{R}}
\newcommand{\cD}{\mathcal{D}}
\newcommand{\cL}{\mathcal{L}}
\newcommand{\cO}{\mathcal{O}}

\newcommand{\cE}{\mathcal{E}}

\newcommand{\cF}{\mathcal{F}}

\newcommand{\bZ}{\mathbb{Z}}

\newcommand{\bF}{\mathbb{F}}

\newcommand{\kk}{\Bbbk}

\newcommand{\PP}{\mathbb P}

\newcommand{\F}{\mathbb F}

\newcommand{\R}{\mathbb R}
\newcommand{\T}{\mathbb T}
\newcommand{\V}{\mathbb V}
\newcommand{\Z}{\mathbb Z}

\newcommand{\Q}{\mathbb Q}
\newcommand{\QQ}{\mathbb Q}

\newcommand{\CC}{{\mathcal C}}
\newcommand{\CF}{{\mathcal F}}
\newcommand{\CE}{{\mathcal E}}
\newcommand{\CG}{{\mathcal G}}

\newcommand{\CM}{{\mathcal M}}

\newcommand{\CL}{{\mathcal L}}
\newcommand{\CO}{{\mathcal O}}
\newcommand{\CP}{{\mathcal P}}

\newcommand{\CR}{{\mathcal R}}

\newcommand{\CV}{{\mathcal V}}

\newcommand{\matc}[2]{\left(\begin{array}{@{}*{#1}{c@{}}} #2
\end{array}\right)}

\newcommand{\Matc}[2]{\left(\begin{array}{@{}*{#1}{c}@{}} #2
\end{array}\right)}

\newcommand{\Matrc}[3]{\left(\begin{array}{@{}*{#1}{r}|*{#2}{c}@{}} #3
\end{array}\right)}

\newcommand{\Matrr}[3]{\left(\begin{array}{@{}*{#1}{r}|*{#2}{r}@{}} #3
\end{array}\right)}

\DeclareMathOperator{\SL}{SL}
\DeclareMathOperator{\rank}{rk}
\DeclareMathOperator{\innt}{int}

\DeclareMathOperator{\coker}{coker}

\DeclareMathOperator{\Hom}{Hom}
\DeclareMathOperator{\Div}{Div}

\DeclareMathOperator{\Pic}{Pic}
\DeclareMathOperator{\Cl}{Cl}
\DeclareMathOperator{\orb}{orb}

\DeclareMathOperator{\supp}{supp}
\DeclareMathOperator{\id}{id}

\DeclareMathOperator{\relint}{relint}
\DeclareMathOperator{\Trunc}{Trunc}

\DeclareMathOperator{\conv}{conv}
\DeclareMathOperator{\cone}{cone}

\newcommand{\kG}{\Gamma}

\newcommand{\wt}{\widetilde}

\newcommand{\ifff}{\ensuremath{\,\Longleftrightarrow\,}}
\newcommand{\then}{\ensuremath{\implies}}

\newcommand{\gH}{\operatorname{H}}
\newcommand{\tH}{\wt{\operatorname{H}}}

\newcommand{\gExt}{\mbox{\textrm{Ext}}}

\newcommand{\gHom}{\mbox{\textrm{Hom}}}

\renewcommand{\div}{\operatorname{div}}

\newcommand{\disjcup}{\sqcup} 

\newcommand{\Sl}{\operatorname{SL}}

\newcommand{\spann}{\operatorname{span}}

\newcommand{\kst}{\,|\;}
\newcommand{\kSt}{\,\big|\;}

\newcommand{\ku}{\underline}

\newcommand{\surj}{\rightarrow\hspace{-0.8em}\rightarrow}

\newcommand{\kss}{\scriptscriptstyle}
\newcommand{\kbb}{{\kss \bullet}}
\newcommand{\ko}{\overline}

\newcommand{\charact}{{\operatorname{char}}\,}

\newcommand{\dual}{^{\scriptscriptstyle\vee}}

\DeclareMathOperator{\toric}{\T\V}  
\DeclareMathOperator{\Eff}{Eff}  
\DeclareMathOperator{\Nef}{Nef}  
\DeclareMathOperator{\Imm}{Imm}  
\DeclareMathOperator{\Nuns}{Seed}

\newcommand{\Ccap}{\CC}

\newcommand{\relP}{P^{\circ}}
\newcommand{\tail}{\operatorname{tail}}
\newcommand{\suppSig}{\supp\Sigma}

\newcommand{\WQC}{$\QQ$-Cartier Weil\ }
\newcommand{\sDelta}{\Delta}  
\newcommand{\dm}{m}  
\newcommand{\rps}{u^+_\sigma} 
\newcommand{\rms}{u^-_\sigma} 
\newcommand{\rpms}{u^{\pm_\sigma}} 
\newcommand{\preRmacSet}{\Z^{\Sigma(1)\setminus\CR}_{\geq 0} 
   \times \Z^{\CR}_{\leq -1}} 
\newcommand{\preRmacRegion}{\R^{\Sigma(1)\setminus\CR}_{\geq 0} 
   \times \R^{\CR}_{\leq -1}} 
\newcommand{\pms}{\preRmacSet}  
\newcommand{\pmr}{\preRmacRegion}  
\newcommand{\kuc}{c} 
\newcommand{\kue}{\ku{1}} 
\newcommand{\kun}{\ku{0}} 
\newcommand{\koc}{\ko{c}} 
\newcommand{\kul}{\ell} 
\newcommand{\prb}{b} 
\newcommand{\pob}{\ko{b}} 
\newcommand{\prc}{c} 
\newcommand{\poc}{\ko{c}} 
\newcommand{\pre}{\ku{1}} 
\newcommand{\prn}{\ku{0}} 
\newcommand{\fes}{FES\ }

\usepackage{hyperref}

\newcommand{\set}[1]{\left\{#1\right\}}
\newcommand{\fromto}[2]{#1, \dotsc, #2}
\newcommand{\setfromto}[2]{\set{\fromto{#1}{#2}}}

\newcommand{\RR}{\mathbb{R}}

\usepackage[textsize=tiny]{todonotes}

\newcommand{\pd}{\partial}
\DeclareMathOperator{\bl}{bl}

\title
{Immaculate line bundles on toric varieties} 

\author{
Klaus Altmann
\and
Jaros{\l}aw Buczy\'{n}ski
\and
Lars Kastner
\and
Anna-Lena Winz
}

\newcommand{\hexlinediag}{\vcenter{\hbox{\begin{tikzpicture}[scale=.5]   
\draw[color=white, thick] (-0.2,-0.2) grid (1.2,1.2);
\draw[thick] (0,0) -- (1,1); 
\end{tikzpicture}}}}    

\newcommand{\hexlinehor}{\vcenter{\hbox{\begin{tikzpicture}[scale=.5]
\draw[color=white, thick] (-0.2,-0.2) grid (1.2,1.2);
\draw[thick] (0,0) -- (1,0);
\end{tikzpicture}}}}    
 
\newcommand{\hexlinevert}{ \ \vcenter{\hbox{\begin{tikzpicture}[scale=.5]
\draw[color=white, thick] (-0.2,-0.2) grid (1.2,1.2);
\draw[thick] (0,0) -- (0,1);
\end{tikzpicture}}}}    

\newcommand{\hextriangleupleft}{ \ \vcenter{\hbox{\begin{tikzpicture}[scale=.5]
\draw[color=white, thick] (-0.2,-0.2) grid (1.2,1.2);
\draw[thick] (0,0) -- (1,1) -- (0,1) -- cycle;
\end{tikzpicture}}}}    

\newcommand{\hextriangledownright}{\vcenter{\hbox{\begin{tikzpicture}[scale=.5]
\draw[color=white, thick] (-0.2,-0.2) grid (1.2,1.2);
\draw[thick] (0,0) -- (1,1) -- (1,0) -- cycle;
\end{tikzpicture}}}}

\begin{document}

\maketitle

\begin{abstract}
We call a sheaf on an algebraic variety immaculate if it lacks
any cohomology including the zero-th one, that is, if the derived
version of the global section functor vanishes. Such sheaves
are the basic tools when building exceptional sequences, 
investigating the diagonal property, or the toric Frobenius morphism.

In the present paper we focus on line bundles on toric varieties. 
First, we present a possibility of understanding their cohomology
in terms of their (generalized) momentum polytopes. Then we present
a method to exhibit the entire locus of immaculate divisors
within the class group. This will be applied to the cases of
smooth toric varieties of Picard rank two and three
and to those being given by splitting fans.

The locus of immaculate line bundles contains several linear strata of
varying dimensions. We introduce a notion of relative immaculacy
with respect to certain contraction morphisms. This 
notion will be stronger than plain immaculacy and provides an
explanation of some of these linear strata.
\end{abstract}

{\footnotesize
\noindent\textbf{Addresses:\\}
K.~Altmann, \nolinkurl{altmann@math.fu-berlin.de}, 
Institut f\"ur Mathematik, FU Berlin, Arnimalle 3, 14195 Berlin, Germany\\
J.~Buczy\'nski, \nolinkurl{jabu@mimuw.edu.pl}, 
   Institute of Mathematics of the Polish Academy of Sciences, ul. \'Sniadeckich 8, 00-656 Warsaw, Poland,
   and Faculty of Mathematics, Computer Science and Mechanics, University of Warsaw, ul.~Banacha 2, 02-097 Warszawa, Poland\\
L.~Kastner \nolinkurl{kastner@math.fu-berlin.de}, 
Institut f\"ur Mathematik, TU Berlin, Geb\"aude MA, Stra\ss e des 17.~Juni 136, 10623 Berlin, Germany\\
A.-L.~Winz, \nolinkurl{anna-lena.winz@fu-berlin.de}, Institut f\"ur Mathematik, FU Berlin, Arnimalle 3, 14195 Berlin, Germany.

\noindent\textbf{Keywords:}
toric variety, immaculate line bundle, splitting fan, toric varieties of Picard rank~3, primitive collections.

\noindent\textbf{AMS Mathematical Subject Classification 2010:}
Primary:~14M25; Secondary:~14F05, 14F17, 52B20, 14L32.}

\tableofcontents

\section{Introduction}
\label{intro}

We work over an algebraically closed field $\kk$ of any characteristic.

\subsection{Exceptional sequences ask for immaculacy}
A major tool for the process of understanding derived categories 
$\cD(X)$ on an algebraic variety $X$ is 
full exceptional sequences ``\fes''
$(\CF_1,\ldots,\CF_k)$ of sheaves or complexes. 
That is, its members are supposed to generate $\cD(X)$ and,
up to $\gHom_{\cD(X)}(\CF_i,\CF_i)=\kk$,
one asks for $\gHom_{\cD(X)}(\CF_i,\CF_j[p])=0$
for all shifts $p\in\Z$ and pairs $i\geq j$.
These conditions call to mind the shape of unitary upper triangular matrices.
If \fes exist, then they provide a semi-orthogonal 
decomposition of $\cD(X)$ into the simplest summands possible.

Whenever the $\cF_i$ are sheaves, 
then $\gHom_{\cD(X)}(\CF_i,\CF_j[p])$
can alternatively be written as the classical 
group $\gExt^p_{\CO_X}(\CF_i,\CF_j)$.
If, moreover, $\cF_i$ are locally free, e.g.\ invertible sheaves,
then this equals $\gH^p(X,\CF_i^{-1}\otimes\CF_j)$. Thus, we require 
certain sheaves $\CG=\CF_i^{-1}\otimes\CF_j$ to lack any cohomology,
including the seemingly innocent $0$-th one:
$$
\R\kG(X,\CG)=0.
$$ 
We will call this property of a sheaf $\CG$ immaculate, 
see Definition~\ref{def-immaculate} in Subsection~\ref{imaculateCheck}.

We are going to focus on invertible sheaves on smooth, projective 
varieties $X$ with $\R\Gamma(\CO_X)=\kk$. 
So, when looking for exceptional sequences of line bundles,
the case $i=j$ yielding $\CG=\CO_X$ is already taken care of.
That is, whenever we have sufficiently good knowledge of 
the locus of immaculate sheaves
within the Picard or class group $\Cl(X)$, then we can freely use its 
elements $\CG_\nu=\CO_X(D_\nu)$
as building blocks to mount exceptional sequences
via $\CF_i:=\CO_X(\sum_{\nu=2}^i D_\nu)$. The defining property 
of the vanishing $\gExt$ groups can then be understood
as asking consecutive sums of the $D_\nu$
to be immaculate, too.

The comparison of the shape of several \fes
can shed light on several features of the given variety $X$. 
Thus, the shape of the
tool box of immaculate line bundles should serve as a rich invariant.
In addition, immaculate line bundles appear in different contexts.
In \cite{achinger_toric_var_char_p} they are exploited to show a characterisation of toric varieties in terms of Frobenius splitting property.
In \cite{pragacz_srinivas_pati_diagonal_property} they are used to study the diagonal property of smooth projective varieties 
   (see for instance \cite[Thm~4]{pragacz_srinivas_pati_diagonal_property}).
For a surface of general type, the property of immaculacy of line bundles is relevant  to the spectral theory 
\cite{kulikov_zheglov_spectral_surfaces}.

\subsection{The situation on toric varieties}
\label{sitToric}
Suppose that $X$ is a smooth, projective toric variety.
The main result in this context is Kawamata's proof of the
existence of \fes of sheaves on smooth, projective
toric Deligne-Mumford stacks, see \cite{kaw1,kaw2}.
An earlier conjecture of King about the existence of 
full, strongly exceptional ($\gExt^{\geq 1}(\CF_i,\CF_j)=0$ for all $i,j$) 
sequences of line bundles was disproved in 
\cite{hillePerlingCounterKing}, \cite{michalek_Kings_conjecture}.
But, when abstaining from the additional property ``strong'',
it is still an open question whether smooth, projective toric varieties
admit \fes of line bundles, 
let alone provide an understanding of which equivariant divisors 
represented by which abstract polyhedra
will form those sequences. The only rather general, positive result
is that of \cite[Theorem~4.12]{CostaMiroRoig} where the existence of those
sequences was established for splitting fans, 
see Subsection~\ref{immSplit}. From a different viewpoint, this
was reproven for a special case in \cite{crawQuiverFlag}.

Another remarkable result can be found in \cite{HilleRatSurf}. 
There, the authors start
with an arbitrary, that is, not necessarily toric, smooth complete
rational surface and show that \fes of line bundles do always exist.
But the interesting point
is that these sequences can easily be transformed into
a cycle of divisors imitating the toric situation, that is,
to each \fes one can associate a toric surface
materialising this sequence.

\subsection{Visualizing the cohomology of toric line bundles}
In the present paper, we keep the notion of exceptional sequences in the
background. Instead, for a given projective (often smooth)
toric variety we are 
just interested in the immaculacy property of divisor classes.
Classically, the cohomology of a reflexive rank one sheaf, that is, of a Weil
divisor on a toric variety $X$ 
can be expressed in terms of special polyhedral complexes 
whose vertices are some rays of the fan $\Sigma$ of $X$. 
In particular, the complexes live in $N_\R=N\otimes \R$, where $N$ is the 
lattice of one parameter subgroups of the torus acting on $X$.

We propose a different point of view on the cohomology
of toric \WQC divisors. We will make it literally visible 
in terms of polytopes in the dual space $M_\R$. 
As usual, one writes
$M_\R=M\otimes R$ with $M = \Hom(N,\Z)$ being the monomial lattice 
of the acting torus $T$.
Since each \WQC divisor can be decomposed into a difference $D=D^+-D^-$
of nef (or even $\Q$-ample) ones, this means that the $T$-invariant 
among them can be encoded by a pair of polytopes $(\Delta^+,\Delta^-)$,
see Subsection~\ref{cohomPolyhedra} for more details
and the more general situation of semi-projective varieties.

Polytopes form a cancellative semigroup under Minkowski addition.
In this context,
the pair $(\Delta^+,\Delta^-)$ represents the formal difference 
$$
D=\Delta^+-\Delta^-
$$ 
within the Grothendieck group of generalized polytopes.
On the other hand, each $T$-invariant Weil divisor $D$ leads to a 
(possibly empty) polytope of sections $\sDelta(D)\subseteq M_\R$. 
Its lattice points parametrize the monomial basis of $\kG(X,\CO_X(D))$.
If $D$ is nef, then the pair 
consisting of $\Delta^+:=\sDelta(D)$ and $\Delta^-:=0$
can be used to represent $D$.
For general $D$ being represented by some $(\Delta^+,\Delta^-)$, one 
can still recover the polytope of sections as
$$
\sDelta(D)=\{r\in M_\R\kst \Delta^-+r\subseteq\Delta^+\},
$$
cf.~Remark~\ref{rem-recoverPolSect}.
This can be visualized as a kind of a
materialized shadow of the abstract difference $\Delta^+-\Delta^-$.

So it is quite a surprising fact that,
after using the formal difference $\Delta^+-\Delta^-$ and its shadow
$\sDelta(D)$,
the cohomology of $D=\Delta^+-\Delta^-$ can be understood by a third 
flavour, namely by the naive 
and original meaning of the set theoretic differences of these polytopes.

\begin{theorem}
On a projective toric variety $X$ the cohomology groups $\gH^i(X,\CO_X(D))$ are $M$-graded,
and for each $\dm\in M$, the homogeneous component of degree $\dm$ equals
$\tH^{i-1}(\Delta^- \setminus(\Delta^+ -\dm),\kk)$.
Here $\Delta^+ - \dm$ means the shift by $\dm$ of $\Delta^+$ in $M \otimes \R$.
\end{theorem}

See Example~\ref{ex_hexagon-someCohom} for an illustration of this claim.
The theorem is stated more generally 
as Theorem~\ref{thm_toric_cohomology_in_terms_of_nef_polytopes}
in the context of semi-projective toric varieties.
It implies that the immaculacy of $D=(\Delta^+,\Delta^-)$ can be measured by the
fact whether $\Delta^- \setminus (\Delta^+ - \dm)$ is $\kk$-acyclic
for all shifts $\dm\in M$. See Subsection~\ref{imaculateCheck} for a
discussion of the notion of being $\kk$-acyclic.

Besides its elementary geometric nature, the description of sheaf cohomology
via the defining polyhedra in the vector space $M_\R$ also has another
advantage. It allows one to think about a generalization to the more general 
setup of Okounkov bodies, as introduced in \cite{okounkovLazarsfeld}: 
   after fixing a complete flag of subspaces in an arbitrary 
   (not necessarily toric) smooth projective variety $X$, 
   convex polytopes of sections $\sDelta(\CL)$
   are assigned to each invertible sheaf $\CL$.
Thus, a description of Cartier divisors $D$ via pairs of polytopes 
$(\Delta^+,\Delta^-)$ is possible, and one can ask for the relation
between $\gH^i(X,\CO_X(D))$, 
and the cohomology of the set theoretic differences 
$\Delta^- \setminus(\Delta^+ -\dm)$.
Since in especially nice situations the Okounkov bodies induce a toric
degeneration of $X$, see \cite{anderson}, semi continuity suggests
that the latter might serve as an upper bound for the first.

\subsection{Immaculate loci for toric varieties}
The ultimate goal of this project
is to understand the structure of the set of all immaculate line bundles on
a fixed toric variety $X=\toric(\Sigma)$
as a subset of the class group of $X$.
Although some of our statements are more general, throughout this introduction
we will assume $X$ is in addition smooth and projective.

We show that in sufficiently nice situations the immaculacy is preserved under
pullback, see Proposition~\ref{prop_pullback_of_immaculate} and
Corollary~\ref{cor_pullback_of_immaculate_on_toric_varieties}).
Moreover, in Definition~\ref{def-pImmac} we introduce a relative version of 
immaculacy, and we show how this stronger version is responsible
for the presence of certain linear strata within the immaculacy locus,
see Theorems~\ref{thm_p_immaculate_and_stable_immaculate}
and~\ref{thm_line_of_immaculates_on_toric_varieties}.
However, the example of the flag variety $\F(1,2,3)$ depicted in 
Figure~\ref{fig_Pic_of_flag} shows that not all of them (here it is affine lines)
can be explained by this notion. The diagonal immaculate line 
is not induced from any map giving rise to relative immaculacy.
Some features of Corollary~\ref{cor_pullback_of_immaculate_on_toric_varieties} and Theorem~\ref{thm_line_of_immaculates_on_toric_varieties}
 are summarised as the following statement.
 
\begin{theorem}\label{thm_intro_pullbacks_and_p_immaculate}
  Suppose $X$ and $Y$ are projective toric varieties and $p \colon X \to Y$ is a surjective toric morphism with connected fibres. 
  Let $\cL$ be a line bundle on $Y$, and let $D^-$ be a nef line bundle on~$X$. 
  \begin{enumerate}
   \item  
        $\cL$ is immaculate if and only if $p^*\cL$ is immaculate.
   \item \label{item_intro_p_immaculate}
        If $\cL$ is ample on $Y$, 
        then the following conditions are equivalent:
         \begin{itemize}
          \item for infinitely many integers $a$ the divisor $a\cdot  p^*\cL - D^-$ is immaculate,
          \item $p^*\cL' - D^-$ is immaculate for any line bundle $\cL'$ on $Y$,
          \item the image of the polytope $\Delta^-$ (of sections of $D^-$)
under the quotient map $M_X \mapsto M_X/M_Y$ has no internal lattice points.  
         \end{itemize}
  \end{enumerate}
\end{theorem}

In Section~\ref{avoidTemptations} we demonstrate our principal approach
to obtain the immaculacy locus. It uses the natural map
$\pi:\Z^{\Sigma(1)}\to\Cl(X)$ assigning to each $T$-invariant divisor its
class. All non-immaculate classes,
that is, those carrying some cohomology, must be contained in some 
of the so-called $\CR$-maculate images 
$$
\CM_\Z(\CR)=\pi(\pms)
$$
for certain ``tempting'' subsets $\CR\subseteq\Sigma(1)$.
The notion of temptation is introduced in
Definition~\ref{def-tempting}; it selects those subsets such that
the induced subcomplexes of $\Sigma$ in $N_\R$ have some cohomology
after being intersected with the unit sphere.

Thus, to recognise the immaculacy locus in the Picard group involves two different problems.
First, one has to find an efficient method to identify the tempting subsets 
$\CR\subseteq\Sigma(1)$. In Subsection~\ref{sect_three_temptations}
we have collected some standard situations implying or avoiding
immediate temptation. In small examples they already suffice to check the
status of most subsets of $\Sigma(1)$.
The second problem is to keep control over the interrelation of the different
maculate sets or of their convex counterparts, the so-called maculate regions.
While a divisor class cannot be immaculate if it is touched by one single
maculate set, one has to check all of these regions for checking the opposite.
This behavior is much better around the vertices of the maculate regions
-- and this is the content of Theorem~\ref{cor-injMacCube}.

\subsection{Special situations}
After these general investigations, we turn to very concrete situations.
In Section~\ref{immSplit} we look at the situation of splitting fans,
that is, of those fans where all primitive collections 
(see Subsection~\ref{3rd_temptation} for a definition) are mutually disjoint.
While we have already remarked in Subsection~\ref{sitToric} that the
existence of \fes is known for this class, we understand this situation from a
different viewpoint -- namely by describing the entire locus of immaculate
line bundles. The main result is contained in 
Theorem~\ref{thm_immaculates_for_splitting_fan},
A special case of this class is the smooth, projective toric varieties of
Picard rank 2. We have decided to treat these varieties in a separate 
section. On the one hand, the result can be described in a very
clear manner -- we did this in 
Theorem~\ref{thm_immaculates_for_Pic_rank_2} -- and serves as a
concrete example to illustrate the more general situation of
Section~\ref{immSplit}. On the other, it is a good starting point for
the much tougher situation of Picard rank 3 coming in Section~\ref{immPicThree}.

Without going into details of the notation, the highlights of the results in Sections~\ref{PicTwo}--\ref{immPicThree} 
   can be summarised in the following theorem:
\begin{theorem}
   Suppose $X$ is a smooth projective toric variety. 
   \begin{itemize}
      \item 
         If the Picard rank of $X$ is $2$ and $X$ is not a product of projective spaces, 
         then the set of immaculate line bundles in the Picard group forms a 
union of finitely many parallel (infinite) lines 
 {\em (arising as in
Theorem~\ref{thm_intro_pullbacks_and_p_immaculate}\ref{item_intro_p_immaculate}
from a projection $p \colon X \to \PP^{\ell_1-1}$)}
         and two bounded triangles.
      \item 
         If the fan of $X$ is a splitting fan, in particular $X =X_{k}= \PP (\cL_1 \oplus \dotsb \oplus \cL_{\ell_k})$ 
            for line bundles $\cL_i$ on a smaller splitting fan variety $X_{k-1}$,
            then set of immaculate line bundles contains the pullbacks of immaculate line bundles from $X_{k-1}$, their Serre duals,
            and a family of $\ell_k -1$  hyperplanes arising as
               in Theorem~\ref{thm_intro_pullbacks_and_p_immaculate}\ref{item_intro_p_immaculate}
               from the projection $p \colon X_k \to X_{k-1}$.
         Moreover, for sufficiently ``general'' choices of  $\cL_i$, these are all immaculate line bundles on $X$ 
{\em (see Theorem~\ref{thm_immaculates_for_splitting_fan} for the exact
phrasing of the sufficiently ``general'' condition)}.  
      \item 
        If the Picard rank of $X$ is $3$ and $X$ does not have a splitting fan, 
            then the set of immaculate line bundles contains a collection of parallel lines 
            (parametrised by lattice points in the union of two parallelograms), 
            and a finite collection of bounded line segments.
        For sufficiently general 
{\em (see Proposition~\ref{prop-mainPicThree})} choices of such $X$, these are all immaculate line bundles. 
   \end{itemize}
\end{theorem}

The article concludes with Section~\ref{sec:computational}, which briefly treats the computational aspects of the approach. 

Throughout the paper the theory will be illustrated by
one running example. We call it the hexagon example since
$\Sigma$ equals the normal fan of a lattice hexagon in $\R^2$.
The associated toric variety is the del Pezzo surface of degree 6,
which equals the blowing up of $\PP^2$ in three points. In particular, it has
Picard rank 4 which makes it possible to demonstrate many possible features
explicitly. The example is spread under the names
Example~\ref{ex_hexagon_introduce_coordinates},
 \ref{ex_hexagon-someCohom},
 \ref{ex_hexagon-pullBack},
 \ref{ex_hexagon-lines},
 \ref{ex_hexagon_tempting_subsets},
 \ref{ex_hexagon_all_really_immaculate},
 \ref{ex_hexagon_first_temptation},
 \ref{ex_hexagon_second_temptation},
 \ref{ex_hexagon_third_temptation},
 and~\ref{ex_hexagon_cube}.
In addition, its immaculate locus and exceptional sequences can be completely recovered 
  from computer calculations, which are summarised in Section~\ref{sec:computational}.

\subsection*{Acknowledgements}
The project was initiated by the work within the DerivedTV
group during the special semester at the Fields Institute in 2016.
In particular, we would like to thank 
Jenia Tevelev and Barbara Bolognese for interesting discussions
and the Fields Institute for hosting us.
Many thanks also to Piotr Achinger, Weronika Buczy{\'n}ska, Oskar K{\k e}dzierski, 
   and Mateusz Micha{\l}ek for discussing several issues 
concerning exceptional sequences and immaculate line bundles.
Buczy{\'n}ski is partially supported by Polish National Science Center (NCN), project 2013/11/D/ST1/02580 
   and by a scholarship of Polish Ministry of Science.
Kastner is supported by the Collaborative Research Centre TRR195 of the German research foundation (DFG SFB/Transregio 195).

\section{Differences of polytopes}
\label{pairPol}

In Section~\ref{cohomPolyhedra} we will encode invertible sheaves on
projective toric varieties by pairs of polytopes. Then, the cohomology 
of these sheaves will
be expressed by the differences of shifts of the polytopes. Hence, we will
start with gathering some general remarks about this construction.

\subsection{Removing open subsets}
\label{sect_homotopy_of_polytopal_complexes}
Fix a real vector space, e.g.\ $\RR^d$ with the Euclidean topology. 
In this subsection we will show that certain subsets of $\RR^d$ are 
homotopy equivalent. 
In fact, in most statements below, for  $A\subset B \subset \RR^d$ we will show 
that $A$ is a \emph{strong deformation retract}  of $B$.
Recall, that a \emph{retract} is a continuous map $r\colon B \to A$, such that $r|_{A} = \id_A$, and a \emph{strong deformation retract}
  is a retract which is homotopic to the identity $\id_B$ in a way that preserves $A$, 
  that is there exists continuous $H\colon B\times [0,1] \to B$, 
such that $H|_{A\times [0,1]}(a,t) = a$, $\,H(\cdot , 0) = \id_B$, 
  and $\,H(\cdot , 1)\colon B \to A$ is the retract of $B$ to $A$.
We will mostly use the \emph{standard} ``strong deformation'', that is, once 
we have defined $r$, the standard definition of $H$ is $H(b,t)= tb + (1-t)r(b)$.
Note that this requires that the interval between $b$ and $r(b)$ is 
contained in $B$, which will often be guaranteed by some sort of convexity.
This standard way of defining $H$ will allow us to glue together several such homotopies.

For a convex subset $P \subset \RR^d$, by its \emph{span} we mean the 
smallest affine subspace containing $P$.
The \emph{relative interior} $\relP$ of $P$ is its interior as a subset 
of its span. 
Analogously, the \emph{relative boundary} $\partial P$ is the boundary of 
$P$ within $\spann P$.
Note that every convex subset of $\RR^d$ contains an open subset of its 
span, so the relative interior of non-empty
$P$ is never empty either. 

\begin{lemma}
\label{lem_difference_of_convex_polytope_and_open_convex_set}
   Let $P\subset \RR^d$ be a compact convex subset and let $Q\subset \RR^d$ be an open convex subset.
   If $P \cap Q \ne \emptyset$, 
   then $(\partial P) \setminus Q$ is a strong deformation retract of $P \setminus Q$.
\end{lemma}
\begin{proof}
   Since $Q$ is open and $P \cap Q \ne \emptyset$, 
      there exists a point $p_0 \in \relP \cap Q$.
   Define the retract $r\colon P\setminus \set{p_0} \to \partial P$ by $r(p)$ to be the unique point on the boundary $\partial P$ 
      that is contained in the semiline originating at $p_0$ and passing through $p$. 
   Since $P$ and $Q$ are convex, the standard strong deformation map $H$ is well defined, showing the claim.
\end{proof}

We adapt the convention that polyhedra are intersections of finitely many
closed halfspaces, polytopes denote bounded hence compact polyhedra,
that the empty set is a $(-1)$-dimensional face of every convex polytope, 
and that each $P$ is a face of itself.
In particular, polytopes and polyhedra are always convex.
A proper face is any face that is not $\emptyset$ or $P$.
By a \emph{(finite) polytopal complex} we mean a finite collection $\Xi$ of compact convex polytopes in $\RR^d$ satisfying the usual conditions:
\begin{itemize}
   \item if $P\in \Xi$, then every face of $P$ is in $\Xi$, and
   \item if $P_1, P_2 \in \Xi$, then $P_1 \cap P_2$ is a face of both $P_1$ and $P_2$.
\end{itemize}
Note that the \emph{support} of a polytopal complex $\Xi$, $\supp \Xi := \bigcup \set{P : P \in \Xi} \subset \RR^d$ is compact. 
A convex polytope $P$ gives rise to a natural polytopal complex $\set{F: \text{$F$ is a face of $P$}}$, whose support is $P$.

For a polytopal complex $\Xi \subset \RR^d$, and a convex subset $Q\subset \RR^d$
    we denote by $\CC(\Xi,Q)$ the polytopal complex 
    \[
      \CC(\Xi,Q)= \set{F \in \Xi \mid F\cap Q = \emptyset}. 
    \]
If $P\subset \RR^d$ is a convex polytope, then this gives rise to the
special case
    \[
      \CC(P,Q)= \set{F  \mid \text{$F$ is its face, and } F\cap Q = \emptyset}. 
    \]
For example, for
$
\begin{tikzpicture}[baseline=.4cm,scale=0.5]
\draw[step=1, ciemnyblekit, thin] (-.5,-.5) grid (2.5, 2.5);
\fill[pattern color=fiolet!20, pattern=north east lines] (0,0) -- (2,0) -- (0,2)
-- cycle;
\draw[thick, fiolet] (0,0) node[anchor=north west]{$P$} -- (2,0) -- (0,2) --
cycle;
\fill[thick, granat] (0,1) node[anchor=north west]{$Q$} circle (4pt);
\end{tikzpicture}
$
we get 
$
\begin{tikzpicture}[baseline=.4cm,scale=0.5]
\draw[step=1, ciemnyblekit, thin] (-.5,-.5) grid (2.5, 2.5);
\draw[thick, fiolet] (0,0) -- (2,0) -- (0,2);
\end{tikzpicture}
$
as $\Ccap(P,Q)$.

This leads to an analogue of 
Lemma~\ref{lem_difference_of_convex_polytope_and_open_convex_set} 
for $P$ replaced with a polytopal complex.    
\begin{proposition}\label{prop_retract_on_the_boundary_complex}
      Let $\Xi$ be a polytopal complex and $Q$ an open convex set. 
      Then $\supp \CC(\Xi,Q)$ is a strong deformation retract of $(\supp \Xi) \setminus Q$.
\end{proposition}

\begin{proof}
    We argue by induction on the number of elements (faces) of $\Xi$.
    If  $Q \cap \supp \Xi = \emptyset$, or equivalently, $\CC(\Xi,Q) = \Xi$, then there is nothing to prove.
    So suppose $P\in \Xi$ is such that $P \cap Q\ne \emptyset$ and assume that $P$ has maximal possible dimension among such faces.
    Then there is no other face $F\in \Xi$ that intersects the 
relative interior $\relP$.
    In particular, $\Xi' := \Xi \setminus \set{P}$ is a polytopal complex, such that $\supp \Xi' \cap P = \partial P$.
    By the inductive assumption,
       $\supp \CC(\Xi,Q) = \supp \CC(\Xi',Q)$ is a strong deformation retract of $(\supp \Xi') \setminus Q$.

    It remains to show, that  $(\supp \Xi') \setminus Q$ is a strong deformation retract of $(\supp \Xi) \setminus Q$.
But this follows directly by applying
Lemma~\ref{lem_difference_of_convex_polytope_and_open_convex_set}.
\end{proof}

\subsection{Compact approximation of open semialgebraic sets}
\label{compApprox}
Now we will discuss a way of replacing a semialgebraic set in $\RR^d$ 
with a homotopy equivalent subset that is additionally closed in $\RR^d$.

\begin{proposition}\label{prop_retract_on_a_closed_subset}
   Suppose $X\subset \RR^d$ is a compact semialgebraic subset of $\RR^d$.
   Let $\phi\colon \RR^d\to \RR$ be a continuous, piecewise polynomial function.
   Denote by $\phi_{>0} := \phi^{-1}((0, \infty))$ the set of points that are mapped to the positive axis, 
      and for $\epsilon\in \RR$  define $\phi_{\ge \epsilon} := \phi^{-1}([\epsilon, \infty))$.
   Then there exists a real number $c>0$ such that for all $0<\epsilon\le c$ 
      the intersection $X\cap \phi_{\ge \epsilon}$ is a strong deformation retract of $X \cap \phi_{>0}$.
\end{proposition}

\begin{proof}
We may and will assume that $X$ is contained in $\phi_{\ge 0}$.
We use the Whitney stratification of $X$, see for example \cite{thom_stratified_sets_and_morphisms} or \cite{kaloshin_Whitney_stratifications}.
We argue by restricting to one stratum of $X$ at a time.
When $c$ is sufficiently small, then the strata whose closures do not intersect 
$\phi_0:=\phi^{-1}(0)$ are contained in $X\cap \phi_{\ge \epsilon}$.
Hence the homotopy does not move these strata.
   The strata that are contained in $\phi_0$ are neither existent in 
     $X\cap \phi_{\ge \epsilon}$ nor $X \cap \phi_{>0}$.
   Hence it is enough to consider the strata whose closures intersect $\phi_0$, but are not contained in $\phi_0$.
   Let $M$ be such a stratum, and suppose that $M$ has a maximal dimension among all such strata.

Define $M_{<\epsilon} \subset M$ to be the intersection 
$M \cap \phi^{-1}(0, \epsilon)$.
Similar to the proof of 
Proposition~\ref{prop_retract_on_the_boundary_complex},
     we can find a strong deformation retract of 
$\overline{M} \cap \phi_{>0}$ onto 
     $(\pd M \cap \phi_{>0}) \cup (M \cap \phi_{\ge \epsilon}) = \overline{M}\setminus M_{< \epsilon}$.
   Then we replace $X$ with $X' = \overline{X \setminus (M_{<\epsilon} \cup \phi_0)}$, and we can argue inductively to show the claim.
\end{proof}

Suppose $Q \subset \RR^d$ is a (compact) polytope 
defined by affine inequalities
$\phi_i(v)\ge 0$ for $i \in \setfromto{1}{k}$.
Let $\epsilon>0$ be a positive real number.
Then the \emph{$\epsilon$-widening of $Q$} (with respect to the collection of inequalities $\set{\phi_i(v)\ge 0 \mid i \in \setfromto{1}{k}}$) is 
the set:
\[
  Q_{> -\epsilon} := \set{v\in \RR^d \mid  \forall_{i} \ \phi_i(v) > - \epsilon}.
\]
Note that $Q_{> -\epsilon}$ is open and contains $Q$. 
The shape of $Q_{> -\epsilon}$ may depend on the choice of the inequalities defining $Q$, but we will ignore this dependence in our notation, 
  as it will be irrelevant to our statements.

\begin{lemma}
\label{lem_difference_of_convex_polytope_and_epsilon_widening}
Suppose $P,Q\subset \RR^d$ are two polytopes.
Then there exists a positive constant $c>0$, such that for all 
$0< \epsilon \le c$, the difference $P \setminus Q_{>-\epsilon}$ is a 
strong deformation retract of $P \setminus Q$.
Similarly, if $\Xi$ is a polytopal complex,
then $\supp \Xi \setminus Q_{>-\epsilon}$ is a strong deformation 
retract of $\supp \Xi \setminus Q$ for sufficiently small $\epsilon$.
\end{lemma}
\begin{proof}
   Suppose $Q  = \set{\phi_i(v)\ge 0 \mid i \in \setfromto{1}{k}}$.
   If the intersection $Q \cap P$ is empty, then the statement is easy, 
just choose $c$ such that $P \cap Q_{>-c} = \emptyset$.
   So assume otherwise $Q \cap P \ne \emptyset$ and fix a point $v \in  Q \cap P$.
   For any $x \in P \setminus Q$ consider the unique line $\ell_x$ passing through $x$ and $v$.
   Let 
   \[
     c_x = -\frac{1}{\min \set{\phi_i(y) \mid i \in \setfromto{1}{k}, y \in
\ell_x\cap P}}.
   \]
   Note that $c_x>0$ and the set $\set{c_x\mid x \in  P \setminus Q}$ is closed,
      as its values are equal to those on $\supp\CC(P,Q)$, which is compact.
   So let $c=\min \set{c_x \mid x \in  P \setminus Q}$ and choose $0 <\epsilon \le c$.
   Then for every $x\in P \setminus Q$ the line $\ell_x$ has non-empty intersection
      with the compact set $P \setminus Q_{>-\epsilon}$.
   Define the retract as $x \mapsto r(x) =v + \lambda_x(x-v)$
      where $\lambda_x = \min \set{\lambda: \lambda \ge 0, v + \lambda_x(x-v) \in P \setminus Q_{>-\epsilon}}$.
   The standard homotopy $H(x,t) = t x + (1-t) r(x)$ gives the desired strong deformation.
   
   Note that in the above arguments, $r$ and $H$ preserve faces of $P$, in the sense, that if $F$ is a (closed) face of $P$, 
      and $r_{F}$ and $H_F$ are the retract and its deformation as above, but defined for $F$, 
      then $r_F = r_P|_F$ and $H_F = H_P|_{F\times [0,1]}$.  
   Thus, they glue well to define the appropriate retract and its strong deformation of $\supp \Xi \setminus Q_{>-\epsilon}$
      onto  $\supp \Xi \setminus Q$. 
\end{proof}

As a collorary we have an analogue of 
Proposition~\ref{prop_retract_on_the_boundary_complex} 
and Lemma~\ref{lem_difference_of_convex_polytope_and_open_convex_set}
for polytopes $Q$:
\begin{lemma}
\label{lem_difference_of_convex_polytope_and_closed_convex_set}
Let $\Xi\subset \RR^d$ be a polytopal complex, and let 
$Q\subset \RR^d$ be a polytope.
   Then $\supp \CC(\Xi,Q)$ is a strong deformation retract of $ \supp \Xi \setminus Q$.
   In particular, if $P \subset \RR^d$ is a polytope, then  $\supp \CC(P,Q)$ is a strong deformation retract of $P \setminus Q$.
\end{lemma}
\begin{proof}
   This is a combination of Lemma~\ref{lem_difference_of_convex_polytope_and_epsilon_widening} 
      and Proposition~\ref{prop_retract_on_the_boundary_complex},
      together with an observation that $\CC(\Xi,Q) = \CC(\Xi,Q_{>-\epsilon})$ for sufficiently small $\epsilon>0$.
\end{proof}

\begin{corollary}\label{cor_difference_of_closed_polytopes_and_boundary}
Let $P, Q\subset \RR^d$ be two polytopes and assume their intersection 
is nonempty. Then $\partial P \setminus Q$ is homotopy equivalent 
to $P \setminus Q$.
\end{corollary}
\begin{proof}
   The complex of $P$ consists of all faces of $\partial P$ and in addition $P$. 
   Since $Q \cap P \ne \emptyset$, the complexes $\CC(P, Q)$ and $\CC(\partial P, Q)$ are equal.
   Therefore, by Lemma~\ref{lem_difference_of_convex_polytope_and_closed_convex_set}
     both  $P \setminus Q$ and $\partial P \setminus Q$ are homotopy equivalent to $\supp \CC(P, Q)= \supp\CC(\partial P, Q)$.
\end{proof}

\subsection{Allowing common tail cones}
\label{Allowing common tail cones}
Finally, we conclude this section with an argument that reduces 
considerations of homotopy types 
   of differences of (closed) polyhedra to the case of (compact) polytopes.
A simplifying assumption is that the polyhedra have the same tail cone.
Recall that $\tail(P):=\{v\in\R^d\kst P+v\subseteq P\}$
is the polyhedral cone indicating the unbounded directions of 
a polyhedron $P$.

\begin{proposition}\label{prop_truncate_infinite_part}
Suppose $P\subset \RR^d$ is a polyhedron with a pointed tail cone 
and $Q\subset \RR^d$ is a polyhedron or the interior of a polyhedron
with the same tail cone $\tail Q=\tail P$.
Then there exists a sequence of linear forms $\fromto{H_1}{H_k}$ and  
sufficiently large numbers $\fromto{t_1}{t_k}\in \RR$,
such that the truncated difference
      \[
          \Trunc(P\setminus Q):=(P\setminus Q) \cap \bigcap_{i} \set{H_i \le t_i}
      \]
is compact and a strong deformation retract of $P\setminus Q$.
\end{proposition}

\begin{proof}
Proceeding inductively on the number of rays of the common tail cone,
we may assume that 
there are polyhedra $P'$ and $Q'$ with $\tail P'=\tail Q'$ not containing
a certain ray $\rho$ such that $P = P'+\rho$ and $Q = Q'+\rho$.
We choose a linear form $H$ with $H(\rho)>0$ that is non-positive on 
$\tail P'=\tail Q'$.
Then there exists a real number $t$ such that both $P'$ and $Q'$ 
are contained in the halfplane $\set{ H< t}$.
It follows that $(P\setminus Q) \cap \set{H\le t}$ 
is a strong deformation retract of $P\setminus Q$.
Indeed, the map $r_\rho\colon \RR^d \to \set{H\le t}$ projecting 
along $\rho$ 
does the job.
\end{proof}

As a conclusion, we remark that the homotopy equivalences, such as that in  
Lemma~\ref{lem_difference_of_convex_polytope_and_closed_convex_set}, are valid also for polyhedra with common tail cones.

\section{Toric geometry}
The main subject of our paper is to 
investigate a toric variety $X$ and its immaculacy locus within $\Cl(X)$. 
For this we will make use of the classical method
of calculating the cohomology of equivariant line bundles from the fan
in $N_\R$.
However, after introducing
the usual toric notation in Subsection~\ref{sect_toric_notation},
we will  provide an alternative method
using the momentum polyhedra in $M_\R$ in Subsection~\ref{cohomPolyhedra}.
It is appropriate to make the cohomology of equivariant line bundles
or its absence visible.

\subsection{Basic toric notation}
\label{sect_toric_notation}
All our toric varieties are normal.
Our main references for dealing with toric varieties
are \cite{CoxBook, fultonToric, KKMS}.
We denote by $N$ the lattice of one-parameter subgroups of the torus acting on the toric variety, 
  and by $M$ the monomial lattice.
Throughout $\Sigma$ denotes a fan in $N$ and $X=\toric(\Sigma)$ the corresponding toric variety.
Occasionally, if there is more than one toric variety involved, we may add a subscript $N_X$, $M_X$, $\Sigma_X$,\dots.
For a cone $\sigma$ in $N_{\RR} = N \otimes \RR$ or $M_{\RR} = M \otimes \RR$
   we denote the dual cone in $M_{\RR}$ or $N_{\RR}$, respectively,  by $\sigma^{\vee}$ .

The set of all cones of dimension $k$ of a fan $\Sigma$ is denoted $\Sigma(k)$.
Similarly, for a cone $\sigma$, by  $\sigma(k)$ we mean the set of all faces of dimension $k$.
In general, every cone $\sigma$ generates a unique fan consisting of all  
faces of $\sigma$, and the fan will be denoted by the same letter $\sigma$.
In order to reduce the notation, we will follow the standard convention to denote rays (one dimensional strictly convex lattice cones)
  and their primitive lattice generators by the same letter, usually $\rho$.

We will frequently assume that our toric variety $X$ is \emph{semiprojective},
that is that it is projective over an affine (toric) variety.
This means that the fan of $X$ has a convex support $\suppSig\subseteq N_{\RR}$.
Another assumption simplifying the notation in the proofs is that $X$ has no torus factors.
In particular, (with both these assumptions) the fan $\Sigma$ is generated by cones of dimension equal to $\dim X$.

Every Weil divisor on $X$ is linearly equivalent to a 
torus invariant divisor 
$D=\sum_{\rho\in\Sigma(1)} \lambda_{\rho}\cdot D_{\rho}$ with $D_{\rho}:=\ko{\orb(\rho)}$.
If in addition $D$ is $\QQ$-Cartier, 
then there exists a continuous function $u \colon \suppSig\to \RR$, 
which is linear on the cones of $\Sigma$, 
   and such that 
   $u(\rho) = -\lambda_{\rho}$ for every $\rho\in \Sigma(1)$.
In particular, for every maximal cone $\sigma \in \Sigma$
there is a unique $u_{\sigma} \in M_{\QQ}$, such that 
$u|_{\sigma} = \langle \cdot,  u_{\sigma} \rangle$.
We call $u$ the \emph{support funtion} of $D$. 
The divisor $D$ is Cartier if and only if each $u_\sigma$ is contained in 
the lattice $M$.

The polyhedron of sections $\sDelta=\sDelta(D) \subset M_{\RR}$ of an 
equivariant Weil divisor $D$ is defined by its inequalities:
\[
   \sDelta = \set{ r \in M_{\RR} \mid \langle \rho , r  \rangle \ge - \lambda_\rho \text{ for all } \rho \in \Sigma(1)}.
\]
The name was derived from the fact that $\sDelta\cap M$ provides a 
(monomial) basis of the global sections of $\CO_X(D)$.
If $D$ is in addition $\QQ$-Cartier, then we can describe it also as an intersection of shifted cones that depend on the support function $u$:
\[
  \sDelta = \bigcap_{\sigma \text{ maximal cone of }\Sigma}  
\left(u_{\sigma} + \sigma^{\vee}\right).
\]

A \WQC divisor $D$ on a semiprojective toric variety $X$ of dimension $d$ is nef if and only if its support function $u$ is concave, 
   that is, for all $a,b \in \suppSig$ and for all $0 \le t\le 1$, 
   we have $u(ta +(1-t)b) \ge t u(a) + (1-t) u(b)$.
Equivalently, if $a\in \sigma$ for some $\sigma \in \Sigma(d)$, then for every $\sigma' \in \Sigma(d)$ we have:
$
   \langle a, u_{\sigma}\rangle \le  \langle a, u_{\sigma'}\rangle.
$
Another way to understand nefness is that all $u_\sigma$ are contained in
$\sDelta$; in fact, the set of its vertices 
equals the set $\{u_{\sigma}\}$.
Note that some of the $u_\sigma$ might coincide.
Moreover, in contrast to the projective case treated, e.g., in
\cite[(4.2)]{CoxBook}, for semiprojective $X$, the polyhedron $\sDelta$ is no longer compact
but has $(\suppSig)\dual\subseteq M_{\RR}$ as its tail cone.
Nevertheless, one may still recover the support function of a nef divisor 
from its polyhedron $\sDelta$ by 
$$
u(a)=\min\langle a,\sDelta\rangle:=\min\{\langle a,r\rangle\kst r\in \sDelta\}.
$$
Note that the minimum is well-defined for $a\in\suppSig=(\tail\sDelta)\dual$.

A fan $\Sigma$ in $N_\R\cong\R^d$ gives rise to a map 
$\rho:\Z^{\Sigma(1)}\to N$, which takes the basis element
indexed by a ray of $\Sigma$ to the corresponding primitive element on that ray in $N$.
If the underlying toric variety $X=\toric(\Sigma)$ has no torus factors, 
then the cokernel of $\rho$ is finite.
For simplicity, we always assume that this is the case.
If, moreover, $X$ is smooth, then $\rho$ is surjective.
We denote the kernel by $K$, and  we obtain a short exact sequence
$$
\xymatrix@R=1ex@C=2em{
0 \ar[r] &
K \ar[r]^{} &
\Z^{\Sigma(1)} \ar[r]^{\rho} &
N.
}
$$
It is well known that the 
so-called Gale 
dual of this sequence yields
\begin{equation}\label{equ_exact_sequence_with_Cl_Div_T_and_M}
\xymatrix@R=1ex@C=2em{
0 &
\Cl(X) \ar[l] &
\Div_T(X) \ar[l]_-{\pi} &
M \ar[l]_-{\rho^*} &
0, \ar[l]
}
\end{equation}
where $\Div_T(X)=\left(\Z^{\Sigma(1)}\right)^*$ denotes the group of torus
invariant Weil divisors on $X$.
Note that $\Cl(X)$ may have torsion, which corresponds to the torsion of the
cokernel of $\rho$.
The anticanonical class of $X$ is $-K_X=\pi(\ku{1})$.
The set of effective classes is 
$\Eff_{\Z}(X)=\pi\left(\Z^{\Sigma(1)}_{\geq 0}\right)$, 
although often we really consider the effective cone 
$\Eff_{\RR}(X)=\pi\left(\R^{\Sigma(1)}_{\geq 0}\right)$, 
where $\pi$ is now considered as the map 
$\RR^{\Sigma(1)} \to \Cl(X)\otimes \RR$.

\begin{example}\label{ex_hexagon_introduce_coordinates}
Throughout the the text we will regularly come back to the example of del~Pezzo surface of degree $6$,
   which is the blow up of $\PP^2$ in three points, 
   also refered to as a \emph{hexagon} due to the shapes of its fan and the polytopes of sections of ample divisors.
This is also a smooth projective toric variety of Picard rank $4$,
   which illustrates that our methods go beyond the main results presented in this article (splitting fans and Picard rank~$3$ cases).
The exact sequence \eqref{equ_exact_sequence_with_Cl_Div_T_and_M} in this example is given by the matrices
\[
\rho^*
=\begin{psmallmatrix}
1 & 0\\
0 & 1\\
-1 & 1\\
-1 & 0\\
0 & -1\\
1 & -1
\end{psmallmatrix}
\mbox{ and }
\pi=\begin{psmallmatrix}
1 & 0 & 0 & 1 & 0 & 0\\
0 & 1 & 0 & 0 & 1 & 0\\
0 & 0 & 1 & 0 & 0 & 1\\
1 & -1 & 1 & 0 & 0 & 0
\end{psmallmatrix}.
\]
The rows of $\rho^*$ form the rays of our fan $\Sigma$, meaning we work with
the following two-dimensional fan:
\[
\begin{tikzpicture}
\draw[step=1, black!20, thin] (-1.5,-1.5) grid (1.5, 1.5);
\draw[thick,->] (0,0) -- (1,0) node[anchor=west]{$0$};
\draw[thick,->] (0,0) -- (0,1) node[anchor=south]{$1$};
\draw[thick,->] (0,0) -- (-1,1) node[anchor=south east]{$2$};
\draw[thick,->] (0,0) -- (-1,0) node[anchor=east]{$3$};
\draw[thick,->] (0,0) -- (0,-1) node[anchor=north]{$4$};
\draw[thick,->] (0,0) -- (1,-1) node[anchor=north west]{$5$};
\end{tikzpicture}
\]
With this choice of $\rho^*$ and $\pi$ the $\Nef$ cone is generated by the following
$5$ rays, where we write its polytope of sections $\Delta$ next to it:
\[
\begin{array}{cccccc|c}
\toprule
\multicolumn{6}{c}{\Pic(X) \text{ coordinates}} & \Delta \subset M_{\R}\\
\midrule
&1 & 1 & 0 & 0 && \hexlinediag
\\
&1 & 0 & 1 & 1 && \hexlinehor
\\
&0 & 1 & 1 & 0 && \hexlinevert
\\
&1 & 1 & 1 & 0 && \hextriangleupleft
\\
&1 & 1 & 1 & 1 && \hextriangledownright
\\
\bottomrule
\end{array}
\]
\end{example}
   
\subsection{Toric cohomology}
\label{toricCohom}

Let us review the classical method of calculating the cohomology
groups of toric divisors.
Afterwards, in Subsection~\ref{cohomPolyhedra},
we dualize it to obtain another method that exploits the polyhedra
of sections of nef divisors.

If $D=\sum_{\rho\in\Sigma(1)} \lambda_{\rho}\cdot D_{\rho}$ is a Weil divisor on a toric variety $X=\toric(\Sigma)$, 
then for every $\dm\in M$  we define 
\begin{equation}\label{equ_complex_for_cohomologies_Weil}
   V_{D,\dm} := \bigcup_{\sigma\in\Sigma}\conv\{\rho\kst \rho \in\sigma(1),
\langle \rho,\dm\rangle< -\lambda_{\rho}\}\subseteq N_\R.
\end{equation}

It is a classical result \cite[Thm~9.1.3]{CoxBook} 
   (see also \cite[Thm~2.2 and Rem.~2.3]{achinger_toric_var_char_p} for the characteristic free proof),
   that one obtains the $\dm$-th homogeneous piece of the sheaf cohomology of
$\CO_X(D)$ as
$$
\gH^i\big(\toric(\Sigma), \CO_X(D)\big)_\dm = 
\tH^{\,i-1}(V_{D,\dm},\kk) \ 
\text{ for all } \ i\geq 0.
$$
Recall that the $(-1)$-st reduced cohomology of a set $S$ is defined as
$$
\tH^{-1}(S,\kk)=\left\{
\begin{array}{ll} \kk & \mbox{if } S=\emptyset\\
                   0 & \mbox{if } S\ne\emptyset.
\end{array}\right.
$$
Note that, since $0\notin V_{D,\dm}$, one might retract these sets onto the sphere $S^{d-1}\subseteq N_\R$ 
   (where $d$ is the dimension of $X$, and hence also of $N_{\RR}$) 
without changing their cohomology. 
Alternatively, we can replace $V_{D,\dm}$ with 
$V_{D,\dm}^{>}:=\R_{>0}\cdot V_{D,\dm}$. 
If $\Sigma$ is simplicial, then the latter sets are 
(up to $0\notin V_{D,\dm}^{>}$)
the support of ``full'' or ``induced'' subcomplexes 
$V_{D,\dm}^{\geq}$ of $\Sigma$,
that is,  for every $\sigma\in\Sigma$, the intersection
$V_{D,\dm}^{\geq}\cap\sigma$ consists of a single closed face
of $\sigma$.

If $D=\sum_{\rho\in\Sigma(1)} \lambda_{\rho}\cdot D_{\rho}$ is 
at least a $\QQ$-Cartier divisor on $X$,
then one can alternatively use its support function $u$
to calculate the cohomology of $\cO_X(D)$. The subset
\begin{equation}\label{equ_complex_for_cohomologies_Cartier}
  V_{D,\dm}^{\supp} = \set{a\in \suppSig \mid \langle a, \dm \rangle < u(a)}
\subseteq \suppSig 
\end{equation}
contains $V_{D,\dm}^{>}$ as a strong deformation retract. One can easily prove
this using homotopies as in
Subsection~\ref{sect_homotopy_of_polytopal_complexes}.
See \cite[Theorem 9.1.3]{CoxBook} for a slightly weaker claim.
Actually, in the original \cite[p.42]{KKMS}, it was exactly the sets
$V_{D,\dm}^{\supp}$ which were used to describe 
$\gH^i\big(\toric(\Sigma), \CO_X(D)\big)_\dm$.

\subsection{Cohomology using polyhedra}
\label{cohomPolyhedra}
From now on we assume $X$ to be a semiprojective toric variety, 
in particular it is quasiprojective. 
Let $Y$ be a projective toric variety containing $X$ as an open torus invariant subset.
Fix a torus invariant ample Cartier divisor $L$ on $Y$ such that $L + K_Y$ is effective, 
   where $K_Y  = -\sum_{\rho\in\Sigma_{Y}(1)} D_{\rho}$ is the canonical divisor of $Y$.
Then the piecewise linear function $\| \cdot \|:=-u$ 
corresponding to $L$ is a norm on the vector space $N_{\RR}$. 
The closed balls centred at $0$ with respect to this norm are convex polytopes, whose vertices are on rays of $\Sigma_Y$. 

Since $X$ is quasiprojective, every \WQC divisor is a difference of 
nef divisors: $D = D^+ - D^-$, with both $D^+$ and $D^-$ nef \WQC
\cite[Thm~6.3.22(a)]{CoxBook}.
Thus every such Cartier divisor on $X=\toric(\Sigma)$ is (non-uniquely) 
represented by a pair of polyhedra $(\Delta^+, \Delta^-)$
sharing the same tail cone $|\Sigma|\dual\subseteq M_\R$.
Polyhedra form a semigroup under Minkowski addition.
Restricting to polyhedra with a fixed tail cone, one ensures that
this semigroup is cancellative. In this context,
the pair $(\Delta^+,\Delta^-)$ represents the formal difference
$D=\Delta^+-\Delta^-$ within the
Grothendieck group of generalized polyhedra.
 
The goal of this section is to reinterprete the toric cohomology in terms of this pair of polyhedra.

\renewcommand{\theenumi}{(\roman{enumi})}
\renewcommand{\labelenumi}{\theenumi}

\begin{lemma}\label{lem_nef_difference_and_cohomology_complex}
Let $X$ be a semiprojective toric variety with no torus factors and $D= D^+ - D^-$ 
be a \WQC divisor on $X$ with $D^+$ and $D^-$ nef. 
Assume that $\Delta^+$ and $\Delta^-$ are the associated polyhedra of 
$D^+$ and $D^-$ and denote by $u$ the support function of $D$. Then the sets 
$V_{D,0}^{\supp} = \set{a\in \suppSig \mid u(a)>0}
\subseteq \suppSig $
and $\Delta^- \setminus \Delta^+$ are homotopy equivalent.
\end{lemma}

\begin{proof}
Let $u^\pm$ be the support functions of the nef divisors $D^\pm$.
For each full-dimensional $\sigma\in \Sigma$
we denote by $\rps\in\Delta^+$ and
$\rms\in\Delta^-$ the unique vertices minimising
$\langle a,\kbb\rangle$ on the respective polytopes
for $a\in\innt\sigma$, hence for all $a\in\sigma$.
Thus, for $a\in\sigma$, we have
$\min \langle a,\Delta^\pm\rangle = \langle a, \rpms\rangle=u^\pm(a)$.
Moreover, we can write
$$
V_{D,0}^{\supp}=\{a\in \suppSig \kst 
            u^-(a)         < u^+(a)\}.
$$
Since $V_{D,0}^{\supp}\subseteq N_\R$ and 
$\Delta^- \setminus \Delta^+\subseteq M_\R$ are contained in mutually dual 
spaces, we are going to compare these two sets via the following incidence set:
$$
W:= \{(a,r)\in V_{D,0}^{\supp}\times (\Delta^-\setminus\Delta^+)\kSt
\,\langle a,r\rangle < u^+(a)\}.
$$
It comes with two natural, surjective projections
$$
\xymatrix@R=2ex@C=2em{
& W \ar@{->>}[ld]_-{p_V} \ar@{->>}[rd]^-{p_\Delta} \\
N_\R\supseteq V_{D,0}^{\supp} && 
\hspace{2em}\Delta^-\setminus\Delta^+ \subseteq M_\R,
}
$$
with contractible fibers: Let us start with checking the map $p_V$.
If $a\in V_{D,0}^{\supp}$, 
then there is a cone $\sigma\in\Sigma$ containing $a$, and we obtain that
$$
p_V^{-1}(a) \cong
\{r\in \Delta^-\setminus\Delta^+\kst
\langle a,r\rangle < \langle a,\rps\rangle\}
=
\{r\in \Delta^-\kst
\langle a,r\rangle < \langle a,\rps\rangle\}.
$$
Obviously, the latter is a convex set. However, it is non-empty, too.
The reason is that the fact $a\in V_{D,0}^{\supp}$ (together with $a\in\sigma$)
implies that
$
\,\min \langle a,\Delta^-\rangle < \langle a,\rps\rangle
$.
We turn to the second map $p_\Delta$. Fixing an element
$r\in\Delta^-\setminus\Delta^+\subseteq\Delta^-$ we have
$$
p_\Delta^{-1}(r)\cong
\{a\in V_{D,0}^{\supp}\kst\langle a,r\rangle < 
u^+(a)\}
= 
\{a\in N_\R\kst \langle a,r\rangle < 
\min \langle a,\Delta^+\rangle\}.
$$
Again, the latter is a convex set and, because  $r\notin\Delta^+$,
it is non-empty, too.

Now, the idea is to apply results around
the Vietoris mapping theorem. We are going to use
the stronger version from \cite{smale_Vietoris_mapping_theorem}.
Taking into account Whitehead's theorem, the criterion of Smale says
   that a proper surjective continuous map $f:X\to Y$ between
two CW complexes $X \subset \R^n$ and $Y \subset \R^m$ is a homotopy
equivalence if its fibers $f^{-1}(y)$ (for all $y\in Y$) are contractible and locally
contractible. 

Since our maps $p_V$ and $p_\Delta$ are not proper yet, we will replace
the three objects in the above diagram with homotopy equivalent gadgets which are all compact.
Recall the notions of sufficiently large truncation $\Trunc(\Delta^-)$ of $\Delta^-$ as in Proposition~\ref{prop_truncate_infinite_part} 
   and the $\epsilon$-widening $(\Delta^+)_{> -\epsilon}$ as in Subsection~\ref{compApprox}.
For any $R>0$ and any sufficiently small $\epsilon>0$ we consider the following three compact sets: 
\begin{align*}
    V_{D,0}^{\supp}({R, \epsilon})&:=\set{a \in \suppSig\mid u(a)\ge \epsilon
\text{ and } \|a\| \le R}, \\
    W({R, \epsilon})&:=\set{(a, r) \in \suppSig\times \Trunc(\Delta^-) \mid
\langle a , r  \rangle \le   u^+(a) - \epsilon,\,
             u(a)\ge \epsilon, \text{ and } \|a\| \le R}, \text{ and}\\
    \Trunc(\Delta^-) &\setminus (\Delta^+)_{> -\epsilon}.           
 \end{align*}
By the results from Section~\ref{sect_homotopy_of_polytopal_complexes} these are homotopy equivalent to $V_{D,0}$, $W$, and $\Delta^- \setminus \Delta^+$, respectively 
   (see Lemma~\ref{lem_difference_of_convex_polytope_and_epsilon_widening}, Propositions~\ref{prop_retract_on_a_closed_subset} and~\ref{prop_truncate_infinite_part}).
We need to carefully choose the inequalities used in the $\epsilon$-widening
so that the projection $ W({R, \epsilon}) \to \Trunc(\Delta^-) \setminus
(\Delta^+)_{> -\epsilon}$ is well defined and surjective.
Then with the same arguments as above we show that the fibres of projections $W({R, \epsilon}) \to V_{D,0}^{\supp}({R, \epsilon})$ and $W({R, \epsilon}) \to\Trunc(\Delta^-) \setminus (\Delta^+)_{> -\epsilon}$
   are non-empty convex polytopes. 
Thus by the criterion of Smale, the projection maps are homotopy equivalences, and consequently, $V_{D,0}$ is homotopy equivalent to $\Delta^- \setminus \Delta^+$.
\end{proof}

\begin{theorem}\label{thm_toric_cohomology_in_terms_of_nef_polytopes}
   Let $X$ be a semiprojective toric variety and $D= D^+ - D^-$ 
be a \WQC divisor on $X$ with $D^+$ and $D^-$ nef.
Denote by $\Delta^+$ and $\Delta^-$ the polyhedra of $D^+$ and $D^-$,
respectively.
   Then $\gH^i(X,\cO(D)) = 
\bigoplus_{\dm \in M} \tH^{i-1}(\Delta^- \setminus(\Delta^+ -\dm), \kk)$.
\end{theorem}

\begin{proof}
We will show that
$\,\gH^i(X,\cO(D))_\dm=\tH^{i-1}(\Delta^- \setminus(\Delta^+ -\dm), \kk)$.
From Subsection~\ref{toricCohom} together with
Lemma~\ref{lem_nef_difference_and_cohomology_complex} we obtain this claim for
$\dm=0$.

For general $\dm\in M$ we define $D(\dm):=D+\div(x^\dm)
= D+\sum_{a\in\Sigma(1)}\langle a,\dm\rangle \cdot D_a$.
Compared with $D$, its associated sheaf is twisted with
$\CO_X(\dm):=\CO_X\big(\div(x^\dm)\big)=x^{-\dm}\cdot\CO_X$.
Since the polyhedra $\Delta^{\pm}$ encode, for each affine chart,
the minimal generators of the sheaves $\CO_X(D^\pm)$, this means that
the divisor $D(\dm)$ is represented by the pair $(\Delta^+-\dm,\Delta^-)$ or,
equivalently, by $(\Delta^+,\Delta^-+\dm)$.
In particular, for the support functions we have
$u^{D(\dm)}=u^D-\dm$. Thus, $V_{D(\dm),0}=V_{D,\dm}$.
\end{proof}

\begin{remark}
   Note that the presentation of a toric divisor $D= D^+ - D^-$ as a difference of nef divisors is by far not unique. 
   Thus, one of the consequences of Theorem~\ref{thm_toric_cohomology_in_terms_of_nef_polytopes} is that the 
     reduced cohomology of the difference of polyhedra is independent of the choice of this presentation.
   In particular, choosing a suitable semiprojective toric variety $X$, 
      for any three rational polyhedra $\Delta_0, \Delta_1, \Delta_2$ in $M_{\RR}$ with the same tail cone,
      the differences $\Delta_1\setminus \Delta_2$ and $(\Delta_1+ \Delta_0)\setminus (\Delta_2+\Delta_0)$ are homotopy equivalent.
\end{remark}

\begin{remark}
   Suppose $X$ is a projective toric variety and $D$ is a \WQC divisor on $X$.
   Observe that despite that there are at most finitely many degrees $\dm$ 
for which $\gH^i(\cO(D))_{\dm}\ne 0$,
     in the definitions of $V_{D,\dm}$ and $V_{D,\dm}^{\supp}$ it is not 
immediately clear,
     which $\dm\in M$ can potentially lead to nonzero cohomology.
   Instead, 
     the description in Theorem~\ref{thm_toric_cohomology_in_terms_of_nef_polytopes} provides such a criterion.
   If $\Delta^-$ and  $\Delta^+ + \dm$ are disjoint, then the difference is contractible. 
   We will elaborate more on this criterion in a follow up article about related computational issues.
\end{remark}

\begin{remark}
\label{rem-recoverPolSect}
Let the \WQC divisor $D$ be encoded by the pair of polyhedra
$(\Delta^+,\Delta^-)$. Then, the polyhedron of sections
$\sDelta(D)$ mentioned in Subsection~\ref{sect_toric_notation}
can be recovered as
$$
\sDelta(D)=\bigcap_{r\in\Delta^-}(\Delta^+-r)=\{r\in M_\R\kst
\Delta^-+r\subseteq\Delta^+\}.
$$
\end{remark}

\begin{example}
If $D$ is a nef divisor on a projective toric variety $X$,
 one can choose $\Delta^-=\{0\}$,
and then the formula from Remark~\ref{rem-recoverPolSect} implies
$\sDelta=\sDelta(D)=\Delta^+$. 
   Thus for $r\in \sDelta$ the set 
$\Delta^- \setminus (\Delta^+ - r)$ is empty (thus only has $1$-dimensional $(-1)$st cohomology), 
      or for $r\notin \sDelta$ the set $\Delta^- \setminus (\Delta^+ - r)$ is a single point,
      hence it has no reduced cohomology at all.
   Therefore, $h^0(X,\cO(D)) = \#(\Delta \cap M)$ and $\gH^i(X,\cO(D))=0$ for all $i>0$.  
\end{example}

\begin{example}
   If on the other hand $-D$ is a nef divisor, then $\Delta^+=\{0\}$,
and $\Delta^-$ is the polytope of $-D$.
   Let $i=\dim \Delta^-$.
   Thus for $r\in -\relint \Delta^-$ the set 
$\Delta^- \setminus (\Delta^+ - r)$ is homotopic to a sphere of dimension $i-1$,
     while for $r\notin -\relint\Delta^-$ the set 
$\Delta^- \setminus (\Delta^+ - r)$ is contractible.
   Therefore, $h^{i}(X,\cO(D)) = \#(\relint \Delta^{-} \cap M)$ and $\gH^j(X,\cO(D))=0$ for all $j\ne i$.  
\end{example}

\begin{figure}[hbt]
   \begin{center}
   \begin{tikzpicture}
    \fill[color=bladyfiolet!60] (0,2.3) -- (-2.3,2.3) -- (-2.3,-2.3) -- (0,0) -- cycle;
    \fill[color=jasnyfiolet!60] (2.3,0) -- (2.3,-2.3) -- (-2.3,-2.3) -- (0,0) -- cycle;
    \draw[color=oliwkowy!50] (-2.3,-2.3) grid (2.3,2.3);
    \draw[thick,  color=mocnyfiolet,->] (0,0) -- (2.3,0) node[right, color=ciemnyfiolet] {$-1$};
    \draw[thick,  color=mocnyfiolet,->] (0,0) -- (0,2.3) node[anchor=south, color=ciemnyfiolet] {$-1$};
    \draw[thick,  color=mocnyfiolet,->] (0,0) -- (-2.3,-2.3) node[left, color=ciemnyfiolet] {$1$};
    \fill[thick, ciemnazielen] (0,1) node[anchor=south east]{$V_{D,0}$} node[anchor=south west, color=mocnyfiolet]{$\rho_2$} circle (3pt);
    \fill[thick, ciemnazielen] (1,0) node[anchor=south west, color=mocnyfiolet]{$\rho_1$} circle (3pt);
    \fill[mocnyfiolet] (-1,-1) node[anchor=south east]{$\rho_3$};
   \end{tikzpicture}
   \end{center}
   \caption{The fan of $X=\PP^2\setminus {[0,0,1]}$, the divisor $D = - D_{\rho_1} - D_{\rho_2}+D_{\rho_3} \simeq \cO_X(-1)$, 
            and the complex $V_{D,0}$ consisting of $2$ points responsible for $\gH^1(\cO_X(D))_0 \ne 0$.}
   \label{fig_fan_of_P2_minus_point}
\end{figure}
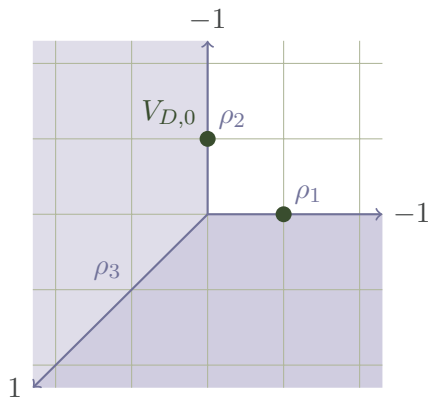

    The claim of Theorem~\ref{thm_toric_cohomology_in_terms_of_nef_polytopes} does not hold for quasiprojective toric varieties,
      that are not semiprojective. 
    To see this, consider $X=\PP^2\setminus\set{[0,0,1]}$ and let $D \simeq \cO_{X}(-1)$ be the negative of the hyperplane divisor.
    Then $D^+ \simeq 0$ and $D^- \simeq \cO_X(1)$, the polytopes are a point
and a basic triangle, respectively.
Thus $\Delta^-\setminus (\Delta^+-\dm)$ is always non-empty and contractible,
      hence the difference never has any reduced cohomologies.
But $\gH^1(\cO_X(-1))\ne 0$ as shown on Figure~\ref{fig_fan_of_P2_minus_point}.

\begin{example}
\label{ex_hexagon-someCohom}
    In the notation and coordinates of the ``hexagon'' example (Example~\ref{ex_hexagon_introduce_coordinates}), consider the divisor $D = (-4,-4,-2,1) \in \Pic(X)$.
   We have $D = D^+ - D^-$ for instance as $D^+ = \hexlinehor + \hexlinevert$ and  $D^- = \hexlinediag + 4 \hextriangleupleft$, where we represent the nef divisors with their polytopes.
    Thus we obtain
    \[
       \Delta^- = 
\vcenter{\hbox{\begin{tikzpicture}[scale=0.5]
\draw[step=1, color=oliwkowy!20, thin] (-0.5,-0.5) grid (5.5, 5.5);
\fill[pattern color=mocnyfiolet!80, pattern=dots] (0,0) -- (0,4)--(1,5)--(5,5)--cycle;
\draw[color=mocnyfiolet, thick] (0,0) -- (0,4)--(1,5)--(5,5)--cycle;
\end{tikzpicture}}}
\text{ and } 
       \Delta^+ = 
\vcenter{\hbox{\begin{tikzpicture}[scale=0.5]
\draw[step=1, color=oliwkowy!20, thin] (-0.5,-0.5) grid (1.5, 1.5);
\fill[pattern color=ciemnazielen!80, pattern=north east lines] (0,0) -- (0,1)--(1,1)--(1,0)--cycle;
\draw[color=ciemnazielen, thick] (0,0) -- (0,1)--(1,1)--(1,0)--cycle;
\end{tikzpicture}}}.
\]
The only non-contractible differences of $\Delta^-$ and a shift $(\Delta^+ +m)$ are:
\[
\vcenter{\hbox{\begin{tikzpicture}[scale=0.5]
\draw[step=1, color=oliwkowy!20, thin] (-0.5,-0.5) grid (5.5, 5.5);
\fill[pattern color=mocnyfiolet!80, pattern=dots] (0,0) -- (0,4)--(1,5)--(5,5)--cycle;
\draw[color=mocnyfiolet, thick] (0,0) -- (0,4)--(1,5)--(5,5)--cycle;
\fill[white] (0,1) -- (0,2)--(1,2)--(1,1)--cycle;
\draw[color=white, thick] (0,1) -- (0,2)--(1,2)--(1,1)--cycle;
\draw[color=mocnyfiolet!50, densely dotted]  (0,2)--(1,2)--(1,1)--(0,1);
\fill[white] (1,1) circle (3pt);
\end{tikzpicture}}},
\vcenter{\hbox{\begin{tikzpicture}[scale=0.5]
\draw[step=1, color=oliwkowy!20, thin] (-0.5,-0.5) grid (5.5, 5.5);
\fill[pattern color=mocnyfiolet!80, pattern=dots] (0,0) -- (0,4)--(1,5)--(5,5)--cycle;
\draw[color=mocnyfiolet, thick] (0,0) -- (0,4)--(1,5)--(5,5)--cycle;
\fill[white] (1,3) -- (1,4)--(2,4)--(2,3)--cycle;
\draw[color=white, thick] (1,3) -- (1,4)--(2,4)--(2,3)--cycle;
\draw[color=mocnyfiolet!50, densely dotted]  (1,3) -- (1,4)--(2,4)--(2,3)--cycle;
\end{tikzpicture}}},
 \text{ and }
\vcenter{\hbox{\begin{tikzpicture}[scale=0.5]
\draw[step=1, color=oliwkowy!20, thin] (-0.5,-0.5) grid (5.5, 5.5);
\fill[pattern color=mocnyfiolet!80, pattern=dots] (0,0) -- (0,4)--(1,5)--(5,5)--cycle;
\draw[color=mocnyfiolet, thick] (0,0) -- (0,4)--(1,5)--(5,5)--cycle;
\fill[white] (3,5) -- (3,4)--(4,4)--(4,5)--cycle;
\draw[color=white, thick] (3,5) -- (3,4)--(4,4)--(4,5)--cycle;
\draw[color=mocnyfiolet!50, densely dotted]  (3,5) -- (3,4)--(4,4)--(4,5);
\fill[white] (4,4) circle (3pt);
\end{tikzpicture}}}.
\]
Therefore $h^0(D) = 0$, $h^1(D) = 2$ and $h^2(D) = 1$.
\end{example}

\section{The immaculacy locus in \texorpdfstring{$\Pic(X)$}{Pic(X)}}
\label{ch-Immaculacy}

In this section we introduce the notion of an immaculate sheaf, 
concentrating on the case of line bundles.
We also study a relative version of this notion, and how 
the immaculacy interacts with morphisms.

\subsection{Immaculate line bundles}
\label{imaculateCheck}

Recall, that a sheaf is called \emph{acyclic}, if it has all higher cohomology groups equal to zero. 
We will also say that for a field $\kk$, a topological space $V$ is 
$\kk$-acyclic, if it is non-empty, arcwise connected, and its singular 
cohomologies $\gH^i(V, \kk) =0$ vanish for all $i>0$.
Note that in such case $\gH^0(V, \kk) =\kk$.
For example, all non-empty contractible spaces are $\kk$-acyclic (for any $\kk$), and the real projective plane $\RR\PP^2$ is  
$\kk$-acyclic (except if $\charact \kk =2$). Spheres $S^k$ are never $\kk$-acyclic.

\begin{definition}
\label{def-immaculate}
We call a sheaf $\cF$ on a variety $X$ {\em immaculate} if all
cohomology groups $\gH^p(X,\CF)$ ($p\in\Z$) vanish. The difference from the 
usual notion of acyclic sheaves is that we ask for the vanishing of $\gH^0$,
too.
\end{definition}

In particular, a toric sheaf of a \WQC divisor $\cO_X(D)$ is immaculate if and only if all
sets $V_{D,\dm}$ are $\kk$-acyclic.
Equivalently, as an immediate consequence of Theorem~\ref{thm_toric_cohomology_in_terms_of_nef_polytopes},
we can identify the immaculate line bundles in terms of properties 
of polyhedra $\Delta^+$ and $\Delta^-$. 

\begin{figure}[hbt]
   \begin{center}
   \includegraphics[width=0.9\textwidth]{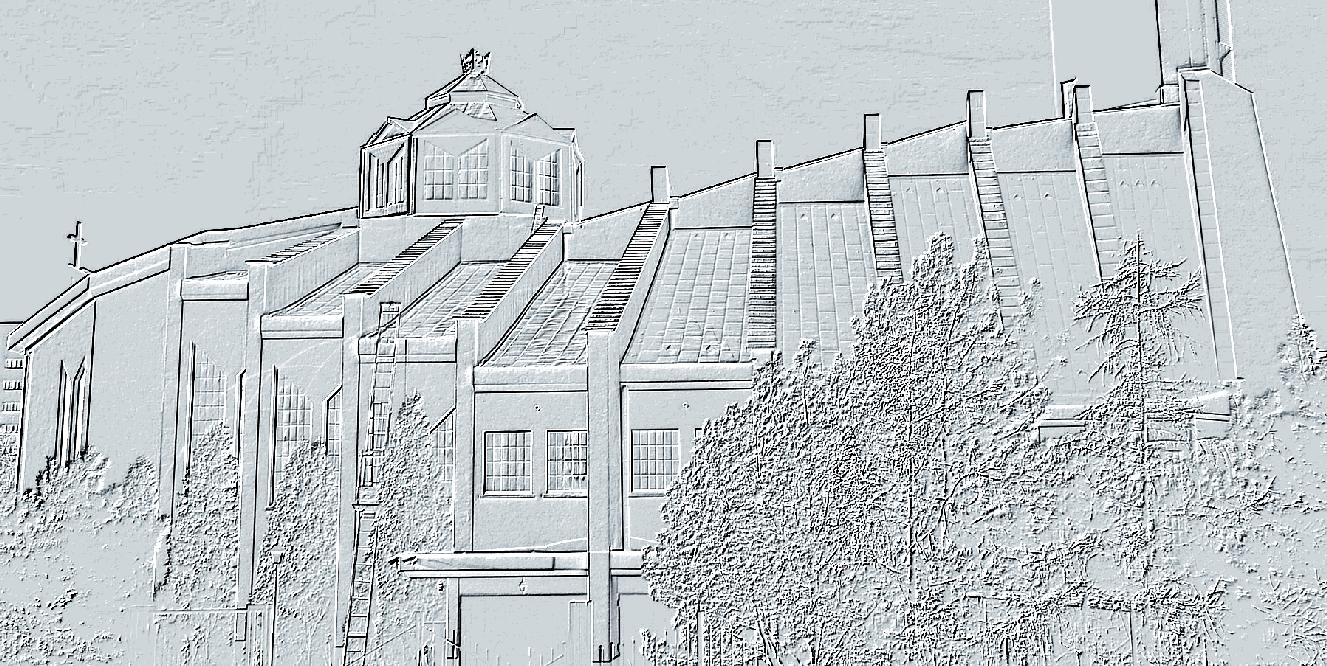}
   \end{center}
   \caption{The church of the Immaculate Conception of Blessed Virgin Mary 
in Warsaw.
   The shape of the roof resembles an illustration of a line bundle.}
\end{figure}

\begin{proposition}
\label{prop-diffPol}
Let $D = D^+ - D^-$ be a \WQC divisor on a semiprojective toric variety $X$ 
  with $D^+$ and $D^-$ nef. 
Denote the polyhedra of $D^+$ and $D^-$ by $\Delta^+$ and $\Delta^-$, 
respectively.
Then $\cO_X(D)$ is immaculate if and only if for all 
$\dm\in M$ the set $\Delta^-\setminus(\Delta^+-\dm)\subseteq M_\R$
is $\kk$-acyclic.
\end{proposition}

\begin{example}
\label{ex-intLattA}
With $D = D^+ - D^-$ and $\Delta^+$ and $\Delta^-$ as in Proposition~\ref{prop-diffPol},
if $\Delta^+$ is a (lattice) point, say $\Delta^+=\{0\}$, that is $D^+=0$,
then $D=-D^-$ is immaculate if and only if 
$\Delta^-$ does not contain (relatively) interior lattice points at all.
Note that this is supposed to exclude the case $\dim\Delta^-=0$, too.
See Theorem~\ref{thm_line_of_immaculates_on_toric_varieties} for a
generalization.
\end{example}

Next, we discuss the behaviour of immaculacy with respect to some special morphisms.
For a variety $X$ we denote by $\R\Gamma_X$ the derived global sections 
functor for coherent sheaves on $X$. Thus, a sheaf $\cF$ is immaculate
if and only if $\R\Gamma_X(\CF)=0$ in the derived category $\cD(X)$,
that is, if $\R\Gamma_X(\CF)$ is exact.
Similarly, for a morphism $p \colon X \to Y$ 
we denote by $\R p_*$ the derived push forward functor.

\begin{proposition}\label{prop_pullback_of_immaculate}
   Suppose $X$ and $Y$ are algebraic varieties over $\kk$ with $Y$ normal,
      and $p \colon X \to Y$ is a surjective proper morphism with 
connected fibres 
      such that $R^i p_* \cO_X =0$ for $i>0$,
that is, $\R p_*\CO_X=\CO_Y$.
   Assume $\cE$ is a locally free sheaf on $Y$. 
   Then $\cE$ is immaculate if and only if $p^*\cE$ is immaculate on $X$.
\end{proposition}

\begin{proof}
This follows from
$\R\Gamma_Y\CE=\R\Gamma_Y (\R p_* p^* \cE)=\R\Gamma_X(p^* \cE)$.
\end{proof}

The assumptions of Proposition~\ref{prop_pullback_of_immaculate} are satisfied
in the typical setting of the morphisms arising from the Minimal Model Program.
\begin{corollary}\label{cor_pullback_of_immaculate_on_toric_varieties}
   Suppose $X$ and $Y$ are toric varieties,
      and $p \colon X \to Y$ is a surjective toric projective morphism with connected fibres, 
      that is a toric projective morphism corresponding to a surjective map of one-parameter subgroups lattices $N_X \to N_Y$.
   Assume $\cE$ is a locally free sheaf on $Y$. 
   Then $\cE$ is immaculate if and only if $p^*\cE$ is immaculate.
\end{corollary}

\begin{proof}
To apply Proposition~\ref{prop_pullback_of_immaculate} we must 
ensure that $R^i p_* \cO_X =0$ for $i>0$.
For this, we may assume that $Y$ is affine and have to check that 
$\gH^i(\cO_X)=0$.
Since $p$ is projective,
the support of the fan $\Sigma_X$ of $X$ is a convex cone.
Thus for $\dm\in M_X$ the $\dm$-th grading of $\gH^i(\cO_X)$ is calculated by  
$V_{0, \dm}^{\supp} = \set{a \in \suppSig_X \mid  \langle a, \dm \rangle <0}$ 
which is convex, hence either contractible or empty.
\end{proof}

\begin{example}
\label{ex_hexagon-pullBack}
   In the notation of Example~\ref{ex_hexagon_introduce_coordinates} (``hexagon''), consider the divisor 
   \[
      D = (-2,-2,-2,-2) = -2 \hextriangledownright.
   \]
   It is a pullback of the immmaculate line bundle $\cO_{\PP^2}(-2)$ under the blow-down map to $\PP^2$ (contracting three disjoint exceptional divisors),
     thus $D$ is also immaculate.
\end{example}

\subsection{Relative immaculacy and affine spaces of immaculate line bundles}

The main goal of this subsection is to explain the 
occurrence of some infinite families of immaculate line bundles.
For this we present a more restrictive notion
than plain immaculacy, which leads to a construction of such families. 

\begin{definition}
\label{def-pImmac}
   Suppose $p\colon X \to Y$ is a morphism of algebraic varieties. 
   We say that a sheaf $\cF$ on $X$ is \emph{$p$-immaculate} 
   if the direct image sheaves $R^{i} p_* \cF$ vanish in all 
cohomological degrees $i \in \Z$, that is,
if $\R p_* \cF=0$ is exact.
\end{definition}
Clearly, a sheaf on $X$ is immaculate if and only if it is $p$-immaculate for 
the map \mbox{$p\colon X\to \set{*}$.}
Moreover, for any map $p\colon X \to Y$, 
the equality $\R\Gamma_X=\R\Gamma_Y\circ\R p_*$ implies that
each $p$-immaculate sheaf is automatically immaculate.
And, finally, it is a consequence of cohomology and base change that for a 
flat morphism the relative immaculacy of locally free sheaves can be 
checked fiberwise:

\begin{proposition}\label{prop_p_immaculate_for_flat_is_fibrewise}
   Suppose that $p\colon X \to Y$ is a flat proper morphism of algebraic 
varieties.
   Let $\cE$ be a locally free sheaf on $X$
and for $y\in Y$ denote by $X_y:=f^{-1}(y)$ the fiber of $y$.
   Then $\cE$ is $p$-immaculate if and only if 
$\cE_y:=\cE|_{X_y}$ is immaculate for every closed point $y$.
\end{proposition}

\begin{proof}
   If $\cE|_{X_y}$ is immaculate for every closed point $y$, then the 
functions $y \mapsto \dim \gH^i(X_y, \cE|_{X_y})$ are constantly equal to $0$,
     on closed points. Hence, by semicontinuity 
\cite[Cor.~1 in Sect.~5, p.~50]{mumford_abelian_varieties}, 
they are also zero on non-closed points.
  Thus by the implication (i)\then(ii) in \cite[Cor.~2 in Sect.~5, pp.~50--51]{mumford_abelian_varieties}, the sheaf $R^ip_*(\cE)$ is locally free
     and zero at every point of $Y$, that is $R^ip_*(\cE)=0$.
    
  If $\cE$ is $p$-immaculate, then for sufficiently large $i$ the equivalent conditions of 
     \cite[Cor.~2 in Sect.~5, pp.~50--51]{mumford_abelian_varieties} are satisfied, 
     hence by the last paragraph of that corollary the map $R^{i-1}p_*(\cE)
\otimes \kappa(y) \to \gH^{i-1}(X_y, \cE|_{X_y})$ is an isomorphism.
  Moreover, $R^{i-1}p_*(\cE) =0$, hence the condition (ii) is satisfied for a smaller value of $i$ and hence also condition (i) is satisfied.
  Going down with $i$, we eventually get the claim.
\end{proof}

\begin{proposition}\label{prop_p_immaculate_implies_twists_are_immaculate}
   Suppose that $X$ and $Y$ are varieties and $p\colon X \to Y$ is a morphism.
   Assume that $\cF$ is a $p$-immaculate coherent sheaf on $X$.
   Then, for any locally free sheaf $\cE$ on $Y$,
     the sheaf $\cF \otimes p^*\cE$ is $p$-immaculate, hence immaculate.
\end{proposition}
\begin{proof}
   The projection formula implies $R^i p_*(\cF \otimes p^*\cE) = R^i p_*\cF \otimes \cE$
     and the latter is zero by the definition of a $p$-immaculate sheaf.
\end{proof}

While the previous claims followed from rather standard arguments, 
it is quite nice that,
in the projective setting, also the converse of the above statement holds
true:

\renewcommand{\theenumi}{(\alph{enumi})}
\renewcommand{\labelenumi}{\theenumi}

\begin{theorem}\label{thm_p_immaculate_and_stable_immaculate}
      Suppose $X$ and $Y$ are varieties, $p\colon X \to Y$ is a morphism, and $Y$ is projective.
      Assume $\cF$ is a coherent sheaf on $X$, and $L$ is an ample line bundle on $Y$. 
      Then the following conditions are equivalent.
      \begin{enumerate}
       \item \label{item_F_p_immaculate}
             $\cF$ is $p$-immaculate,
       \item \label{item_F_twisted_by_vector_bundle}
             for any locally free sheaf $\cE$ on $Y$ the sheaf $\cF \otimes p^*\cE$ is immaculate,
       \item \label{item_F_twisted_by_line_bundle}
             for any Cartier divisor $D$ on $Y$ the sheaf $\cF \otimes \cO_X(p^*D)$ is immaculate,
       \item \label{item_F_twisted_by_ample}
             for any integer $k>0$ the sheaf $\cF \otimes p^*L^{\otimes k}$ is immaculate.
      \end{enumerate}
\end{theorem}
\begin{proof}
   The implication~\ref{item_F_p_immaculate}\then\ref{item_F_twisted_by_vector_bundle}
     is shown in Proposition~\ref{prop_p_immaculate_implies_twists_are_immaculate}.
   The implications~\ref{item_F_twisted_by_vector_bundle}\then\ref{item_F_twisted_by_line_bundle}\then\ref{item_F_twisted_by_ample} 
     are clear.
   Thus we only have to show  
\ref{item_F_twisted_by_ample}\then\ref{item_F_p_immaculate}.
   
   By the derived projection formula $(\R p_* \cF )\otimes L^{\otimes k}
\simeq \R p_*(\cF\otimes p^*L^{\otimes k})$ in $\cD(Y)$.
   Applying the derived global sections functor  $\R\Gamma_Y$ we obtain that
   \[
      \R\Gamma_Y (\R p_* \cF \otimes L^{\otimes k})) \stackrel{\text{q.is.}}{\simeq} (\R\Gamma_Y \circ \R p_*)(\cF\otimes p^*L^{\otimes k})
= \R\Gamma_X(\cF\otimes L^{\otimes k})=0
   \]
by our assumption in~\ref{item_F_twisted_by_ample}.
The entries in the second table,
that is, in the $E_2$ layer of the spectral sequence 
for $\R\Gamma_Y (\R p_* \cF \otimes L^{\otimes k}))$ are 
$\gH^i(Y, \,R^j p_* \cF \otimes L^{\otimes k})$ for varying $i,j$.
   
   By Serre vanishing, for sufficiently large $k$, 
we  have $\gH^i( R^j p_* \cF \otimes L^{\otimes k}) =0$ for all $i>0$
and all $j$.
   Hence for such $k$ the spectral sequence stabilises immediately 
and thus (since it converges to $0$) also the $\gH^0$ row is identically zero.
   That is $\gH^0( R^j p_* \cF \otimes L^{\otimes k}) =0$ for all sufficiently large $k$.
   Hence $R^j p_* \cF$ is a coherent sheaf on a projective variety $Y$, whose 
corresponding graded module is zero for all sufficiently large degrees. 
   Therefore, still for all $j$, the sheaves $R^j p_* \cF$ are identically zero by \cite[Exercise~II.5.9(c)]{hartshorne},
     which is the content of~\ref{item_F_p_immaculate}.
\end{proof}

We now switch our attention back to toric varieties.
Our goal is to reinterprete $p$-immaculacy and apply Theorem~\ref{thm_p_immaculate_and_stable_immaculate} in terms of toric geometry.
The following statement captures our main reason to study the cohomology of divisors on semiprojective varieties,
   despite that we are principally interested in projective varieties.
For a projective toric morphism $X\to Y$, we can restrict to an open affine subset of $Y$, and our theory still works,
   despite that we no longer live in the projective world.
Technically, the following characterization of $p$-immaculacy
differs from the characterisation of plain immaculacy in
Proposition~\ref{prop-diffPol} just by enlarging the tail cones.

\begin{proposition}
   Suppose $p\colon X\to Y$ is a toric map of semiprojective toric varieties,
      and let $p^*\colon M_Y \to M_X$ be the corresponding map of monomial lattices.
   Let $D$ be a \WQC divisor on $X$ and write $D=D^+ - D^-$ as a difference of nef divisors, as usual.
   Then $\cO_X(D)$ is $p$-immaculate if and only if 
      for all maximal cones $\sigma$ in the fan of $Y$ and for all 
$\dm\in M_X$ the difference 
      $(\Delta^- +p^*(\sigma^{\vee})) \setminus 
(\Delta^+ + p^*(\sigma^{\vee})- \dm)$ is $\kk$-acyclic.
\end{proposition}
\begin{proof}
   Let $\Sigma_Y$ be the fan of $Y$ and for a maximal cone $\sigma \in \Sigma(\dim Y)$ denote by $U_{\sigma}$ the open affine subset of $Y$ corresponding to $\sigma$. 
   By \cite[Prop.~III.8.5]{hartshorne} the sheaf $\cO_X(D)$ is $p$-immaculate 
      if and only if $\gH^i(\cO_{p^{-1}(U_{\sigma})}(D))=0$ 
      for all $i$ and for all $\sigma \in \Sigma_Y(\dim Y)$.
   Equivalently, for all $\sigma$ the restriction of $D$ to $p^{-1}(U_{\sigma})$ is immaculate. 
   The restriction of $D^+$ to $p^{-1}(U_{\sigma})$ is still nef 
      and the polyhedron of the restriction is equal to $\Delta^+ +p^*(\sigma^{\vee})$.
   Anologous statements hold for $D^-$ and $\Delta^-$.
   Therefore, the claim follows from Proposition~\ref{prop-diffPol} applied to each $p^{-1}(U_{\sigma})$ separately.
\end{proof}

A sublattice $M'\subset M$ is \emph{saturated} if $M \cap M'_{\RR} = M'$ (the intersection is taken in $M_{\RR}$).
For a rational
polyhedron $\Delta \subset M_{\RR}$ define its \emph{linear sublattice span} to be the smallest saturated sublattice $M'\subset M$
  containing a translate of $\Delta$.
Therefore $\Delta \subset m+ M'_{\RR}$ for any $m\in \Delta$, and $\dim \Delta = \dim M'$.
  
\newcommand{\quot}[2]{{^{\displaystyle #1}\!\! \left/ \!\! _{\displaystyle #2}\right.}}
\newcommand{\scalp}[1]{\langle {#1} \rangle}
\renewcommand{\theenumi}{(\arabic{enumi})}
\renewcommand{\labelenumi}{\theenumi}

The following theorem can be interpreted as a relative version of Example~\ref{ex-intLattA}.
\begin{theorem}\label{thm_line_of_immaculates_on_toric_varieties}
  Assume $X$ is a projective toric variety, $D^-$ is a nef 
  \WQC divisor, and $D'$ is a nef Cartier divisor on $X$.
  Suppose $\Delta^-$ and $\Delta'$ are their respective polytopes, and let $M'\subset M$ 
     be the linear sublattice span of $\Delta'$.
  Let $Y$ be the projective toric variety corresponding to $\Delta'$ and $p\colon X\to Y$ be the natural map of toric varieties.
  Then the following conditions are equivalent.
  \begin{enumerate}
     \item \label{item_all_shifts_immaculate}
          For all integers $a$ the divisors $aD' - D^-$ are immaculate on $X$.
     \item \label{item_infinitely_many_shifts_immaculate}
          For infintely many integers $a$ the divisors $aD' - D^-$ are immaculate.
     \item \label{item_projection_has_no_interior_points}
          The image of $\Delta^-$ under the projection $\varphi\colon M_{\RR} \to \quot{M_{\RR}}{M'_{\RR}} = \left(\quot{M}{M'}\right)\otimes \RR$
             has no lattice points in the relative interior.
     \item \label{item_anti_nef_is_p_immaculate}
          The divisor $- D^-$ is $p$-immaculate.
  \end{enumerate}
\end{theorem}

\begin{proof}
   The implication \ref{item_all_shifts_immaculate}\then\ref{item_infinitely_many_shifts_immaculate} is clear.
   
   To show \ref{item_infinitely_many_shifts_immaculate}\then\ref{item_projection_has_no_interior_points}
      we consider two cases,
      \emph{positive} or \emph{negative}.
   That is, among the integers $a$ such that $D_a:=aD' - D^-$ is immaculate,
      there exists a subsequence either of positive $a_i$ converging to $+\infty$ or 
   of negative $a_i$ converging to $-\infty$.

   In the positive case,
      suppose by contradiction, that there exist an interior lattice point of $\varphi(\Delta^-)$.
   Replacing $\Delta^-$ with its translate (and $D^-$ with a linearly equivalent divisor) if necessary, 
      we may assume that, say, $0\in \relint \varphi(\Delta^-)$.
   Choosing a subsequence if necessary, assume that every $|a_i|\Delta'$
      has a lattice point $\dm_i\in M$ in the relative interior such that the distance (with respect to any fixed norm on $M_\RR$) 
      of $\dm_i$ to the boundary $\partial(|a_i|\Delta')$ converges to 
$+\infty$.
   A nef decomposition of $D_{a_i}= a_i D' - D^-$ is exactly $D_{a_i}^+ = a_i D'$ and $D_{a_i}^-=D^-$.
   By Proposition~\ref{prop-diffPol} for any $i$ the difference 
$\Delta^{-}\setminus  (a_i \Delta'- \dm_i)$ is $\kk$-acyclic.
   Since $\Delta^{-}$ is compact, taking $a_i$ very large we have 
$\Delta^{-}\setminus  (a_i \Delta'- \dm_i) = \Delta^{-}\setminus  M'_{\R}$.
   By the criterion of~\cite{smale_Vietoris_mapping_theorem}, the restricted projection map
      $\varphi \colon \Delta^{-}\setminus  M'_{\R} \to \varphi(\Delta') \setminus \set{0}$ is a homotopy equivalence,
      a contradiction, since the first one $\Delta^{-}\setminus  M'_{\R}$ is $\kk$-acyclic,
      and the latter one $\varphi(\Delta^-) \setminus \set{0}$ is either homeomorphic to a sphere (if $\dim \varphi(\Delta^-)>0$) or empty (if $\varphi(\Delta^-)=\set{0}$).

   In the negative case, 
      the nef decomposition of $D_{a_i}$ is $D_{a_i}^+ = 0$ and $D_{a_i}^-=D^- - a_i D'$.
   By Example~\ref{ex-intLattA} for any $a_i$ the Minkowski sum $\Delta^{-} + |a_i| \Delta'$ has no lattice points in the relative interior.
   Taking $|a_i|$ very large, we see that there are no lattice points in the relative interior of $\Delta^{-}+ M'_{\R}$. 
   Equivalently, there is no (relative) interior lattice point in $\varphi(\Delta^-)$.
   This concludes the proof of \ref{item_infinitely_many_shifts_immaculate}\then\ref{item_projection_has_no_interior_points}.
   
   Next we prove \ref{item_projection_has_no_interior_points}\then\ref{item_all_shifts_immaculate}.
Assume (by shifting $\Delta'$ if necessary) 
that $\Delta'\subseteq M'_\R$, that is,
that $\varphi(\Delta')$ equals $0\in \quot{M}{M'}$. 
Assume $a$ is a nonnegative integer.
   We must show that
   \begin{itemize}
     \item the Minkowski sum $\Delta^- + a\Delta'$ has no interior lattice 
points (hence $-D^- - a D'$ is immaculate), and 
     \item $\Delta^-\setminus (a\Delta' -\dm)$ is contractible and non-empty
for all $\dm\in M$
(hence $-D^- + a D'$ has no cohomology in degree $\dm$). 
   \end{itemize}
   The first claim is straightforward: Such an interior lattice point would be mapped to an interior lattice point of $\varphi(\Delta^-)$,
      which is impossible by the  assumptions of \ref{item_projection_has_no_interior_points}.
   Also the second claim is easy.
   Let $P:=(a\Delta' - \dm) \cap \Delta^-$, which is a convex set contained 
in $\Delta^-$ such that $\Delta^-\setminus (a\Delta' -\dm)= 
\Delta^-\setminus P$.
Since $\Delta'$ is contained in $M'_{\R}$, also $a\Delta'\subset M'_{\R}$, and consequently $P \subset M'_{\R} - m$.
If $\varphi(-\dm) = \varphi(P) \notin \varphi(\Delta^-)$, then $P$ is disjoint with $\Delta^-$ and 
$\Delta^-\setminus P = \Delta^-$, which is contractible and non-empty as 
claimed.
Since $\varphi(\Delta^-)$ has no interior lattice points, 
   it remains to consider 
   $\varphi(-\dm)\in \partial (\varphi(\Delta^-))$ and,  
   consequently, $P\subset \partial \Delta^-$.
So the difference is nonempty and by Corollary~\ref{cor_difference_of_closed_polytopes_and_boundary}
      it is homotopic to $\partial \Delta^- \setminus P$, which is contractible (a sphere with a convex disc taken out).

To show \ref{item_all_shifts_immaculate}\ifff\ref{item_anti_nef_is_p_immaculate} note that $D'= p^*L$ for an ample line bundle $L$ on $Y$.
   We apply the implications 
\ref{item_F_twisted_by_ample}\then\ref{item_F_p_immaculate}\then \ref{item_F_twisted_by_line_bundle}
of Theorem~\ref{thm_p_immaculate_and_stable_immaculate}. 

\end{proof}

The following examples obey the notation of Theorem~\ref{thm_line_of_immaculates_on_toric_varieties}. 

\begin{example}
If \(\Delta'\) is full dimensional, that is if $D'$ is big, then
there is no antinef divisor which is $p$-immaculate, as in this case, 
$M'$ is the whole lattice $M$, and $\varphi(\Delta^-)$ 
is a point, thus having an interior lattice point by definition.
\end{example}

\begin{example}
If $\Delta'$ is just a point, then $M'=0$ and the question becomes whether
$\Delta^-$ contains any interior lattice points, as already discussed in Example~\ref{ex-intLattA}.
\end{example}

\begin{example}
If \(\Delta'\) has codimension one, then $M'$ is a hyperplane. 
The divisor $-D^-$ is $p$-immaculate if \(\Delta^-\) cannot be divided by integral shifts of \(M'\).
In case $D^-$ is in addition Cartier, this is equivalent to 
\[
   \max \scalp{\Delta^-,M'} - \min \scalp{\Delta^-,M'} \le 1,
\]
where we think of the hyperplane $M' \subset M$ as a primitive element of $N$ dual to the hyperplane. 
\end{example}

\begin{example}
\label{ex_hexagon-lines}
   In the hexagon case (Example~\ref{ex_hexagon_introduce_coordinates}), let $D' = (1,1,0,0)$ so that $\Delta' = \hexlinediag$, 
      and let $D^- = (1,1,1,0)$ so that $\Delta^- = \hextriangleupleft$.
   Then all combinations $aD'- D^- = (a-1,a-1,-1,0)$ are immaculate.
   Other lines of immaculate divisors on this surface are listed in Table~\ref{table:hexagon_lines} in Section~\ref{sec:computational}.
\end{example}

\section{Immaculacy by avoiding temptations}
\label{avoidTemptations}

Let us compare three examples of smooth projective varieties with Picard rank $2$: the product projective space $\PP^1\times \PP^1$, 
the Hirzebruch surface $\bF_1$ and the flag variety 
$\F(1,2;3):=\{(p,\ell)\in\PP^2\times(\PP^{2})\dual\kst p\in L\}$.
Note that the first one is simultanously a toric variety and a homogeneous space (for the semisimple group $\SL_2 \times \Sl_2$),
  the second is a toric variety, while the third one is a homogeneous space for the simple group $\SL_3$.
Figures~\ref{fig_Pic_of_P1xP1}, \ref{fig_Pic_of_F1}, \ref{fig_Pic_of_flag} illustrate the Picard lattices of these examples, indicating the regions of line bundles with nontrivial cohomologies.

\begin{figure}[hbt]
   \begin{center}
   \begin{tikzpicture}[scale=0.80]
     \draw[color=oliwkowy!40] (-5.3,-5.3) grid (5.3,5.3); 
     \fill[pattern color=black!80, pattern=north west lines] (0,0) -- (5.3,0) -- (5.3,5.3) -- (0,5.3) -- cycle;              
     \fill[pattern color=mocnyfiolet, pattern=dots] (0,0) -- (5.3,0) -- (5.3,5.3) -- (0,5.3)-- cycle;                   
     \fill[pattern color=ciemnazielen!80, pattern=north east lines] (-2,0) -- (-2,5.3) -- (-5.3,5.3) --(-5.3,0) -- cycle;   
     \fill[pattern color=ciemnazielen!80, pattern=north east lines] (0,-2) -- (5.3,-2) -- (5.3,-5.3) --(0,-5.3) -- cycle;   
     \fill[pattern color=jasnyfiolet, pattern=north west lines] (-2,-2) -- (-2,-5.3) -- (-5.3,-5.3) --(-5.3,-2) -- cycle;   
     
     \draw[thick,  color=black] (5.3,0)--(0,0) -- (0,5.3) (5.6,5.3) node[anchor=west, rotate=270] {\scriptsize{$\Eff$-cone$=\Nef$-cone$=\CM_{\R}(\emptyset)$ }} (-0.3,-0.3) node{\scriptsize{$0$}};     
     \draw[thick,  color=ciemnazielen] (-5.3,0) -- (-2,0) -- (-2,5.3) (-5.6,5.3) node[anchor=east, rotate=90] {\scriptsize{$\gH^1$-cone$=\CM_{\R} \left(\set{(1,0),(-1,0)}\right)$}};  
     \draw[thick,  color=ciemnazielen] (0,-5.3) -- (0,-2) -- (5.3,-2)  (2.7,-5.2) node[anchor=north] 
     {\scriptsize{$\gH^1$-cone$=\CM_{\R}\left(\set{(0,1),(0,-1)}\right)$}} ;					 
     \draw[thick,  color=jasnyfiolet!150] (-5.3,-2) -- (-2,-2) -- (-2,-5.3) (-5.6,-2.1) node[anchor=east, rotate=90] {\scriptsize{$\gH^2$-cone$=\CM_{\R}(\Sigma(1))$}} (-1.6,-1.7) node{\scriptsize{$K_X$}}; 
     
     \draw[thick,  color=oliwkowy] (-5.3,-1) --  (5.3,-1) (3,-0.9) node[anchor=north]{\scriptsize{immaculate locus}}  (-1, -5.3)--(-1, 5.3); 					 
   \end{tikzpicture}
   \end{center}
   \caption{The Picard lattice of the surface $\PP^1\times \PP^1$. 
     The effective cone $\Eff$ is the cone of divisors with nonzero 
$\gH^0$ and it coincides with the $\Nef$-cone.
    There are two cones of divisors with nonzero $\gH^1$, and one cone with nonzero $\gH^2$.
    The remaining line bundles are immaculate, and the immaculate locus 
consist of two lines parallel to the common facets of the 
$\Nef$- and $\Eff$-cones. These two lines correspond to the two projections 
to $\PP^1$.
    The notation $\CM_{\R}(\bullet)$ is explained in Section~\ref{RmacPic}.}
    \label{fig_Pic_of_P1xP1}
\end{figure}
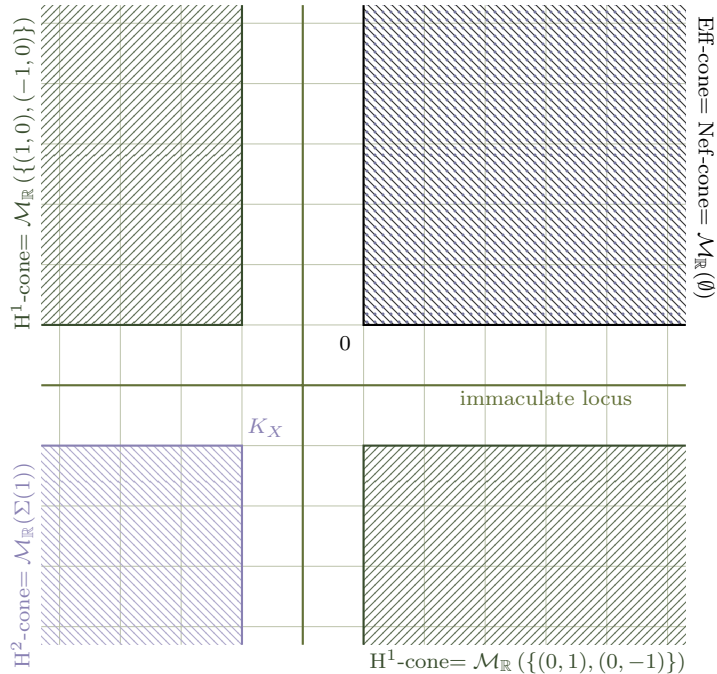

\begin{figure}[hbt]
   \begin{center}
   \begin{tikzpicture}[scale=0.8]
     \fill[color=oliwkowy!35] (-1,0) -- (-1,-1) -- (0,-2)--(0,-1) --cycle; 		 
     \draw[color=oliwkowy!40] (-5.3,-5.3) grid (5.3,5.3); 
     \fill[pattern color=black!80, pattern=north west lines] (0,0) -- (5.3,0) -- (5.3,5.3) -- (-5.3,5.3) -- cycle;              
     \fill[pattern color=red, pattern=dots] (0,0) -- (5.3,0) -- (5.3,5.3) -- (0,5.3)-- cycle;                   
     \fill[pattern color=ciemnazielen!80, pattern=north east lines] (-2,0) -- (-2,5.3) -- (-5.3,5.3) --(-5.3,0) -- cycle;   
     \fill[pattern color=ciemnazielen!80, pattern=north east lines] (1,-2) -- (5.3,-2) -- (5.3,-5.3) --(1,-5.3) -- cycle;   
     \fill[pattern color=jasnyfiolet, pattern=north west lines] (-1,-2) -- (2.3,-5.3) -- (-5.3,-5.3) --(-5.3,-2) -- cycle;   
     
     \draw[thick,  color=black] (5.3,0)-- (0,0) -- (-5.3,5.3) (-0.7,5.15) node[anchor=south] {\scriptsize{$\Eff$-cone$=\CM_{\R}(\emptyset)$}}(0.1,-0.2) node{\scriptsize{$0$}} ;     
     \draw[color=red] (5.3,0)--(0,0) -- (0,5.3) (5.5,3.5) node[rotate=270, anchor=west]  {\scriptsize{$\Nef$-cone}};   
     \draw[thick,  color=ciemnazielen] (-5.3,0) -- (-2,0) -- (-2,5.3) (-5.5,5.7) node[anchor=east,  rotate=90] {\scriptsize{an $\gH^1$-cone$=\CM_{\R}(\set{(0,1),(-1,-1)})$}};  
     \draw[thick,  color=ciemnazielen] (1,-5.3) -- (1,-2) -- (5.3,-2)  (4,-5.2) node[anchor=north] {\scriptsize{an $\gH^1$-cone$=\CM_{\R}(\set{(1,0),(-1,0)})$}} ;					 
     \draw[thick,  color=jasnyfiolet!150] (-5.3,-2) -- (-1,-2) -- (2.3,-5.3) (-5.5,-1.9) node[anchor=east, rotate=90] {\scriptsize{the $\gH^2$-cone$=\CM_{\R}(\Sigma(1))$}} (-1.1,-1.8) node{\scriptsize{$K_X$}}; 					 
     
     \draw[thick,  color=oliwkowy] (-5.3,-1) -- (-1,-1) (0,-1) -- (5.3,-1) 
        (-3.5,-0.9) node[anchor=north]{\scriptsize{immaculate locus}};  
     \draw[thick,  color=oliwkowy] (-1,0)  -- (-1,-1) -- (0,-2)--(0,-1) --cycle
              (-0.9,0.1) node{\scriptsize$\cO_{\PP^1}(-1)$}
              (-0.05,-0.45) node[anchor=north west]{\scriptsize$\cO_{\PP^2}(-1)$}
              (-0.05,-1.45) node[anchor=north west]{\scriptsize$\cO_{\PP^2}(-2)$}
              ;
   \end{tikzpicture}
   \end{center}
   \caption{The Picard lattice of the Hirzebruch surface $\F_1 = \toric(\Sigma)$, where $\Sigma$ has rays 
      $\Sigma(1)= \set{(0,1),(-1,-1),(1,0),(-1,0)}$.
    The effective cone $\Eff$ is the cone of divisors with nonzero $\gH^0$.
    There are two cones of divisors with nonzero $\gH^1$, and one cone with nonzero $\gH^2$.
    In addition the $\Nef$-cone is marked; it is a proper subset of 
the $\Eff$-cone.
    The remaining line bundles are immaculate, and the immaculate locus
consist of a bounded polytope and a line parallel to the 
unique common facet of 
the $\Nef$- and $\Eff$-cones. 
    This common facet $\R_{\geq 0}\cdot(1,0)$ corresponds to the fibration  $p\colon\F_1 \to \PP^1$.
    The other face $\R_{\geq 0}\cdot(0,1)$ of the $\Nef$-cone corresponds to
the blow-down map $\bl\colon \F_1\to \PP^2$.
    The line bundle $\bl^*\cO_{\PP^2}(-1)$ is $p$-immaculate, hence Theorem~\ref{thm_line_of_immaculates_on_toric_varieties} 
       explains the line of immaculate divisors. 
    The line bundles  $\bl^*\cO_{\PP^2}(-2)$ and $p^*\cO_{\PP^1}(-1)$ 
        are immaculate by Corollary~\ref{cor_pullback_of_immaculate_on_toric_varieties}.
    For clarity the figure ommits $p^*$ and $\bl^*$ in the names of line bundles.}
    \label{fig_Pic_of_F1}
\end{figure}
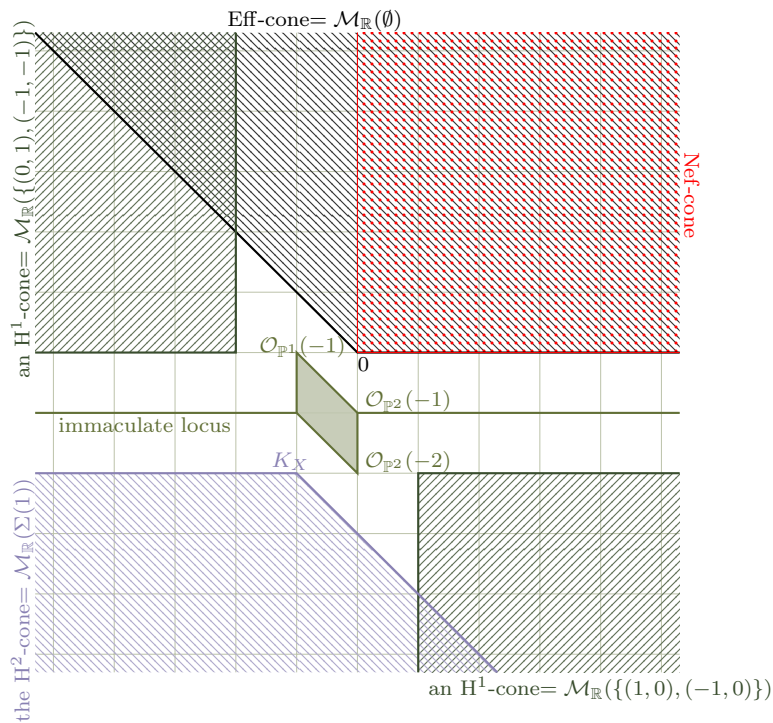

\begin{figure}[hbt]
   \begin{center}
   \begin{tikzpicture}[scale=0.8]
     \draw[color=oliwkowy!40] (-5.3,-5.3) grid (5.3,5.3); 
     \fill[pattern color=black!80, pattern=north west lines]  (5.3,0) -- (5.3,5.3) -- (0,5.3) -- (0,0)-- cycle (-0.3,-0.3) node[color=black]{\scriptsize{$0$}} ;             
     \fill[pattern color=mocnyfiolet, pattern=dots] (0,0) -- (5.3,0) -- (5.3,5.3) -- (0,5.3)-- cycle;                       
     \fill[pattern color=ciemnazielen!80, pattern=north east lines] (-2,1) -- (-2,5.3) -- (-5.3,5.3) --(-5.3,4.3)-- cycle;   
     \fill[pattern color=ciemnazielen!80, pattern=north east lines] (1,-2) -- (5.3,-2) -- (5.3,-5.3) --(4.3,-5.3) -- cycle;   
     \fill[pattern color=jasnyfiolet, pattern=north east lines] (-3,0) -- (-5.3,0) -- (-5.3,2.3)  -- cycle;   
     \fill[pattern color=jasnyfiolet, pattern=north east lines] (0,-3) -- (0,-5.3) -- (2.3,-5.3)  -- cycle;   
     \fill[pattern color=bladyfiolet, pattern=north west lines] (-2,-2) -- (-2,-5.3) -- (-5.3,-5.3)--(-5.3,-2)  -- cycle ;   
     
     \draw[thick,  color=black] (5.3,0)--(0,0) -- (0,5.3) (5.6,4.5) node[anchor=west, rotate=270] {\scriptsize{$\Eff$-cone$=\Nef$-cone}};     
     \draw[thick,  color=ciemnazielen] (-2,5.3) -- (-2,1) -- (-5.3, 4.3) (-3.5,5.3) node[anchor=south] {\scriptsize{$\gH^1$-cone}};  
     \draw[thick,  color=ciemnazielen] (5.3,-2) -- (1,-2) -- (4.3,-5.3)  (5.6,-2.5) node[anchor=west, rotate=270] {\scriptsize{$\gH^1$-cone}} ;	
     \draw[thick,  color=jasnyfiolet!150] (-5.3,2.3) -- (-3,0) -- (-5.3,0) (-5.6, 1.9) node[anchor=east, rotate=90] {\scriptsize{$\gH^2$-cone}}; 
     \draw[thick,  color=jasnyfiolet!150] (2.3,-5.3) -- (0,-3) -- (0,-5.3) (1.15, -5.3) node[anchor=north] {\scriptsize{$\gH^2$-cone}}; 
     \draw[thick,  color=bladyfiolet!150] (-2,-5.3) -- (-2,-2) -- (-5.3,-2) (-5.6, -3) node[anchor=east, rotate=90] {\scriptsize{$\gH^3$-cone}}
     (-1.6,-1.7) node{\scriptsize{$K_X$}}; 
     
     \draw[thick,  color=oliwkowy] (-5.3,-1) --  (5.3,-1) (3.5,-0.9)node[anchor=north]{\scriptsize{immaculate locus}}  (-1, -5.3)--(-1, 5.3)  (-5.3,3.3)--(3.3,-5.3);  
   \end{tikzpicture}
   \end{center}
   \caption{The Picard lattice of the threefold flag variety $\F(1,2,3)$. 
    The effective cone $\Eff$ is the cone of divisors with nonzero $\gH^0$ and it coincides with $\Nef$-cone.
    There are two cones of divisors with nonzero $\gH^1$, two cones with nonzero $\gH^2$, and one cone with nonzero $\gH^3$.
    The remaining line bundles are immaculate, and the immaculate locus
consists of three lines. The diagonal is not parallel to joined faces of
the $\Nef$- and $\Eff$-cones, that is, it is not predicted by a contraction.
    }
    \label{fig_Pic_of_flag}
\end{figure}
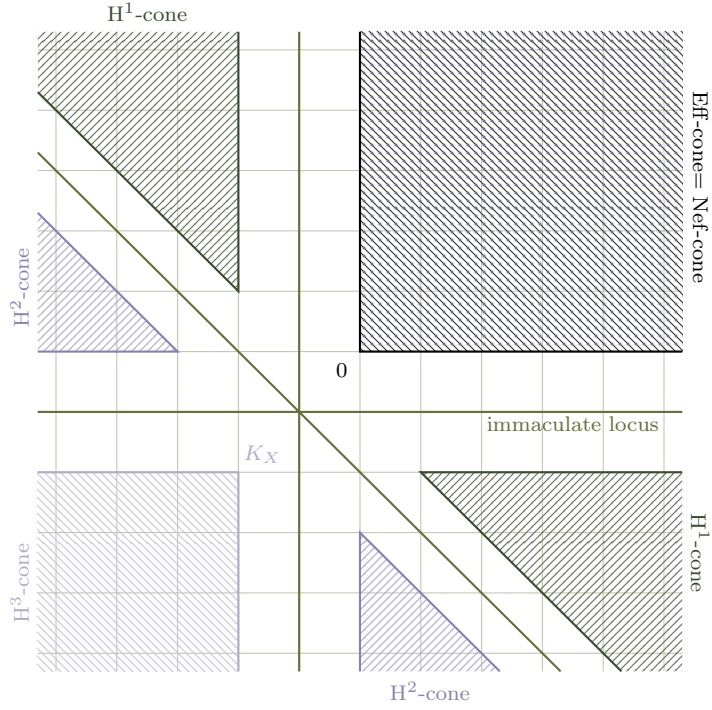

For homogeneous spaces, the regions for various $\gH^i$ are disjoint, that is, for every line bundle $L$ there is at most one value of $i$, 
   such that $\gH^i(L) \ne 0$, see for instance \cite[Thm 5.14]{kostant_Borel_Weil}.
For toric varieties this is not necessarily the case. As illustrated by 
the $\bF_1$ example, the regions may intersect. 
The goal of this section is to show how to obtain these regions of line bundles with various cohomologies for any toric variety.

Our treatment of immaculate divisors in this section is analogous to the treatment of acylic line bundles 
   in \cite[Sect.~4]{borisov_hua_conjecture_of_King_for_DM_stacks}.

\subsection{Temptations}
\label{RmacPic}

Let $X = \toric(\Sigma)$ be a toric variety with no torus factors. 
For any subset $\CR\subseteq\Sigma(1)$ we define  $V^{>}(\CR)\subset M_{\R}$,
   similar to $V^{>}_{D,0}$ as in Section~\ref{toricCohom}:
\[
   V^{>}(\CR):= \R_{>0} \cdot \left(\bigcup_{\sigma \in \Sigma} \conv (\CR \cap\sigma(1)) \right).
\]
Moreover define $V^{\ge}(\CR)$ as the complex of cones $\set{\cone (\CR \cap\sigma(1)) \mid \sigma \in \Sigma}$ in $M_{\R} $, so that 
\[
  \supp V^{\ge}(\CR) = V^> (\CR)  \cup \set{0}.
\]
In fact, $V^{>}(\CR) = V^{>}_{- \sum_{\rho \in \CR} D_{\rho}, 0}$ and analogously for $V^{\ge}$.
Thus, as in  Section~\ref{toricCohom}, if $\Sigma$ is in addition simplicial, 
   then $V^{\ge}(\CR)$ is a full (``induced'') subcomplex of $\Sigma$ generated by $\CR$.

\begin{definition}
\label{def-tempting}
We call $\CR\subseteq\Sigma(1)$ \emph{tempting} if the geometric realization 
$V^{>}(\CR)$ of $V^{\ge}(\CR)\setminus\{0\}$ admits some reduced cohomology,
   that is if it is not $\kk$-acyclic.
\end{definition}

\begin{example}\label{ex_hexagon_tempting_subsets}
   Following with our ``hexagon'' example (see notation in Example~\ref{ex_hexagon_introduce_coordinates}),
     the fan $\Sigma$ of this surface has the following $34$ tempting subsets $\CR\subseteq\Sigma(1)$:
\begin{gather*}
\emptyset, \{0,2\},\{0,3\},\{0,4\},\{1,3\},\{1,4\},\{1,5\},\{2,4\},\{2,5\},\{3,5\},\{0,1,3\},\{0,1,4\},\\
\{0,2,3\},\{0,2,4\},\{0,2,5\},\{0,3,4\},\{0,3,5\},\{1,2,4\},\{1,2,5\},\{1,3,4\},\\
\{1,3,5\},\{1,4,5\},\{2,3,5\},\{2,4,5\},\{0,1,2,4\},\{0,1,3,4\},\{0,1,3,5\},\{0,2,3,4\},\\
\{0,2,3,5\},\{0,2,4,5\},\{1,2,3,5\},\{1,2,4,5\},\{1,3,4,5\},\{0,1,2,3,4,5\}.
\end{gather*}
\end{example}

As in Section~\ref{sect_toric_notation} we denote both natural maps 
$\Z^{\Sigma(1)} \to \Cl(X)$ and $\R^{\Sigma(1)} \to \Cl(X) \otimes \RR$ 
by $\pi$.

\begin{definition}
\label{def-RmacPic}
Let $\CR\subseteq\Sigma(1)$ be a subset. 
Then, we denote the images
\begin{align*}
   \CM_\Z(\CR)&:= \pi\left(\pms\right),\\ 
   \CM_{\R}(\CR)&:= \pi\left(\pmr\right). 
\end{align*}
If $\CR$ is tempting as defined above, then $\CM_\Z(\CR)$ is 
called the \emph{$\CR$-maculate} set of $\Cl(X)$, 
   respectively, $\CM_{\R}(\CR)$ is the \emph{$\CR$-maculate} region of $\Cl(X)\otimes \R$. 
\end{definition}

\begin{remark}\label{rem_Alexander_duality_for_tempting_subsets}
Suppose the fan $\Sigma$ is complete.
The empty set $\CR=\emptyset$ yields
$\CM_{\R}(\emptyset)=\Eff(X)$. 
Moreover, Alexander duality implies that
switching between $\CR$ and $\Sigma(1)\setminus\CR$ does not change the 
temptation status. After applying $\CM$, the
relation between the subsets $\CM_{\Z}(\CR)$ and $\CM_{\Z}(\Sigma(1)\setminus\CR)$ of $\Cl(X)$ 
becomes Serre duality in $X=\toric(\Sigma)$.
\end{remark}

The integral sets $\CM_\Z(\CR)\subseteq \Cl(X)$
reflect more precisely the properties we need, 
  but the real regions $\CM_{\R}(\CR)$ are easier to control and they 
already contain a lot of information.
Note that under the natural map $\kappa \colon \Cl(X) \to \Cl(X)\otimes \R$, 
$[D] \mapsto [D]\otimes 1$,
the \emph{$\CR$-maculate} set is mapped into the \emph{$\CR$-maculate} region, that is $\kappa\colon \CM_{\Z} (\CR) \to \CM_{\R}(\CR)$.
In other words, the preimage $\kappa^{-1} \CM_{\R}(\CR)$ in $\Cl(X)$ 
contains $\CM_{\Z} (\CR)$, or, slightly incorrect,
$\CM_{\Z}(\CR)\subseteq \CM_{\R}(\CR)\cap \Cl(X)$. 
We will encounter several situations when $\kappa^{-1} \CM_{\R}(\CR)$ and  $\CM_{\Z} (\CR)$ are either equal or not equal, depending on the saturation of respective cones.

\renewcommand{\theenumi}{(\roman{enumi})}
\renewcommand{\labelenumi}{\theenumi}

\begin{proposition}
Suppose $X = \toric(\Sigma)$ is a toric variety with no torus factors. 
\label{prop-Rmac}
\begin{enumerate}
 \item\label{item_interpret_M_Z_R}
    Let $\CR\subseteq\Sigma(1)$ be a subset, 
       and suppose $[D]\in \Cl(X)$ is a class of a Weil divisor $D$ on $X$.
    Then $[D]$ belongs to $\CM_{\Z}(\CR)$ if and only if
       $D$ is linearly equivalent to some 
$\sum_{\rho\in\Sigma(1)} \lambda_{\rho} \cdot D_\rho$ 
with $\lambda_{\rho}\in\Z$ and
    $\CR=\{\rho\in\Sigma(1) \mid \lambda_{\rho} <0\}$.
  \item\label{item_interpret_M_R}
     Again, let $\CR\subseteq\Sigma(1)$, 
       and suppose $[D]\in \Cl(X)$ is a class of a Weil divisor $D$ on $X$.
     Then $[D]_{\R} \in \Cl(X)\otimes \R$ belongs to $\CM_{\R}(\CR)$, 
       if and only if $D$ is $\Q$-linearly equivalent to $\sum_{\rho\in\Sigma(1)} \lambda_{\rho} \cdot D_\rho$ 
       (for rational $\lambda_{\rho}$) with 
     $\CR=\{\rho\in\Sigma(1) \mid \lambda_{\rho} < 0\}$.
  \item\label{item_R_maculate_then_nontrivial_cohomology}
     If $\CR\subseteq\Sigma(1)$ is tempting, 
       then for any $i$ such that $\tH^{i-1}(V^>(\CR), \kk) \ne 0$ and any Weil divisor $[D]\in \CM_{\Z}(\CR)$, we have $\gH^{i}(\cO_X(D)) \ne 0$.
  \item\label{item_not_Z_maculate_iff_immaculate}
    A rank one reflexive sheaf $\cO_X(D)$ for $[D]\in\Cl(X)$ is immaculate
       if and only if $ D \notin \bigcup_{\CR=\mbox{\tiny\textrm{tempting}}}\CM_{\Z}(\CR)$.
  \item\label{item_not_maculate_are_immaculate}
     A rank one reflexive sheaf $\cO_X(D)$ such that
$[D]_{\R}\notin 
\bigcup_{\CR=\mbox{\tiny\textrm{tempting}}}\CM_{\R}(\CR)$ 
is immaculate.
\end{enumerate}
\end{proposition}
This statement should be compared with \cite[Prop.~4.3 and~4.5]{borisov_hua_conjecture_of_King_for_DM_stacks}.

\begin{proof}
The divisor $D$ of \ref{item_interpret_M_Z_R} or \ref{item_interpret_M_R} 
belongs to $\CM_{\Z}(\CR)$ or $\CM_{\R}(\CR)$
if and only if it is an image under $\pi$ of $\pms$ or $\pmr$, respectively.
The kernel of $\pi$ is the set of principal torus invariant divisors, hence the claim holds.

To see \ref{item_R_maculate_then_nontrivial_cohomology}, 
take $[D]\in \CM_{\Z}(\CR)$, and a linearly equivalent 
$D'=\sum_{\rho\in\Sigma(1)} \lambda_{\rho} \cdot D_\rho=D(\dm)$ 
as in \ref{item_interpret_M_Z_R}.
Then by \cite[Thm~9.1.3]{CoxBook} the appropriate cohomology group 
is $\gH^{i}(\cO_X(D))_\dm=\gH^{i}(\cO_X(D'))_0\neq 0$.

If $D$ is immaculate, then it is not in $\bigcup_{\CR=\mbox{\tiny\textrm{tempting}}}\CM_{\Z}(\CR)$ by \ref{item_R_maculate_then_nontrivial_cohomology}.
Conversely, if $D$ is not immaculate, then pick a linearly equivalent divisor $\sum_{\rho\in\Sigma(1)} \lambda_{\rho} \cdot D_\rho$ which has non-trivial cohomologies in degree $0 \in M$.
By \cite[Thm~9.1.3]{CoxBook} the set $\CR = \{\rho\in\Sigma(1) \mid \lambda_{\rho}< 0\}$ is tempting and $[D] \in \CM_{\Z}(\CR)$, concluding the proof of \ref{item_not_Z_maculate_iff_immaculate}.

Finally, \ref{item_not_maculate_are_immaculate} follows from \ref{item_not_Z_maculate_iff_immaculate}, since $[D]\in \CM_{\Z}(\CR)$ implies $[D]_{\R} \in \CM_{\R}(\CR)$.
\end{proof}

It is not always true, that $[D]_{\R} \in \CM_{\R}(\CR)$ implies $[D] \in \CM_{\Z}(\CR)$ as the following example shows.
\begin{example}\label{ex_P235_not_really_immaculate}
   Let $X=\toric(\Sigma) = \PP(2,3,5)$, the weighted projective plane with weights $2$, $3$, $5$.
   Consider the \WQC divisor $D \simeq \cO_X(1)$ which can be written as the difference $D_{\rho_2}-D_{\rho_1}$. 
   Then $D$ is immaculate, but $[D]_{\R} \in \CM_{\R}(\CR)$ for 
$\CR = \emptyset$ (corresponding to the $\Eff_\R$-cone). 
\end{example}

This leads to the following definition:
\begin{definition}\label{def_really_immaculate}
   A divisor $D$ is \emph{really immaculate} (or \emph{$\R$-immaculate}), if 
   \[
     [D]_{\R} \in \Cl(X)\otimes \R \setminus \bigcup_{\CR=\mbox{\tiny\textrm{tempting}}}\CM_{\R}(\CR).
   \]
\end{definition}

Thus Example~\ref{ex_P235_not_really_immaculate} shows a simple case of an immaculate Weil divisor that is not  really immaculate.
In Example~\ref{ex_immaculate_but_not_really_immaculate} we construct a line bundle on a smooth toric projective variety with the same property.
Up to the zero-th cohomology group, the concept of really immaculate divisor here is an analogue of the
   \emph{strongly acyclic} line bundle in \cite[Def.~4.4]{borisov_hua_conjecture_of_King_for_DM_stacks}.

\begin{definition}\label{def_immaculate_loci}
The \emph{immaculate loci} of $X$ are 
\begin{align*}
            \Imm_{\Z}(X) &= \Cl(X)  \setminus \bigcup_{\cR\subset \Sigma(1), \ \cR \text{ is tempting}} \CM_{\Z}(\cR), \text{ and}\\
            \Imm_{\R}(X) &= \kappa^{-1}\left((\Cl(X)\otimes \R)  \setminus \bigcup_{\cR\subset \Sigma(1), \ \cR \text{ is tempting}} \CM_{\R}(\cR) \right)  \subset \Cl(X),
\end{align*}
where $\kappa \colon \Cl(X) \to \Cl(X) \otimes \R$ is the natural map 
$[D] \mapsto [D]\otimes 1=[D]_\R$.
\end{definition}
Thus $\Imm_{\Z}(X)$ is the collection of all immaculate divisors.
By Proposition~\ref{prop-Rmac}\ref{item_not_maculate_are_immaculate} all the divisors in $\Imm_{\R}(X)$ are immaculate,
   that is $\Imm_{\R}(X) \subset \Imm_{\Z}(X)$.
More precisely, $\Imm_{\R}(X)$ is the set of all really 
immaculate divisors as in Definition~\ref{def_really_immaculate}.

\begin{example}\label{ex_hexagon_all_really_immaculate}
   In contrast to Examples~\ref{ex_P235_not_really_immaculate} and~\ref{ex_immaculate_but_not_really_immaculate}, 
      we can see that in the case of the hexagon (Example~\ref{ex_hexagon_introduce_coordinates}), all immaculate line bundles are really immaculate. 
   This follows since the matrix $\pi$ defining the map $(\Z^{\Sigma(1)})^* \to \Pic(X)$ is totally unimodular.
\end{example}

\begin{example}
\label{ex-Hirz}
We illustrate Proposition~\ref{prop-Rmac} with the example of
the Hirzebruch surface $\F_a=\toric(\Sigma_a)$. 
The special cases $a=0$ and $a=1$ are presented in 
the Figures~\ref{fig_Pic_of_P1xP1} and~\ref{fig_Pic_of_F1}, respectively.
More general cases are explained in Subsection~\ref{picTwo}
   --- our surface case corresponds to $\ell_1=\ell_2=2$ there.

The Gale transform, that is the map $\pi$, is given by the matrix
$$
\pi=\Matrr{2}{2}{1 & 1 & 0 & -a\\
            0 & 0 & 1 & 1}.
$$
The associated rays of the fan $\Sigma_a$ are given by the matrix
$$
\rho=\Matrr{2}{2}{ 0 & -a & 1 & -1\\
                   1 & -1 & 0 &  0}.
$$
If we denote the four columns, that is the rays, by $\rho_1,\ldots,\rho_4$,
then the tempting subsets of $\Sigma_a(1)$ are 
just  $\emptyset$, $\Sigma_a(1)$, $\CR_1=\{\rho_1,\rho_2\}$, 
and $\CR_2=\{\rho_3,\rho_4\}$. 
The corresponding maculate regions are
\begin{align*}
  \CM_{\R}(\emptyset)&= \cone \big\langle (1,0), (0,1), (-a,1) \big\rangle  
          = \cone \big\langle (1,0),\, (-a,1)\big\rangle,\\
  \CM_{\R}(\Sigma_a(1))&=  (a-2, -2) + \cone \big\langle (-1,0), (a,-1)\big\rangle,\\
  \CM_{\R}(\CR_1)& = (-2,0) + \cone \big\langle (-1,0), (0,1), (-a, 1)\big\rangle 
                   = (-2,0) + \cone \big\langle (-1,0), (0,1)\big\rangle,\\
  \CM_{\R}(\CR_2)& = (a,-2) + \cone \big\langle (1,0), (0,-1)\big\rangle.
\end{align*}
The lattice points within the complement of the union of these four 
regions consist of the line $(*,-1)$ and, if $a\geq 1$, the two isolated points $(-1,0)$ and $(a-1, -2)$.
In the degenerate case of $a=0$, there is an additional line $(-1,*)$, see Figure~\ref{fig_Pic_of_P1xP1}.
Here, all immaculate divisors are really immaculate.
\end{example}

\subsection{Conditions on  presence or absence of temptations}
\label{sect_three_temptations}
In this section we describe straightforward criteria that imply 
that a given subset of rays is  tempting or it is nontempting.
The upshot is that, for all sets $\CR\subseteq\Sigma(1)$ covered by one of these claims,
one does not need to look at the topology of $V^{>}(\CR)=\supp V^{\geq}(\CR)\setminus\{0\}$.

\subsubsection{Monomials do not lead into temptation}

The first criterion is similar to the boundedness condition in \cite[Prop.~2]{hering_kuronya_payne_asymptic_cohomologies_toric}.

\begin{proposition}
\label{prop_the_first_temption}
    Suppose $X= \toric(\Sigma)$ is a complete toric variety and $\CR \subset\Sigma(1)$ is a tempting subset.
    Denote by $\rho^* \colon M_{\R} \to \R^{\Sigma(1)}$ the natural embedding 
of the principal torus invariant divisors into all torus invariant divisors. 
    Then 
    \[
      \rho^*(M_{\R}) \cap \left(\R^{\Sigma(1)\setminus\CR}_{\geq 0} \times \R^{\CR}_{\leq 0}\right) = \set{0}.
    \]
\end{proposition}
    
\begin{proof}
    Suppose on the contrary, that $ (\rho^*)^{-1}\left(\R^{\Sigma(1)\setminus\CR}_{\geq 0} \times \R^{\CR}_{\leq 0}\right)$ 
       is a positive dimensional cone $\tau\subset M_{\R}$.
    Consider the divisor $D= \sum_{\varrho \in \CR} -D_{\varrho}$.
    Since $\CR$ is tempting, the divisor has non-zero cohomologies in 
       degree $-m$ for all $m \in \tau\cap M$. 
    Thus, the cohomology groups $\bigoplus_{i=0}^{\dim X} \gH^i (D)$ are infinitely dimensional, a contradiction with the completeness of $X$. 
\end{proof}

\begin{example}
    Consider the Hirzebruch surface $\F_a$ as in Example~\ref{ex-Hirz}, and suppose $a>0$.
    Then out of $16$ subsets of $\set{\rho_1, \rho_2, \rho_3, \rho_4}$, only
six survive the test provided by Proposition~\ref{prop_the_first_temption}.
    Namely, these are the four tempting subsets as listed in Example~\ref{ex-Hirz}, 
       and $\set{\rho_4}$ and its complement $\set{\rho_1, \rho_2, \rho_3}$
having the property of the associated cone intersecting $M$ in just $\set{0}$.
\end{example}

\begin{example}\label{ex_hexagon_first_temptation}
    In the ``hexagon'' case (see Examples~\ref{ex_hexagon_introduce_coordinates} and~\ref{ex_hexagon_tempting_subsets}), 
      Proposition~\ref{prop_the_first_temption} shows that the following $18$ out of $64 = 2^6$ subsets of $\Sigma(1)$ are non-tempting:
    \begin{gather*}
      \set{ 0, 1 },
      \set{ 0, 5 },
      \set{ 1, 2 },
      \set{ 2, 3 },
      \set{ 3, 4 },
      \set{ 4, 5 },
      \set{ 0, 1, 2 },
      \set{ 0, 1, 5 },
      \set{ 0, 4, 5 },
      \set{ 1, 2, 3 },\\
      \set{ 2, 3, 4 },
      \set{ 3, 4, 5 },
      \set{ 0, 1, 2, 3 },
      \set{ 0, 1, 2, 5 },
      \set{ 0, 1, 4, 5 },
      \set{ 0, 3, 4, 5 },
      \set{ 1, 2, 3, 4 },
      \set{ 2, 3, 4, 5 }.
    \end{gather*}
\end{example}

\subsubsection{Faces are not tempting}
\label{2nd_temptation}
\begin{proposition}
\label{prop_2nd_temptation}
Suppose $X = \toric(\Sigma)$ is a complete toric variety and 
$\sigma \in \Sigma$ is any cone (or a proper subfan with strictly convex support).
Then the subsets $\CR=\sigma(1)\subset \Sigma(1)$ and $\Sigma(1)\setminus \CR$ 
are not tempting.
\end{proposition}
\begin{proof}
    The complex $V^>(\CR)$ is equal to the convex set $\sigma\setminus\set{0}$, hence it is contractible.
    By Alexander duality (see Remark~\ref{rem_Alexander_duality_for_tempting_subsets})  the complement is also not tempting.
\end{proof}

\begin{example}
    For the Hirzebruch surface $\F_a$, 
      only the four tempting subsets fail this test.
    All the other subsets are either faces or complements of faces.
\end{example}

\begin{example}\label{ex_hexagon_second_temptation}
    According to Proposition~\ref{prop_2nd_temptation}, in the ``hexagon'' case 
      (see Examples~\ref{ex_hexagon_introduce_coordinates}, \ref{ex_hexagon_tempting_subsets}), 
      the following $24$ subsets of $\Sigma(1)$ are non-tempting: all single element subsets $\set{i}$, all consecutive two elements subsets $\set{i, i+1}$,
      and their complements (which have either four or five elements), which are all faces or their complements.
    Moreover, considering also three consecutive elements  $\set{i, i +1, i+2}$ (which are rays of a subfan with a strictly convex support),
     we obtain $30$ subsets, which are all the non-tempting subsets of $\Sigma(1)$. 
    Alternatively, the three element subsets can be understood from Example~\ref{ex_hexagon_first_temptation}.
\end{example}

\subsubsection{Primitive collections delude}
\label{3rd_temptation}
A \emph{primitive collection} 
of a simplicial fan $\Sigma$ is a ``minimal non-face'', 
   that is, a subset of rays $\CR\subset \Sigma(1)$,
   such that the cone spanned by $\CR$ is not in $\Sigma$, 
   but the cone spanned by $\CR\setminus\set{\rho}$ is in $\Sigma$ for every $\rho \in \CR$.
More generally, a subset $\CR\subset \Sigma(1)$ of any fan is a primitive collection, if $\CR$ is not contained in any single cone of $\Sigma$, 
   but every proper subset is.
See \cite{picRank3}, \cite{cox_von_renesse_primitive_collections_and_toric_varieties} for more details and explanations why this notion is important 
   and relevant to projective toric varieties, see also Section~\ref{basicWellKnown}. 

\begin{proposition}
    \label{prop_3rd_temptation}
    Suppose $X = \toric(\Sigma)$ is a complete simplicial toric variety with 
no torus factors. Let $\CR \subset \Sigma(1)$ be either empty or 
a primitive collection. Then $\CR$ and its complement are tempting.
\end{proposition}
\begin{proof}
    If $\CR= \emptyset$ or $\CR= \Sigma(1)$, then the claim is clear, 
       so suppose $\CR$ is a primitive collection, that is, a subset which is does not generate a cone of $\Sigma$,
       but all its proper subsets do generate such cones.
    By Alexander duality it is enough to prove that $\CR=\set{\rho_1,\ldots,\rho_k}$ is tempting.
    Since every ray belongs to $\Sigma$, we have $k \geq 2$.
    We distinguish between two cases: either $\CR$ is linearily independent or not.
    
    If $\CR$ is linearily independent, then $\CV:=\spann_\R\CR$ is $k$-dimensional,
       and $\CR^+:=\sum_{j=1}^k \R_{\geq 0}\cdot \rho_j$ is a $k$-dimensional
       simplicial cone in $\CV$ which does not belong to $\Sigma$. On the other hand,
    its boundary $\partial\CR^+$ is a subcomplex of $\Sigma$; 
    it is exactly the complex $V^{\ge}(\CR)$ as in Section~\ref{RmacPic}.
    Thus, $|V^{\ge}(\CR)|\setminus\{0\}=|\partial\CR^+|\setminus\{0\}$ is homotopy equivalent to a sphere $S^{k-2}$.
    In particular, it is not $\kk$-acyclic.

    On the other hand, suppose $\CR$ is linearly dependent.
    Since $\CR$ is a primitive collection, all the cones generated by $\CR
\setminus\set{\rho_j}$ are necessarily simplicial.
    In particular, $\CV:=\spann_\R\CR$ is $(k-1)$-dimensional,
       and each 
$\CR \setminus\set{\rho_j}$ spans
a full-dimensional cone in $\CV$ that belongs to $\Sigma$.
    Thus, these cones generate $V^{\ge}(\CR)$, and this is a complete fan 
in $\CV$ which (up to $\R$-linear change of coordinates) looks like the 
$\PP^{k-1}$-fan in $\R^{k-1}$.
    Again, $V^>(\CR)=|V^{\geq}(\CR)|\setminus\{0\}$ is homotopy equivalent 
to $ S^{k-2}$, hence it is not $\kk$-acyclic.
\end{proof}

\begin{example}
    For the Hirzebruch surface $\F_a$, all tempting subsets 
are predicted by Proposition~\ref{prop_3rd_temptation}.
       That is all four of them are either empty, or $\Sigma(1)$, or a primitive collection.
\end{example}

\begin{example}
\label{ex_hexagon_third_temptation}
    Proposition~\ref{prop_3rd_temptation} applied to the hexagon example     
       (see Examples~\ref{ex_hexagon_introduce_coordinates}, \ref{ex_hexagon_tempting_subsets}),
       implies that the following $20$ subsets are tempting:
    \begin{gather*}
      \emptyset,
      \set{ 0, 2 },
      \set{ 0, 3 },
      \set{ 0, 4 },
      \set{ 1, 3 },
      \set{ 1, 4 },
      \set{ 1, 5 },
      \set{ 2, 4 },
      \set{ 2, 5 },
      \set{ 3, 5 },
      \set{ 0, 1, 2, 4 },
      \set{ 0, 1, 3, 4 },\\
      \set{ 0, 1, 3, 5 },
      \set{ 0, 2, 3, 4 },
      \set{ 0, 2, 3, 5 },
      \set{ 0, 2, 4, 5 },
      \set{ 1, 2, 3, 5 },
      \set{ 1, 2, 4, 5 },
      \set{ 1, 3, 4, 5 },
      \Sigma(1).
    \end{gather*}

\end{example}

\subsection{The cube}
\label{vertCube}

Throughout this subsection we will assume $X =\toric(\Sigma)$ is a complete 
and simplicial toric variety.

Let $\CR\subseteq\Sigma(1)$ be an arbitrary, not necessarily tempting
subset.
This gives rise to a vertex
$v(\CR):=-(0^{\Sigma(1)\setminus \CR}, 1^\CR)$ of the cube $W$ spanned by all points 
of $\R^{\Sigma(1)}$ with $0/-1$ coordinates. 
It is the only vertex of the polyhedral cone $\pmr$,
that is, the class of the corresponding divisor 
$D_{\CR}:=-\sum_{\rho \in \CR} D_{\rho}$ is the most prominent element of
the $\CR$-maculate region 
$\CM_{\R}(\CR):=\pi\big(\pmr\big)$
introduced in Definition~\ref{def-RmacPic}.

We have discussed in Proposition~\ref{prop-Rmac} that the temptation
of $\CR$ implies the maculacy of $D_{\CR}$.
In the following we will show in 
Theorem~\ref{cor-injMacCube} 
that for $[D_{\CR}]\in\Cl(X)$ this is the only source of disgrace.

\begin{lemma}
\label{prop-DIr-immac}
Suppose $X = \toric(\Sigma)$ is a complete simplicial toric variety, and 
$\CR\subseteq \Sigma(1)$ is an arbitrary subset.
Let $\dm\in M\setminus\{0\}$. 
Then the complex $V^>_{D_{\CR},\dm}$ (or $V_{D_{\CR},\dm}$)
from \eqref{equ_complex_for_cohomologies_Weil} in Subsection~\ref{toricCohom}
is contractible and non-empty.
\end{lemma}

\begin{proof}
We prove a slightly more general claim.
Let $\Sigma$ be a simplicial, complete fan, let $\dm\in M\setminus\{0\}$ and
$S\subseteq\Sigma(1)$ such that
$$
\{\rho\in\Sigma(1)\kst\langle \rho,\dm\rangle < 0\}
\subseteq S \subseteq
\set{\rho\in\Sigma(1)\kst\langle \rho,\dm\rangle \leq 0},
$$
then the pointed full subcomplex $\langle S\rangle\setminus\{0\}\subset\Sigma$
is contractible and non-empty: First, since the cohomology of
$\langle S\rangle\setminus\{0\}$ is stable under small perturbations
of the position of the rays from $\Sigma(1)$, we
may move those from $\Sigma(1)\cap \dm^\bot$
into the open halfspaces $(\dm<0)$ or $(\dm>0)$ 
depending on whether they belong to $S$ or if they do not, respectively. 
Hence, we may assume that $\Sigma(1)\cap \dm^\bot=\emptyset$, that is, that
$S=[\dm{<0}]\cap\Sigma(1)=[\dm{\leq 0}]\cap\Sigma(1)$.
But now, the proof follows from the fact that
$\langle S\rangle$ is a deformation retract of the $(\dm\leq 0)$ part
of the geometric realization of $\Sigma$,  that is, of
$(\suppSig) \cap (\dm\leq 0)$.
This is a closed half space which stays 
cohomologically trivial even after removing the origin. 
\end{proof}

\begin{remark}
\label{rem-zeroDiv}
Note that full subcomplex generated by $S=[\dm{<0}]\cap\Sigma(1)$ equals 
$V^{\geq}_{0,\dm}$ for the zero divisor. Its contractibility for $\dm\neq 0$
is reflected by the fact that the structure sheaf is almost immaculate.
\end{remark}

As an immeadiate corollary of Lemma~\ref{prop-DIr-immac} we obtain

\begin{theorem}
\label{cor-injMacCube}
Suppose $X =\toric(\Sigma)$ is a complete simplicial toric variety, and
$\CR \subseteq \Sigma(1)$ is a subset.
It gives rise to the divisor $D_{\CR}:=-\sum_{\rho\in\CR}D_\rho$. 
\begin{enumerate}
\item
\label{cor-injMacCubeA}
Let $D=\sum_{\rho\in\Sigma(1)}\lambda_{\rho}D_{\rho}$ be a Weil divisor
linearly equivalent to $D_{\CR}$.
\vspace{0.5ex}
Then either $D=D_{\CR}$, or 
$\CR_D:=\{\rho \in\Sigma(1) \mid \lambda_{\rho} < 0\}$ is not tempting.
\item
\label{cor-injMacCubeB}
If $D_{\CR}$ is not immaculate, then
$\CR$ is tempting,
$D_{\CR}$ is the only divisor from $\pms$ that maps to $[D_{\CR}]$ under
$\pi$, and $\CR$ is the only tempting set containing $[D_{\CR}]\in \Cl(X)$ in
its maculate region $\CM_{\R}(\CR)$.
\item
\label{cor-injMacCubeC}
Non-tempting $\CR$ lead to immaculate $[D_{\CR}]$,
and the class map $\pi$ is injective on the maculate vertices of the cube $W$.
\end{enumerate}
\end{theorem}

\begin{proof}
\ref{cor-injMacCubeA} If $D\neq D_{\CR}$, then there is an $\dm\in M\setminus\{0\}$ such
that $D_{\CR}=D-\div(x^{\dm})=D(-\dm)$. Thus, as in
the proof of Theorem~\ref{thm_toric_cohomology_in_terms_of_nef_polytopes},
we obtain $V_{D, 0}=V_{D_{\CR}(\dm),\,0}=V_{D_{\CR}, m}$. The latter is
cohomologically
trivial by Lemma~\ref{prop-DIr-immac}, and the first
is generated by $\CR_D$.

\ref{cor-injMacCubeB}
Since $[D_{\CR}]$ is maculate, there must be a tempting $\CR'$
such that $[D_{\CR}]\in\CM_{\Z}(\CR')$, that is there is a divisor $D$ with
$[D]=[D_{\CR}]$ and $\CR_D=\CR'$. 
By \ref{cor-injMacCubeA} this implies $D=D_{\CR}$, that is $\CR_D=\CR$. 

\ref{cor-injMacCubeC} This is just a reformulation of \ref{cor-injMacCubeB}.
\end{proof}

\begin{remark}
\label{rem-AlwaysImm}
Immaculate divisors on $X$ do always exist: 
By Proposition~\ref{prop_2nd_temptation},
there exist non-tempting subsets induced from cones.
Other possibilities to produce such subsets arise from 
Lemma~\ref{prop-DIr-immac} -- just take
$\CR:=\Sigma(1)\cap[\dm<0]$ for some $\dm\in M\setminus\{0\}$.
In any case, the corresponding vertex of the cube is immaculate by 
Theorem~\ref{cor-injMacCube}\ref{cor-injMacCubeC}.
\end{remark}
   
\begin{example}
We work with the Hirzebruch surface $X=\F_a$, and we follow the coordinates and notation of Example~\ref{ex-Hirz}.
The cube $W$ is $4$-dimensional, so that it has $16$ vertices.
They correspond to subsets $I\subseteq\{\rho_1,\rho_2,\rho_3,\rho_4\}$,
and only four of them are tempting. 
These tempting ones lead to the four  different images
$$
(0,0)=-\pi(0000),\; (0,-2)=-\pi(1100),\; 
(a,-2)=-\pi(0011),\;(-2+a,-2)=-\pi(1111)
$$
in $\Z^2=\Cl(\F_a)=\Cl(\F_a)$. The remaining twelve are 
$$
\begin{array}{ll}
(-1,0)=-\pi(1000)=-\pi(0100),& (-1+a,-2)=-\pi(1011)=-\pi(0111),
\\
(-1,-1)=-\pi(1010)=-\pi(1001), & (-1+a,-1)=-\pi(1001)=-\pi(0101)
\end{array}
$$ 
(appearing twice) and,
with single occurrence,
$$
(0,-1)=-\pi(0010),\, (a,-1)=-\pi(0001),\, 
(-2,-1)=-\pi(1110),\, (-2+a,-1)=-\pi(1101).
$$ 
Thus, the two isolated points of
the immaculate locus plus two points from the immaculate line 
could have been guessed from the non-injectivity of $\pi$ alone.
More precisely, it is caused by $(1,-1,0,0)\in\ker\pi$,
which corresponds to the second row of the matrix $\rho$
describing the fan.

\end{example}

\begin{example}
\label{ex_hexagon_cube}
   In the notation of  Examples~\ref{ex_hexagon_introduce_coordinates} and~\ref{ex_hexagon_tempting_subsets} (the hexagon) 
      the image of $[-1,0]$-cube in $\Pic(X)\otimes \R$ is a lattice polytope $P$ with 46 vertices and 54 lattice points.
   34 of the vertices come from the 34 tempting vertices of the cube. 
   The remaining 12 vertices of $P$ are images of 12 non-tempting vertices of the cube.
   The 8 lattice points in $P$ that are not vertices (including 2 interior points), 
   are images of the remaining 18 non-tempting vertices of the cube, each with a repetition 
      (the two interior lattice points appear three times each as images of the vertices of the cube, the other points appear twice each).
   In particular, the cube produces 20 immaculate line bundles on $X$.
\end{example}

\section{Toric manifolds with Picard rank two}
\label{PicTwo}

We commence this section with recalling a well-known fact about 
smoothness of toric varieties in terms of Gale duality. 
Then we study our first family of examples, that is smooth complete toric varieties of Picard rank $2$.
Such varieties are described in \cite{toricPicRankTwo}, 
and we can classify all the immaculate line bundles on them.

\subsection{Spotting smoothness via Gale duality}
\label{smoothGale}
While working with fans having only few generators, the Gale transform
becomes the essential tool to investigate their combinatorial structure.
We recall an argument showing that this instrument,
considered for abelian groups instead of vector spaces,
can spot smoothness, too.
Let
$$
\xymatrix@R=1ex@C=2em{
0 \ar[r] &
K \ar[r]^{\iota} &
\Z^n \ar[r]^{\rho} &
N \ar[r] &
0
}
$$
be an exact sequence of free abelian groups with $d:=\rank N$.
This situation gives rise to the Gale transform being just the dual
sequence
$$
\xymatrix@R=1ex@C=2em{
0 &
K^* \ar[l] &
\Z^{n*} \ar[l]_-{\iota^*} &
N^* \ar[l] &
0. \ar[l]
}
$$

Denote by $\Z^d\subseteq\Z^n$ and $\Z^{(n-d)*}\subseteq\Z^{n*}$ 
the orthogonal subgroups being generated by
$\{e_1,\ldots,e_d\}$ and $\{e^{d+1},\ldots,e^n\}$, respectively.

\begin{proposition}
\label{prop-spotSmooth}
The determinant of $\{\rho(e_1),\ldots,\rho(e_d)\}$ equals, maybe up to sign,
the determinant of $\{\iota^*(e^{d+1}),\ldots,\iota^*(e^{n})\}$.
\end{proposition}

\begin{proof}
Assuming that the restriction $\rho|_{\Z^d}:\Z^d\to N$ has a finite cokernel
$C$ (which is equivalent to $\rho|_{\Z^d}$ being injective or to
$\,\Q^d\stackrel{\rho}{\to} N\otimes\Q$ being an isomorphism),
we obtain 
$$
\xymatrix@R=2ex@C=2em{
&& 0 \ar[d]& 0 \ar[d]
\\
&& \Z^d \ar@{=}[r] \ar[d] & 
\Z^d \ar[d]^-{\rho}
\\
0 \ar[r] &
K \ar[r]^{\iota} \ar@{=}[d]&
\Z^n \ar[r]^{\rho} \ar[d]&
N \ar[r] \ar[d]&
0
\\
0 \ar[r] &
K \ar[r]^(0.4){\ko{\iota}} &
\Z^{n-d} \ar[r]^{\rho} \ar[d]&
C \ar[r] \ar[d]&
0
\\
&& 0 & 0
}
$$
Dualizing the bottom row yields
$\coker\big(\Z^{(n-d)*}\stackrel{\iota^*}{\to} K^*\big)=\gExt^1_\Z(C,\Z)$.
That is, the cokernels of $\Z^d$ in $N$ and of
$\Z^{(n-d)*}$ in $K^*$ have the same order.
\end{proof}

\subsection{Immaculate locus for Picard rank two}
\label{picTwo}
After illustrating the general method for classifiying immaculate line bundles in Section~\ref{avoidTemptations}, 
we completely describe the immaculate loci in a specific case.
Explicitly, this section is devoted to smooth (complete) toric varieties of Picard rank $2$.

Investigating Gale duals leads to the well-known classification
of the combinatorial type of $d$-dimensional, simple, convex polytopes
with $d+2$ vertices -- they are
$(\triangle^{\ell_1-1}\times\triangle^{d-{\ell_1}+1})^\vee$ for some ${\ell_1}=2,\ldots, d$,
where $\triangle^r$ means the $r$-dimensional simplex and $(\ldots)^\vee$
denotes the dual of a polytope.
This is a special case of the situation we will meet 
in Subsection~\ref{splittingFans}.

Explicitly, in \cite[Thm~1]{toricPicRankTwo}, this classification was refined to
find all complete smooth $d$-dimensional fans with $d+2$ rays,
that is, all smooth complete toric varieties with Picard rank two. 
They are parametrized by the following data:
\begin{enumerate}
\item
a decomposition $d+2=\ell_1+\ell_2$ with $\ell_1,\ell_2\geq 2$ and
\item
a choice of non-positive integers $0= c^1\geq\ldots\geq c^{\ell_2}$
which are jointly denoted by $\kuc\in\Z^{\ell_2}_{\leq 0}$.
\end{enumerate}
These data provide the $2\times(\ell_1 + \ell_2)$-matrix
$$
\Matrr{3}{4}{1 & \ldots & 1 & 0 &  c^2 &\ldots & c^{\ell_2}\\
            0  & \ldots & 0 & 1 & 1 & \ldots & 1}
$$
encoding $\pi\colon (\Z^{(d+2)})^* \surj \Z^2 = \Cl X$
  (compare with Example~\ref{ex-Hirz}, where we had $a=-c^2$). 
That is, the rays of the associated fan $\Sigma_{\kuc}$ are
$u_i=\rho(e_i)$ and $v_j=\rho(f_j)$ in
$$
N:=\Z^{\ell_1+\ell_2}\big/
\big(\Z\cdot (\kue,\kuc) + \Z\cdot (\kun,\kue)\big)
\hspace{0.2em}\cong\Z^d
$$
where $\{e_1,\ldots,e_{\ell_1},f_1,\ldots,f_{\ell_2}\}$ denotes the canonical basis
in $\Z^{d+2}=\Z^{\ell_1+\ell_2}$ and $\rho:\Z^{d+2}\to N$
is the canonical projection. 
The fan structure is easy -- the
$d$-dimensional cones are $\sigma_{ij}$ which are generated 
by $\Sigma_{\kuc}(1)\setminus\{u_i,v_j\}$
($i=1,\ldots,\ell_1$, $j=1,\ldots,\ell_2$). 
That is, $\#\Sigma_\kuc(d)=\ell_1\ell_2$.
Comparing with (\ref{smoothGale}), 
one sees that the corresponding
cross-frontier $(2\times 2$)-minors of the above matrix are
$\,\det\matc{2}{1 & c^i \\ 0 & 1}=1$, that is,
by Proposition~\ref{prop-spotSmooth}, the cones
$\sigma_{ij}$ are indeed smooth.
We will denote $\koc:=\sum_{\nu=1}^{\ell_2} c^{\nu}$.
Note that in \cite[Thm~2]{toricPicRankTwo} it is shown that 
$X_{\kuc}$ is Fano if and only if $- \koc \leq   \ell_1 -1$.

\begin{figure}[hbt]
   \begin{center}
   \begin{tikzpicture}[scale=0.7]
     \fill[color=oliwkowy!35] (-5.3,-1) -- (-3,-1) -- (-3,2)--(0,-1)--(5.3, -1)--(5.3, -2)--(1,-2)--(1,-5)-- (-2,-2)--(-5.3, -2) --cycle; 		 
     \draw[color=oliwkowy!40] (-5.3,-5.3) grid (5.3,5.3); 
     \fill[pattern color=black!80, pattern=north west lines] (0,0) -- (5.3,0) -- (5.3,5.3) -- (-5.3,5.3) -- cycle;              
     \fill[pattern color=mocnyfiolet!80, pattern=north east lines] (-4,0) -- (-4,5.3) -- (-5.3,5.3) --(-5.3,0) -- cycle;   
     \fill[pattern color=ciemnazielen!80, pattern=north east lines] (2,-3) -- (5.3,-3) -- (5.3,-5.3) --(2,-5.3) -- cycle;   
     \fill[pattern color=jasnyfiolet, pattern=north west lines] (-2,-3) -- (0.3,-5.3) -- (-5.3,-5.3) --(-5.3,-3) -- cycle;   
     
     \draw[thick,  color=black] (5.3,0)-- (0,0) -- (-5.3,5.3) (5.6,2) node[anchor=east, rotate=270] {$\Eff$-cone}(0.1,-0.2) node{\scriptsize{$0$}} ;     
     \draw[thick,  color=mocnyfiolet] (-5.3,0) -- (-4,0) -- (-4,5.3) (-5.6,3.7) node[anchor=east,  rotate=90] {$\gH^3$-cone};  
     \draw[thick,  color=ciemnazielen] (2,-5.3) -- (2,-3) -- (5.3,-3)  (5.6,-5.3) node[anchor=east, rotate=270] {$\gH^2$-cone} ;					 
     \draw[thick,  color=jasnyfiolet!150] (-5.3,-3) -- (-2,-3) -- (0.3,-5.3) (-5.6,-3) node[anchor=east, rotate=90] {$\gH^5$-cone} (-2.1,-2.8) node{\scriptsize{$K_X$}}; 					 
     
     \draw[thick,  color=oliwkowy] (-5.3,-1) -- (-3,-1) -- (-3,2)--(0,-1)--(5.3, -1) (5.3, -2)--(1,-2)--(1,-5)-- (-2,-2)--(-5.3, -2)
        (-3,-0.9) node[anchor=north]{immaculate locus};  
   \end{tikzpicture}
   \end{center}
   \caption[The Picard lattice and immaculate locus of a smooth projective toric $5$-fold $X$ with $\Cl X =\Z^2$ and the matrix $\pi$.]{The Picard lattice and immaculate locus of a smooth projective toric $5$-fold $X$ with 
            $\Cl X= \Z^2$ and the matrix 
  $\pi = \Matrr{4}{3}{
               1 & 1 & 1 & 1 &   0 &-1 &-1\\
               0 & 0 & 0 & 0 &   1 & 1 & 1}$,
that is, $\kuc=\threedim{0}{-1}{-1}$.}
    \label{fig_Pic_rank_2}
\end{figure}

\begin{theorem}\label{thm_immaculates_for_Pic_rank_2}
   Suppose $X =\toric(\Sigma_{\kuc})$ is a smooth complete toric variety of Picard rank $2$.
   Then $\Imm_{\Z}(X) = \Imm_{\R}(X)$. 
   Moreover, the line bundle represented by $(x,y) \in \Z^2 = \Cl X$
   is immaculate if and only if one of the following holds:
   \begin{itemize}
     \item  $-\ell_2 < y < 0$ or
     \item  $y \ge 0$  and $ - \ell_1 <  x < c^{\ell_2} y $ or
     \item  $y \leq -\ell_2$ and 
            $0> x+\koc> c^{\ell_2}(y+\ell_2) -\ell_1$.
   \end{itemize}
\end{theorem}

Note that the second and the third case in the theorem are Serre dual to one another, while the first item is self-dual.
The first item is a bunch of (horizontal) affine lines of the same type as in 
Theorem~\ref{thm_line_of_immaculates_on_toric_varieties}, 
corresponding to $p$-immaculate line bundles $\fromto{(0, -1)}{(0,-\ell_2+1)}$, 
where $p \colon X \to \PP^{\ell_1-1}$ is the natural projection. 
If $\kuc=\kun$, that is, if $X\simeq \PP^{\ell_1-1}\times \PP^{\ell_2-1}$, 
then the divisors appearing in the second and third item 
   form parts of the lines  corresponding to the other projection $X\to \PP^{\ell_2-1}$.
Otherwise, $c^{\ell_2} < 0$, and there are only finitely many line bundles in the second and third items 
   (the inequalities define triangles). 
The special case of $\ell_1=\ell_2=2$ is illustrated on Figure~\ref{fig_Pic_of_F1}, and another case of a $5$-fold is on Figure~\ref{fig_Pic_rank_2}.
Points of the form $\fromto{(-1,0)}{(-\ell_1+1,0)}$ are always contained in
the second item (independently of $\kuc$).
Later, in the more general setup of Section~\ref{immSplit},
these points, together with the lines from the first item, 
will form the ``generating seeds'' in the sense of Definition~\ref{def-posNuns}.

\begin{proof}[Proof of Theorem~\ref{thm_immaculates_for_Pic_rank_2}]
The only tempting subsets of $\Sigma_{\kuc}(1)$ are $\emptyset$, $U = \setfromto{u_1}{u_{\ell_1}}$,
      $V = \setfromto{v_1}{v_{\ell_2}}$ and $\Sigma_{\kuc}(1) = U \sqcup V$.
We proceed along the lines of Example~\ref{ex-Hirz} 
   and calculate the maculate loci:
\begin{align*}
   \CM_{\R}(\emptyset)&= \cone \big\langle (1,0),\, (c^{\ell_2},1)\big\rangle, \\
   \CM_{\R}(\Sigma_{\kuc}(1))&= (-\koc -\ell_1, -\ell_2) + \cone \big\langle (-1,0),\, (-c^{\ell_2},-1)\big\rangle, \\
   \CM_{\R}(U)&= (-\ell_1, 0) + \cone \big\langle (-1,0),\, (0,1)\big\rangle, \\
   \CM_{\R}(V)&= (-\koc, -\ell_2) + \cone \big\langle (1,0),\, (0,-1)\big\rangle.
\end{align*}
Note that for every maculate $\CR \subset \Sigma_{\kuc}(1)$, 
the tail cone in the above locus is smooth and the primitive generators 
of rays are all in the image of the set $\pms$.
Thus, the map $\pms\to \CM_{\R}(\CR) \cap \Cl(X)$
   is surjective, i.e.\ $\CM_{\Z}(\CR) = \CM_{\R}(\CR) \cap \Cl(X)$.
It follows that $\Imm_{\Z}(X) = \Imm_{\R}(X)$ 
and the explicit description of the immaculate locus follows by an 
explicit calculation of the inequalities of the cones above, 
and by taking the complement in  $\Cl(X)$.
\end{proof}

\begin{proposition}
   Suppose as above that $X =\toric(\Sigma_{\kuc})$ is a smooth complete toric variety of Picard rank $2$.
   If $L$ is a line bundle on $X$ such that $\gH^i(X,L) \ne 0$, 
      then $i \in \set{0, \ell_1 -1, \ell_2 -1, \dim X}$. 
\end{proposition}
\begin{proof}
As in the previous proof,
the only tempting subsets are $\emptyset$, $U$, $V$ and $U \sqcup V$.
  Each of them leads to line bundles with nontrivial cohomology in one of the degrees: $\dim X$, $\ell_2-1$, $\ell_1 -1$ or $0$, respectively.
  Thus no other $\gH^i$ can be non-zero. 
\end{proof}

\section{The immaculate locus for splitting fans}
\label{immSplit}

In this section we apply the theory of Section~\ref{avoidTemptations} 
to the case of splitting fans and calculate 
the essential part of the immaculate locus of line bundles in this setup.
Let $X =\toric(\Sigma)$ be a smooth complete  toric variety.
Recall from Subsection~\ref{3rd_temptation} that a primitive collection 
of a (smooth, hence simplicial) fan $\Sigma$ is another word for a 
``minimal non-face''.
We say $\Sigma$ is a \emph{splitting fan}, if the primitive collections of $\Sigma$ are pairwise disjoint.
This is equivalent to an existence of a chain $\Sigma=\Sigma_k,\ldots,\Sigma_1$ of fans such that $\toric(\Sigma_1)=\PP^n$ 
  and  $\toric(\Sigma_{i+1})\to\toric(\Sigma_i)$ is a toric split bundle, that is a projectivisation of a direct sum of toric line bundles 
  (see \cite[Cor.~4.4]{picRank3}).
In particular, all such $X$ are projective.
Note that every smooth complete toric variety with Picard rank two satisfies
this property with $k=2$, see Subsection~\ref{picTwo}.

\subsection{Primitive relations}
\label{basicWellKnown}
In this subsection we recall the notion of the primitive relation associated to a primitive collection
   and express all such relations for a splitting fan.
We also give a lower bound on the number of primitive collections and characterise the splitting fans 
   as those smooth complete fans that have the least possible primitive collections, with respect to that lower bound.
Having in mind the application to splitting fans, which are smooth by definition, we restrict our presentation of primitive relations to the smooth case, following \cite{picRank3}. 
See \cite[\S1.3]{cox_von_renesse_primitive_collections_and_toric_varieties} for a more general treatment.

Let $\Sigma$ be a fan of a smooth complete toric variety $X$.
Recall that $\#\Sigma(1) = \rank \Cl(X) + \dim X$.
For every primitive collection $\CP\subseteq\Sigma(1)$ we denote
$e_{\CP}:=\sum_{\rho\in \CP} e_{\rho}$, where the $e_{\rho} \in \Z^{\Sigma(1)}$ 
is the basis element corresponding to $\rho$, 
which under the natural map $\Z^{\Sigma(1)} \to N$ is mapped to the 
primitive generator of the corresponding ray.
We denote by $\sigma(\CP)$ (called the ``focus of $\CP$'') the unique cone 
$\sigma\in\Sigma$ such that the image of $e_\CP$ in $N$ is contained in 
$\innt\sigma\subset N_\R$.
It leads to a unique element $f(\CP)\in\Z_{\geq 1}^{\sigma(\CP)(1)}$ 
with $e_\CP-f(\CP)\in \ker (\Z^{\Sigma(1)} \to N)$.
(Here, by convention, $\Z_{\geq 1}^{\emptyset} = \set{0}$.)
The expression $e_\CP-f(\CP)$ is called the \emph{primitive relation} 
associated to $\CP$. As an element of $\Cl(X)^*$, it represents a 
class of 1-cycles.

In \cite[Prop.~3.1]{picRank3} it is shown that 
$\CP\cap\sigma(\CP)=\emptyset$, that is the elements of $\CP$ are not among the
generators of $\sigma(\CP)$. Moreover, if $\Sigma$ is projective,
then there exists a primitive collection $\CP$ with $\sigma(\CP)=0$,
see \cite[Prop.~3.2 and Thm~4.3]{picRank3}.

\begin{proposition}
\label{prop-rhoPrimSplit}
Let $X=\toric(\Sigma)$ be a complete  toric variety.
Then every ray of $\Sigma$ is contained in some primitive collection.
If $X$ is in addition $\Q$-factorial, then the number of primitive collections is at least the rank of $\Cl(X)$.
Moreover, if $X$ is smooth, then equality holds if and only if $\Sigma$ is a splitting fan.
\end{proposition}

\begin{proof}
Let $\{\CP_1,\ldots,\ldots\CP_k\}$ be the primitive collections of $\Sigma$.
For each $\rho\in\Sigma(1)$ there exists  a cone $\tau\in\Sigma$ such that 
$\tau(1)\cup\{\rho\}$ is not contained in any cone of $\Sigma$
(otherwise, $\Sigma$ would not be complete). 
Thus $\rho$ is also contained in a minimal set of this type. 
This proves  $\bigcup_{i=1}^k\CP_i=\Sigma(1)$ as claimed.

Let $X$ be $\Q$-factorial. For each $i=1,\ldots,k$ 
we choose a $\rho_i\in\CP_i$. Then,
$\Sigma(1)\setminus\setfromto{\rho_1}{\rho_k}$ does not contain any
of the primitive collections, hence it generates a cone in $\Sigma$.
Thus, 
$$
\textstyle
(\dim X+\rank \Cl(X))-k= \#\Sigma(1)-k\leq \#\left(\Sigma(1)\setminus\{\rho_1,\ldots,\rho_k\}\right)\leq \dim X.
$$
If there is some overlap, say $\CP_i\cap\CP_j\neq\emptyset$, then we might
choose $\rho_i=\rho_j$, thus the above inequalities 
 even yield $\rank \Cl(X)\leq k-1$.
On the other hand, if all $\CP_i$ are pairwise disjoint,
then we know that the facets of $\Sigma$ look like
$\disjcup_{i=1}^k(\CP_i\setminus\{\rho_i\})$, in particular,  
$\#\Sigma(1)-k=\dim X$.
\end{proof}

\subsection{Temptation for splitting fans}
\label{splittingFans}

Here we assume $X$ is a smooth toric projective variety of dimension $d$ whose fan $\Sigma$ is a splitting fan.
We will first identify  all of the  tempting subsets $\CR\subseteq\Sigma(1)$.
Later in Section~\ref{picMapPi} we will investigate 
the associated $\pi$-images $\CM_{\R}(\CR)$ or $\CM_\Z(\CR)$
as introduced in Definition~\ref{def-RmacPic}.

Let $\Sigma(1)=\CP_1\sqcup\ldots\sqcup\CP_k$ be the decomposition
into primitive collections of lengths $\ell_1,\ldots,\ell_k\geq 2$,
respectively. Thus, maximal cones $\sigma\in\Sigma(d)$ correspond to maximal
subsets of $\Sigma(1)$ not containing any entire set $\CP_i$
($i=1,\ldots,k$). 
This establishes a bijection
$$
\CP_1\times\ldots\times\CP_k
\stackrel{\sim}{\longrightarrow}
\Sigma(d),
\hspace{1em}
(p_1,\ldots,p_k)\mapsto \Sigma(1)\setminus\{p_1,\ldots,p_k\}.
$$
In particular, $\Sigma$ is combinatorially equivalent to the normal fan
of 
\[
  \square = \square(\setfromto{1}{k}):=
\triangle^{\ell_1-1}\times\ldots\times\triangle^{\ell_k-1}
\]
where $\triangle^{\ell-1}$ denotes the $(\ell-1)$-dimensional
simplex with $\ell$ vertices.
In particular, we know that
$\#\Sigma(1)=\sum_{i=1}^k\ell_i$,
$\,d=\sum_{i=1}^k (\ell_i-1)$,
$\,\#\Sigma(d)=\prod_{i=1}^k\ell_i$, and,
in compliance with Proposition~\ref{prop-rhoPrimSplit},
$\rank(\Cl X)=\#\Sigma(1)-d=k$.

Now, the essential point is that the temptation of a subset
$\CR\subseteq\Sigma(1)$  depends only on the combinatorial structure 
of $\Sigma$. 
The finer structure, the true shape of the fan reflected by the maps 
$\rho:\Z^{\Sigma(1)} \to N$ or 
$\pi\colon\Z^{\Sigma(1)} \to \Cl(X)$, does matter only 
for the second step of turning the tempting sets $\CR$ into the 
maculate regions $\CM_{\R}(\CR)$.

\begin{lemma}
\label{lem-temptingSplit}
If $\Sigma$ is a splitting fan with the decomposition
$\Sigma(1)=\CP_1\sqcup\ldots\sqcup\CP_k$ into primitive collections $\CP_i$,
then the tempting subsets of $\Sigma(1)$ are
$\CR(J):=\bigcup_{j\in J} \CP_j$ with $J\subseteq\{1,\ldots,k\}$.
\end{lemma}

\begin{proof}
Instead of the complex 
$(\supp V^{\ge}(\CR)) \setminus\{0\}\subseteq \suppSig \setminus\{0\}\sim S^{d-1}$ we
consider its dual version $G(\CR)$ built as the union of all (closed) facets
$G(\rho)<\square$ dual to $\rho\in\CR$.
Clearly, 
$\supp V^{\ge}(\CR) \setminus\{0\}$ is homotopy equivalent to $G(\CR)$, 
  thus one is $\kk$-acyclic if  and only if the other is.
A subset $J\subseteq\{1,\ldots,k\}$ defines a splitting
$\square=\square(J)\times\square( \setfromto{1}{k} \setminus J)$ and accordingly  we have
$G(\CR(J))=\partial\square(J)\times\square( \setfromto{1}{k} \setminus J)$, 
which is not $\kk$-acyclic. Thus every set $\CR(J)$ is tempting.

On the other hand, suppose that, for some $j$, the set $\CR\subset \Sigma(1)$ 
does not contain the whole $P_j$, but at least one element of $P_j$.
We claim $G(\CR)$ is contractible. 
Indeed, if, without loss of generality, $j=1$, then we split off
$\square=\triangle^{\ell_1-1}\times\square (2,\ldots,k)$.
Let $f < \triangle^{\ell_1-1}$ be the (non-empty) face corresponding to the 
subset $\CR \cap P_1 \subset P_1$ 
  and pick a standard strong deformation retract $\triangle^{\ell_1-1} \to  f$.
Then $G(\CR)$ can be retracted to the contractible  
$f\times\square(2,\ldots,k)$ 
by gluing together the retractions of the contributing faces: 
each of the faces is either of the form $\triangle^{\ell_1-1}\times F$ 
for some face $F < \square(2,\ldots,k)$ 
or of the form $F'\times\square(2,\ldots,k)$ for a face 
$F' <\triangle^{\ell_1-1}$ containing $f$.
Note that the image of any face of the latter type is just all of 
$f\times\square(2,\ldots,k)$, hence $G(\CR)$ is contractible.
\end{proof}

\begin{remark}
The $2^k$ different sets $J\subseteq \{1,\ldots,k\}$
yield $2^k$ tempting sets $\CR(J)$, hence $2^k$ maculate regions
$\CM_{\R}(\CR(J))$ within the $k$-dimensional space $\Cl(X)\otimes\R\cong\R^k$.
This looks a little like the structure of $2^k$ octants in this space,
   but we will see that typically the octants are ``leaning'', and they may intersect as illustrated on Figures~\ref{fig_Pic_of_F1} and~\ref{fig_Pic_rank_2}.
\end{remark}

\begin{proposition}
Let \(X = \toric(\Sigma)\) with \(\Sigma\) a splitting fan with \(k\) primitive collections, 
and \(L\) be a line bundle on \(X\) such that \(\gH^i(X,L) \ne 0\), 
then \(i \in \set{\sum_{j \in J} (\ell_j -1)}_{J \subset \set{1, \dots, k}}\). 
\end{proposition}

\begin{proof}
In the previous proof we have seen that \(\CR(J)\) leads to the non-$\kk$-acyclic
$G(\CR(J))=\partial\square(J)\times\square( \setfromto{1}{k} \setminus J)$.
For cohomological considerations we can now focus on the first part 
\( \partial\square(J) = \partial \left( \prod_{j \in J} \triangle^{\ell_j-1} \right)\).
Thus we have the boundary of a polytope of dimension \(\sum_{j \in J} (\ell_j -1)\),
so \(\CR(J)\) is homotopy equivalent to a (\(\sum_{j \in J} (\ell_j -1)-1\))-dimensional
sphere.
\end{proof}

\subsection{The refined structure of the fan and the class map}
\label{picMapPi}

We have treated the combinatorial structure of the splitting fan $\Sigma$ 
in the previous section.
Here we concentrate on the more refined information, specifically,
we focus on the detailed structure of the 
class map $\pi\colon \Z^{\Sigma(1)} \to \Cl(X)$, 
where $X = \toric(\Sigma)$.

Write $\Sigma(1)=\bigsqcup_{i=1}^k\CP_i$ the decomposition of the rays into 
the disjoint sets of primitive collections.
In \cite[Corollary 4.4]{picRank3}, Batyrev has proved that $X$
can be obtained via a sequence of projectivations of decomposable bundles.
Within the fan language this means that we can assume that there is a 
sequence of fans $\Sigma=\Sigma_k,\ldots,\Sigma_1,\Sigma_0=0$
in abelian groups $N=N_k\surj\ldots\surj N_1\surj N_0=0$
such that the focus $\sigma(\CP_j)=0$ in $N_j$ and 
$N_{j-1}=N_j/\spann \CP_j$.
The fans $\Sigma_j$ in $N_j$ are splitting with 
$\Sigma_j(1)=\disjcup_{i=1}^j\CP_i$, and 
they admit subfans $\wt{\Sigma}_{j-1}\subset\Sigma_j$ such that
$\psi_j:N_j\surj N_{j-1}$ induces an isomorphism
$\wt{\Sigma}_{j-1}\stackrel{\sim}{\to}\Sigma_{j-1}$ 
(piecewise linear on the geometric realizations) and 
$\Sigma_j$ consist of the sums of cones from $\wt{\Sigma}_{j-1}$ and proper
subsets of $\CP_j$.

With $\ell_i=\#\CP_i$,
this explicit structure of $\Sigma$ can be translated into the fact that
$\pi$ is a triangular block matrix
\begin{equation}\label{equ_pi_matrix_for_a_splitting_fan}
\pi = \Matc{4}{\kue & \kuc_{12} & \ldots & \kuc_{1 k}\\
              \kun& \kue & \ldots & \kuc_{2 k}\\
               \vdots & \vdots & \ddots & \vdots\\
              \kun&\kun& \ldots & \kue}
\end{equation}
with $k$ rows and $k$ blocks of columns of thickness $\ell_j$ 
($j=1,\ldots,k$). While $\kue$ denotes $(1,1,\ldots,1)$ with
$\ell_j$ entries, we have $\kuc_{ij}\in\Z_{\leq 0}^{\ell_j}$.
If needed, its entries will be denoted by $c_{ij}^\nu\in\Z_{\leq 0}$ 
$(i<j,\;\nu=1,\ldots,\ell_j)$.
Every row encodes a primitive relation, hence each $\kuc_{ij}$ has at least one
zero entry (the support of $\kuc_{i\kbb}$ is supposed to be a face, that is to not contain any full $\CP_j$).

\begin{proposition}
\label{prop-matrixToFan}
Each matrix $\pi$ in  a triangular block form as in 
\eqref{equ_pi_matrix_for_a_splitting_fan} with 
$\kuc_{ij}\in\Z^{\ell_j}$  and $\ell_j \ge 2$ for all $j$ 
   gives rise to a smooth splitting fan of dimension $d:=\sum_{i=1}^k(\ell_i-1)$.
\end{proposition}

\begin{proof}
We argue inductively on the number of rows $k$ (and, at the same time, the number of blocks of columns).
For $k=1$ there is nothing to prove, 
  so assume that $\Sigma_{k-1}$ is a splitting fan obtained from $\pi'$, a matrix with last row and last block of columns removed from $\pi$.
For each $\nu\in \setfromto{1}{\ell_k}$ let $\CL_{\nu}$ be the line bundle on $X_{k-1} = \toric(\Sigma_{k-1})$ corresponding to the point $c^{\nu}_{*k}$ in $\Cl X_{k-1}$.
Then $X=\PP(\bigoplus_{\nu=1}^{\ell_k} \CL_{\nu})$ is the desired 
smooth toric variety.
\end{proof}

Note that the smootheness of the variety $\toric(\Sigma)$
associated to the matrix~$\pi$
can also be derived directly from the method of Subsection~\ref{smoothGale}.
The co-facets of the fan $\Sigma$ give rise to choosing 
one column of $\pi$ in every block. But this yields an upper triangular matrix
with only $1$ as the diagonal entries. Hence, the determinant equals $1$, too.

\begin{example}
\label{ex-CFQ}
A simple case to have in mind is $k=2$.
The matrix of $\pi$ is
$$
\Matrc{3}{4}{1 & \ldots & 1 & 0 & c^2 & \ldots & c^{\ell_2}\\
            0 & \ldots & 0 & 1 & 1 & \ldots & 1}
=
\Matc{2}{\kue & \kuc\\ \kun & \kue}
$$
It covers the case of Hirzebruch surfaces. 
In Subsection~\ref{picTwo} we  have discussed the immaculate locus 
of this matrix in detail.
\end{example}

\begin{example}\label{ex_immaculate_but_not_really_immaculate}
   Consider the following smooth projective three dimensional toric variety $X=\toric(\Sigma) = \PP (\cO_Y(-2,0) \oplus \cO_Y(0,-2))$,
      where $Y = \PP^1 \times \PP^1$, and $\cO_Y(i,j): = \cO_{\PP^1}(i)\boxtimes\cO_{\PP^1}(j)$.
   Then the fan $\Sigma$ is a splitting fan with matrix 
   \[
   \pi = \Matc{6}{1 & 1 & 0 & 0 & -2 &  0\\
                  0 & 0 & 1 & 1 &  0 & -2\\ 
                  0 & 0 & 0 & 0 &  1 &  1}.
   \]
    The line bundle represented by 
    \[
      \pi\left( (0,0,0,0,\tfrac{1}{2}, \tfrac{1}{2}) \right)=(-1,-1,1)  \in \Cl(X)
    \]
     is immaculate but not really immaculate 
     (in the sense of Definition~\ref{def_really_immaculate}).
\end{example}

\subsection{Generating immaculate seeds}
\label{genNuns}
We fix a format $\kul:=(\ell_1,\ldots,\ell_k)$ of splitting fans,
that is a block format of the associated matrix $\pi$. We understand
$\kuc$, that is
the entries $\kuc_{ij}\in\Z_{\leq 0}^{\ell_j}$ of $\pi$, as coordinates of the 
``moduli space'' of splitting fans $\Sigma(\kul ,\kuc)$
of this fixed format $\kul $. All these fans share the 
same combinatorial type -- that of the normal fan of
$\square:=\triangle^{\ell_1-1}\times\ldots\times\triangle^{\ell_k-1}$,
see Section~\ref{splittingFans}. Similarily, the associated toric varieties share
the same Picard group. Since we use the primitive relations for the rows
of $\pi$, we have even distinguished coordinates leading to a 
simultaneous identification
$\Cl\toric(\Sigma(\kul ,\kuc))=\Z^k$. This makes it possible to
compare the immaculate loci of different 
$\Sigma(\kul ,\kuc)$ sharing the same $\kul$.

Now, the basic idea is simple: For special $\kuc$, e.g.\ $\kuc=\kun$,
the immaculate locus is large -- but it becomes smaller for
growing $|\kuc|:=-\kuc$. 
Roughly speaking, we will show that this shrinking of 
the immaculate locus becomes stationary, and we are going to calculate
the limit. 

There is, however, a technical obstacle. 
The center of symmetry $K_X/2$ arising from Serre duality moves with $\kuc$.
Thus, it is not the whole immaculate locus that becomes stationary -- this works only for some 
generating seed.
That is, there is a certain subset of $\Z^k$ 
which is immaculate for all $\Sigma(\kul ,\kuc)$ and which generates 
(via some operations/reflections corresponding to successive Serre dualities) 
the full immaculate locus if $-\kuc$ is sufficiently large. 

\begin{definition}
\label{def-posNuns}
We call $\Nuns(\kul ):=\bigcup_{j=1}^k
\big(\Z^{j-1}\times\{-1,\ldots,-(\ell_j-1)\}\times \kun^{k-j}\big)$
the ``generating immaculate seed'' for $\kul $ in $\Z^k$.
\end{definition}

Recall the integral matrix expression 
$$
\pi = \pi(\kul ,\kuc)
             = \Matc{4}{\kue & \kuc_{12} & \ldots & \kuc_{1k}\\
              \kun& \kue & \ldots & \kuc_{2k}\\
               \vdots & \vdots & \ddots & \vdots\\
              \kun&\kun& \ldots & \kue}
$$
from Section~\ref{picMapPi} with non-positive entries within
the vectors $\kuc_{ij}$.
For fixed $i,j\in \set{1,\ldots,k}$ we set 
$\koc_{ij}:=\sum_{\nu=1}^{\ell_j}c_{ij}^\nu\in\Z_{\leq 0}$.
Moreover,  denote
$$
v_j:=(\koc_{1j},\ldots,\koc_{j-1,\,j},\ell_j,\kun)
\in(\Z^j\times\kun)\subseteq\Z^k.
$$
Depending on $\kuc$, we can now define the operator enlarging a given seed
in $\Z^k$:

\begin{definition}
\label{def-hullOpa}
For a given subset $G\subseteq\Z^k$ we define its \emph{$\kuc$-hull}
as the smallest set $\langle G \rangle_{\kuc}\supseteq G$ satisfying
the following recursive property:
If $a\in \langle G \rangle_{\kuc}\cap(\Z^j\times\kun^{k-j})$ 
for some $j=0,\ldots,k-1$, then so is the shift 
$a-v_{j+1}\in \langle G \rangle_{\kuc}$.

Note that 
$a-v_{j+1} \in \Z^j\times\set{-\ell_{j+1}} \times \kun^{k-j-1}$.
Hence, building $\langle G \rangle_{\kuc}$ out of $G$ can be done 
by enlarging $\langle G \rangle_{\kuc}$ successively with increasing $j$.
\end{definition}

\begin{remark}\label{rem_double_Serre_shift}
   Note that Serre duality replaces $a \in \Z^k$ with $-a -\sum_{i=1}^{k} v_i$.
   As we will see, we can also use Serre duality on the level of 
the smaller varieties $\toric(\Sigma_j)$.
   Thus the shift $a-v_{j+1}$ is obtained by a ``double Serre duality'':
   First we dualise in $\toric(\Sigma_j)$ obtaining $ -a-\sum_{i=1}^{j}v_{i}$, 
      and then we dualize in $\toric(\Sigma_{j+1})$ to get the shift 
      $-(-a - \sum_{i=1}^{j} v_i) - \sum_{i=1}^{j+1} v_i$ from Definition~\ref{def-hullOpa}.
\end{remark}

These definitions of $\Nuns(\kul )$ and the hull operations allow to
describe the locus of immaculate line bundles for ``general'' $\kuc$.
Recall the notions of maculate regions from Definition~\ref{def-RmacPic}, and immaculate loci from Definition~\ref{def_immaculate_loci}. 

\begin{theorem}\label{thm_immaculates_for_splitting_fan}
\label{th-immSplit} Fix $\kul $ and let $\kuc$ be a parameter leading to a matrix
$\pi=\pi(\kul ,\kuc)$ with the associated splitting fan 
$\Sigma=\Sigma(\kul ,\kuc)$.
Then:
\begin{enumerate}
 \item \label{item_nuns_are_really_immaculate}
       $\Nuns(\kul )\subseteq\Imm_{\RR}(\Sigma(\kul ,\kuc))$ for all
$\kuc$, that is the generating seeds are really immaculate.
 \item \label{item_Imm_is_closed_under_hull}
       Both the loci $\Imm_{\Z}(\Sigma)$ and $\Imm_{\R}(\Sigma)$ are closed under the $\kuc$-hull operation.
 \item \label{item_immaculate locus_for_general_c}
   For ``general'' $\kuc$, the immaculate loci are both equal to the minimal set satisfying the above.
   That is, $\Imm_{\Z}(\Sigma(\kul ,\kuc))=\Imm_{\R}(\Sigma(\kul ,\kuc))=\langle \Nuns(\kul ) \rangle_{\kuc}$.
\end{enumerate}
More precisely, a sufficient condition for ``general'' in \ref{item_immaculate locus_for_general_c}
   is that for each $j=1,\ldots,k-1$ the vector 
$\kuc_{j,j+1}\in\Z^{\ell_{j+1}}$ has at least two entries differing by more
than $\ell_j$.
\end{theorem}

\begin{proof}
   Consider the sequence of projections $X = X_k \to X_{k-1} \to X_{k-2} \to \dotsb \to X_1 = \PP^{\ell_1-1}$
     and the corresponding fans  $\Sigma= \Sigma_k, \Sigma_{k-1}, \Sigma_{k-2}, \dotsc, \Sigma_1$,
     such that $X_{j+1} = \PP(\cE_j)$, where $\cE_j$ is a split vector bundle over $X_j = \toric(\Sigma_j)$.
   We argue by induction on the Picard number $k$.
   If $k=1$, then $X\simeq \PP^{\ell_1-1}$, $\Nuns(\kul )=\Imm(\Sigma_1) = \setfromto{-1, -2}{-(\ell_1-1)} \subset \Z\simeq \Cl X$,
     $v_1= - \ell_1$, 
     and the shift in the definition of $\kuc$-hull operation is 
irrelevant since $0$ is not a generating seed.
   Thus there is nothing to prove in this case.

   So suppose that the statement holds for Picard number at most $k-1$ 
      and denote by $p$ the projection $p\colon X= \PP(\cE_{k-1}) \to X_{k-1}$.
   To present a description of the maculate regions, 
      we recognize the vectors $v_j$ (for $j=1,\ldots,k$)
      defined above as the building blocks of the $\pi$-images of the ``maculate vertices of the cube'' (see Section~\ref{vertCube}).
   The associated tail cones are built from the polyhedral cones
   $$
      C_j:=\langle (c_{1j}^\nu,\ldots,c_{j-1,\,j}^\nu,1,\kun)\kst
      \nu=1,\ldots,\ell_j\rangle\subseteq (\R^j\times\kun)\subseteq\R^k=\Cl(X)\otimes \R.
   $$
   Note that $v_j$ is the sum of the generators of $C_j$.
   The maculate regions are now parametrized by $J\subseteq\{1,\ldots,k\}$ and equal
   $$
    \textstyle
    \CM_{\R}(\CR_J)=\sum_{j\notin J} C_j + \sum_{j\in J} (-v_j-C_j).
   $$
   
We claim the elements 
$(*,\ldots,*,j)\in
\Z^{k-1}\times\{j\}\subset\Z^k$ with $j=-1,\ldots,-(\ell_k-1)$
are always really immaculate.
Indeed, the vectors $v_j$ and the cones $C_j$ for $j=1,\ldots,k-1$
have zero as their last entries. The last entry of $v_k$ is $\ell_k$,
and the last entries of the generators of $C_k$ are always $1$. 
Thus, any point in a maculate cone $\CM_{\R}(\CR_J)$ has either the last entry at least $0$ or at most $ -\ell_k$.

We turn our attention to the points of the form
$(*,\ldots,*,0)\in \Z^{k-1}\times\{0\}\subset\Z^k$.
We remark that those line bundles are exactly the pullbacks of line bundles on $X_{k-1}$.
Thus Corollary~\ref{cor_pullback_of_immaculate_on_toric_varieties} is relevant here, 
  if we restrict the attention to $p^*(\Imm_{\Z}(X_{k-1})) \subset \Imm_{\Z}(X)$.
Instead, our argument here is stronger, about the really immaculate locus.
If $k\in J$, then $\CM_{\R}(\CR_J)\cap (\Z^{k-1}\times\{0\})=\emptyset$.
Thus, those $J$ are never a source of maculacy for points with the last coordinate $0$.
On the other hand, 
if $k\notin J$, then $\CM_{\R}(\CR_J)\cap (\Z^{k-1}\times\{0\})=
\ko{\CM}_{\R}(\CR_{J})$ and
$\ko{\CM}_{\R}$ denotes the maculate region with respect to $\ko{\pi}$
resulting from $\pi$ by deleting the last row and the last block of columns.
In particular, $\ko{\pi}$ corresponds to the fan $\Sigma_{k-1}$.

Therefore, by the inductive assumption, all the generating seeds 
are really immaculate concluding the proof of 
\ref{item_nuns_are_really_immaculate}.
Moreover, the inductive assumption together with Alexander/Serre duality 
(see Remarks~\ref{rem_Alexander_duality_for_tempting_subsets}, \ref{rem_double_Serre_shift}, 
   and Corollary~\ref{cor_pullback_of_immaculate_on_toric_varieties}) 
  show that both loci $\Imm_{\Z}(X)$ and $\Imm_{\R}(X)$ are closed under $\kuc$-hull operation, proving \ref{item_Imm_is_closed_under_hull}.
The induction and the above discussion also show \ref{item_immaculate locus_for_general_c} for points of the form $(*,\ldots,*,j)$ for $-\ell_k < j \le 0$.
By Alexander/Serre duality (or the shift from the definition of $\kuc$-hull), 
   the case of $j=-\ell_k$ is also proved: if $a = (*,\ldots,*,-\ell_k)$, 
   then $a$ is (really) immaculate if and only if its shift $a - v_k \in \Z^{k-1} \times 0$ is (really) immaculate.

The duality also swaps the points of the form $(*,\ldots,*,\ge 1)\in \Z^{k-1}\times\Z_{\ge 1}\subset\Z^k$
   with those of the form $(*,\ldots,*,\leq -(\ell_k +1))$.
Therefore, to complete the proof of  \ref{item_immaculate locus_for_general_c} it remains to show that no line bundle
whose class in $\Cl(X)$ is $a=(*,\ldots,*,\geq 1)$ is immaculate.
We must show this under the assumptions that \ref{item_immaculate locus_for_general_c} holds for $X_{k-1}$ 
   and $\kuc_{k-1,k}$ has two entries differing by more than $\ell_{k-1}$,
   say, $|\kuc_{k-1, k}^1 - \kuc_{k-1, k}^2| > \ell_{k-1}$.

As before, $\CR_J$-maculacy can only go along with $k\notin J$.
Let $a^k$ be the last coordinate of $a$ and consider two vectors 
$b_\nu= a - a^k\cdot  (c_{1k}^\nu,\ldots,c_{k-1,\,k}^\nu,1) \in \Z^{k-1} \times 0$ for $\nu = 1$ or $2$.
The difference between the $(k-1)$-st coordinates of $b_1$ and $b_2$ is at least $\ell_{k-1}+1$.
Using \ref{item_immaculate locus_for_general_c} for $X_{k-1}$, since all immaculate line bundles for $X_{k-1}$ 
   have the last coordinate in  the set $\setfromto{-\ell_{k-1}}{0}$, at least one of $b_1$ or $b_2$ is not immaculate.
Say $b_1$ is not immaculate.
Let $J \subset \setfromto{1}{k-1}$ be the subset such that $b_1\in \overline{\CM}_{\Z}(\cR_{J})$.
Then $\CM_{\Z}(\cR_{J})$ is the sum of 
$\overline{\CM}_{\Z}(\cR_{J})$ and the monoid generated by 
$(c_{1k}^\nu,\ldots,c_{k-1,\,k}^\nu,1)$.
Therefore $a = b_1 + a^k \cdot (c_{1k}^1,\ldots,c_{k-1,\,k}^1,1)\in  
\CM_{\Z}(\cR_{J})$ and $a$ cannot be immaculate.
\end{proof}

We remark that for non-general values of $\kuc$ the conclusion of \ref{item_immaculate locus_for_general_c} needs not to hold, see for example Figure~\ref{fig_Pic_rank_2}.

\section{The immaculate locus for Picard rank 3}
\label{immPicThree}

In this section we finally make everything concrete
in the case of Picard rank $3$. 
We first review the classification of Batyrev, and then describe the tempting subsets of rays.
Finally, we list a lot of immaculate line bundles and prove (similarly to Theorem~\ref{thm_immaculates_for_splitting_fan}) 
  that for sufficiently general parameters the listed ones are all immaculate line bundles.

\subsection{Classification by Batyrev}
\label{classBatyrev}
In \cite{picRank3} a classification of smooth, projective toric varieties of Picard rank 
three is given by using its primitive collections.
See also \cite{cox_von_renesse_primitive_collections_and_toric_varieties}.

\begin{proposition}[{\cite[Thm~5.7]{picRank3}}]
If \(\Sigma\) is a complete, regular \(d\)-dimensional fan with \(d+3\) generators, then 
the number of primitive collections of its generators is equal to 3 or 5.
\end{proposition}

In the case that there are exactly three primitive collections the fan $\Sigma$ is a splitting fan by Proposition~\ref{prop-rhoPrimSplit}.
Thus the associated toric variety is
   isomorphic to a projectivisation of a decomposable bundle over a smooth toric variety of smaller dimension and Picard rank two. 
In particular, Theorem~\ref{thm_immaculates_for_splitting_fan} provides a 
valid description of the immaculate loci in this case.

Therefore, for the rest of this section we are going to assume that
  $X$ is a smooth variety of Picard rank $3$, which has exactly five primitive collections.
Following \cite{picRank3} we give a more precise description of the fan.
There is a decomposition of the rays \(\Sigma(1)\) into five disjoint subsets \(J_{\alpha}\) and the primitive collections are given by
\(J_{\alpha} \cup J_{\alpha +1 }\) for \(\alpha \in \quot{\bZ}{5 \bZ}\).

\begin{proposition}[{\cite[Thm~6.6]{picRank3}}]
\label{prop-picRank3Batyrev}
Let us denote \(\mathcal{J}_{\alpha} = J_{\alpha} \cup J_{\alpha +1 }\), where \(\alpha \in \quot{\bZ}{5 \bZ}\),
\[
J_0 = \{v_1, \dots ,v_{p_0}\},
J_1 = \{y_1, \dots ,y_{p_1}\},
J_2 = \{z_1, \dots ,z_{p_2}\},
J_3 = \{t_1, \dots ,t_{p_3}\},
J_4 = \{u_1, \dots ,u_{p_4}\},
\]
and \(p_0+\cdots+p_4 = d+3\). Then any complete regular \(d\)-dimensional fan \(\Sigma\)
with the set of generators \(\Sigma(1) = \bigcup J_{\alpha}\) and five primitive collections
\(\mathcal{J}_{\alpha}\) can be described up to a symmetry of the pentagon 
by the following
primitive relations with non-negative integral coefficients \(c_2, \dots , c_{p_2}\), 
\(b_1, \dots b_{p_3}\):
\begin{align*}
\sum_{i=1}^{p_0} v_i + \sum_{i=1}^{p_1} y_i - \sum_{i=2}^{p_2} c_i z_i - \sum_{i=1}^{p_3} (b_i+1) t_i &=0,\\
\sum_{i=1}^{p_1} y_i + \sum_{i=1}^{p_2} z_i - \sum_{i=1}^{p_4} u_i &=0,\\
\sum_{i=1}^{p_2} z_i + \sum_{i=1}^{p_3} t_i &=0,\\
\sum_{i=1}^{p_3} t_i + \sum_{i=1}^{p_4} u_i - \sum_{i=1}^{p_1} y_i &=0,\\
\sum_{i=1}^{p_4} u_i + \sum_{i=1}^{p_0} v_i - \sum_{i=2}^{p_2} c_i z_i - \sum_{i=1}^{p_3} b_i t_i &=0.
\end{align*}
\end{proposition}

It looks less scary if we write those equations as a matrix
whose rows indicate the five primitive relations. 
This matrix consists of five blocks of columns of
sizes $p_0,\ldots,p_4$.
By $\prn=(0,0,\ldots,0)$ and 
$\pre=(1,1,\ldots,1)$ we mean row vectors of the appropriate size
to fit into the indicated block.
Denoting \(\prc = (0, c_2, \dots , c_{p_2})\in\Z_{\geq 0}^{p_2}\)
and $\prb= (b_1, \dots , b_{p_3})\in\Z_{\geq 0}^{p_3}$, 
the primitive relation matrix looks like
\[
\left(
\begin{array}{ccccc}
\pre & \pre & -\prc & -(\prb+\pre)& \prn \\
\prn& \pre& \pre& \prn& -\pre \\
\prn& \prn& \pre& \pre& \prn \\ 
\prn& -\pre& \prn& \pre& \pre \\
\pre& \prn &-\prc& -\prb& \pre
\end{array}
\right).
\]

\subsection{Tempting Subsets}

As above, we suppose $X= \toric(\Sigma)$ is a smooth projective toric variety of dimension $d$ and Picard rank $3$,
   whose fan $\Sigma$ has five primitive relations.

For finding the immaculate line bundles the first step is to find the tempting subsets of \(\Sigma(1)\).
We have seen in Proposition~\ref{prop_3rd_temptation} 
that the primitive collections, their complements, the empty set and the 
full subset \(\Sigma(1)\) are tempting.

\begin{lemma}
The only tempting subsets are primitive collections, their complements, the empty set and the full subset \(\Sigma(1)\).
\end{lemma}

\begin{proof}
Let \(\mathcal{R}\) be a non-empty tempting subset, which is not equal to $\Sigma(1)$.
Then $\mathcal{R}$ and $\Sigma(1) \setminus \CR$ do not span cones in $\Sigma$
by Proposition~\ref{prop_2nd_temptation}.  
It follows that there exist two primitive collections \(P,P'\), with 
\(P \subseteq \mathcal{R} \subseteq \Sigma(1) \setminus P'\).
In the notation as above we obtain
\[J_{\alpha} \cup J_{\alpha+1} \subseteq \mathcal{R} \subseteq J_{\beta+2} \cup J_{\beta+3} \cup J_{\beta+4}\]
for some \(\alpha, \beta \in \quot{\bZ}{5 \bZ}\). This already implies that \(\beta = \alpha \pm  2\):
\begin{align*}
  J_{\alpha} \cup J_{\alpha+1} &\subseteq \mathcal{R} \subseteq J_{\alpha-1} \cup J_{\alpha} \cup J_{\alpha+1}, \text{ or}\\
  J_{\alpha} \cup J_{\alpha+1} &\subseteq \mathcal{R} \subseteq J_{\alpha} \cup J_{\alpha+1} \cup J_{\alpha+2}.
\end{align*}
For brevity we denote by $J_{\bullet}$, respectively, either $J_{\alpha-1}$ or $J_{\alpha+2}$. 
So the only question is, what is \(\mathcal{R} \cap J_{\bullet}\). 
If we show the intersection is empty or the whole \(J_{\bullet}\), the proof will be completed.

Denote by \(\mathcal{R}_{\bullet}  := \mathcal{R} \cap J_{\bullet}\),
  and assume conversely that the \(\mathcal{R}_{\bullet} \) is not equal to $\emptyset$ or \(J_{\bullet}\), 
  and consider  \(J_{\alpha} \cup \mathcal{R}_{\bullet}\).
This set does not contain any primitive collection, thus it is a face. The same holds for \(J_{\alpha+1} \cup \mathcal{R}_{\bullet}\). 
Hence \(\mathcal{R}\) is the union of two faces which intersect in a common 
face \(\mathcal{R}_{\bullet}\). This implies that \(\mathcal{R}\) is not 
tempting.
\end{proof}

\begin{proposition}
Suppose as above that \(X\) is a smooth projective toric variety of Picard rank 3 with five primitive collections \(J_{i}\) of lengths \(p_i\). 
If L is a line bundle on \(X\) such that \(\gH^i(X, L)\ne 0\), 
then \(i \in \set{0,\ p_{\alpha}+p_{\alpha+1}-1,\ p_{\alpha-1}+p_{\alpha-2}+p_{\alpha-3}-2,\ \dim X}_{\alpha \in \quot{\bZ}{5 \bZ}}\).
\end{proposition}

\begin{proof}
The tempting subset \(\Sigma(1)\) and \(\emptyset\) lead to line bundles with nontrivial cohomology 
in degrees \(0\) and \(\dim X\) respectively.

Along the lines of the proof of  Proposition~\ref{prop_3rd_temptation} 
we see that the complement of the primitive collection \(\mathcal{J}_{\alpha}\) 
leads to line bundles with nontrivial cohomology
in degree \(p_{\alpha}+p_{\alpha+1}-1\),
and the primitive collection itself to line bundles 
with nonvanishing cohomology in degree \(p_{\alpha-1}+p_{\alpha-2}+p_{\alpha-3}-2\).
Since there are no other tempting subsets, there cannot occure other degrees.
\end{proof}

\subsection{Immaculate line bundles for Picard rank \texorpdfstring{$3$}{3}}
We can calculate the immaculate line bundles as described in Proposition~\ref{prop-Rmac}.
For this we have to consider $\pi(\pms)$ for all maculate \(\mathcal{R}\)
where \(\pi\) is given as the transpose of the map embedding the kernel of
the ray map into \(\bZ^{\Sigma(1)}\).
This can be realized by selecting a $\Z$-basis out of the rows
of the matrix of primitive relations presented at the end of
Subsection~\ref{classBatyrev}. Picking its first, second and fourth row,
we obtain
\[
\pi=
\left(
\begin{array}{ccccc}
\pre & \pre & -\prc & -(\prb+\pre)& \prn \\
\prn& \pre& \pre& \prn& -\pre \\
\prn& -\pre& \prn& \pre& \pre 
\end{array}
\right).
\]

These are the primitive relations that, being understood as classes of
1-cycles, correspond to the rays of the Mori cone 
which in this case is a three-dimensional simplicial cone.

\begin{remark}\label{rem_everyone_is_really_immaculate_Pic_3}
\begin{enumerate}
 \item  For all parameters $\prb,\prc$, the matrix $\pi$ leads to a smooth fan
of Picard rank~3. 
This means the converse of Proposition~\ref{prop-picRank3Batyrev},
and it follows from Subsection~\ref{smoothGale}: 
The 3-minors with respect to
the columns chosen from the blocks $(\alpha,\alpha+1,\alpha+3)$
for $\alpha\in\Z/5\Z$ are always $1$.
\item\label{item_Z_generators_for_maculate_cones_Pic_3}
It is straightforward (although tedious) to check that for
all 12 tempting subsets $\CR \subset \Sigma(1)$
the tail cone of the respective maculate region $\CM_{\R}(\CR)$
is either a smooth cone 
or a cone with 4 rays which do also form its Hilbert basis 
(the latter is the case for \(J_3 \cup J_4\) if \(c_{p_2} < b_1+1\), 
for \(J_4 \cup J_0\) if \(b_1 > 0\), and for their respective complements).
\item From~\ref{item_Z_generators_for_maculate_cones_Pic_3} it follows that, independent of the parameters $\prb,\prc$, 
we always have that $\CM_{\Z}(\CR) = \CM_{\R}(\CR) \cap \Pic X$ 
and thus $\Imm_{\Z} (X) = \Imm_{\R}(X)$.
\end{enumerate}
\end{remark}

We will distinguish three classes (F), (A), (B) of line bundles
which will become the main heros 
for the immaculate locus  presented in 
Proposition~\ref{prop-mainPicThree}.
To locate these classes in $\Z^3$ we will use the horizontal projection
$(x,y,z)\mapsto (y,z)$ and start with some geography on the target space.

\begin{definition}\label{def:parallelograms}
Denote by $P_1$ and $P_2$ the following two planary parallelograms \(P_1, P_2\):
$$
\begin{array}{rcl}
              P_1 & = &\conv( \twodim{p_1-p_2-p_3+2}{p_1-1}, \twodim{-p_1}{p_1-1}, \\
&& \hspace{7em}\twodim{-p_2+p_4}{-p_3-p_4+1}, \twodim{p_3+p_4-2}{-p_3-p_4+1} ),\\[0.8ex] 
              P_2 & = &\conv( \twodim{-p_1-p_2+1}{p_1+p_2-2}, \twodim{p_4-1}{-p_4}, \\
&& \hspace{7em}\twodim{-p_1-p_2+1}{p_1-p_3}, \twodim{p_4-1}{-p_2-p_3-p_4+2})
\end{array}
$$
They are depicted in blue and red in \autoref{figure:pic3_mr_gen}, and we will be interested in their union. 
Note the following two special cases:
       \begin{itemize}
           \item If \(p_2 = 1\), then \(P_2 \subset P_1\) and the simplified vertices of \(P_1\) are:
\[
\twodim{p_1-p_3+1}{p_1-1}, \twodim{-p_1}{p_1-1}, \twodim{p_4-1}{-p_3-p_4+1}, \twodim{p_3+p_4-2}{-p_3-p_4+1} 
\]

\item If \(p_3 = 1\), then \(P_1 \subset P_2\) and the simplified vertices of \(P_2\) are:
\[
\twodim{-p_1-p_2+1}{p_1+p_2-2}, \twodim{-p_1-p_2+1}{p_1-1}, \twodim{p_4-1}{-p_4}, \twodim{p_4-1}{-p_2-p_4+1} 
\]
\end{itemize}       
\end{definition}

\begin{figure}
\newcommand{\vemptycolor}{purple}
\newcommand{\vemptypattern}{crosshatch dots}
\newcommand{\vemptythickness}{20}
\newcommand{\vtwocolor}{blue}
\newcommand{\vtwopattern}{dots}
\newcommand{\vtwothickness}{20}
\newcommand{\vzerocolor}{orange}
\newcommand{\vzeropattern}{vertical lines}
\newcommand{\vzerothickness}{80}
\newcommand{\vfourcolor}{olive}
\newcommand{\vfourpattern}{horizontal lines}
\newcommand{\vfourthickness}{80}
\newcommand{\vonecolor}{red}
\newcommand{\vonepattern}{north east lines}
\newcommand{\vonethickness}{80}
\tikzset{
    nodecountour/.style={
        fill=white,rounded corners=2pt,inner sep=1pt
    }
}
\begin{minipage}{\textwidth}
{  \centering
\begin{tikzpicture}[scale = 0.45,
                    color = {lightgray},
                    font = \small]

  \draw[step=1.0,black!20,thin] (-10.5,-10.5) grid (8.5,7.5);
  
  % POINTS STYLE
  \definecolor{pointcolor_P}{rgb}{ 1,0,0 }
  \tikzstyle{pointstyle_P} = [fill=pointcolor_P]

  % DEF POINTS
  \coordinate (v_K) at (-2, -3);
  \coordinate (v_sig1) at (0, 0);
  \coordinate (v_0) at (2,-7);
  \coordinate (v_0c) at (-4,4);
  \coordinate (v_13c) at (5,-7);
  \coordinate (v_1c3) at (-7,4);
  \coordinate (v_2) at (1,-1);
  \coordinate (v_2c) at (-3,-2);
  \coordinate (v_4) at (-7,2);
  \coordinate (v_4c) at (5,-5);

   \fill[pattern color=\vemptycolor!\vemptythickness, pattern=\vemptypattern] (-7.5,7.5)-- (0,0) -- (8.5,-8.5) --( 8.5,7.5)  -- cycle;
   \draw[color=\vemptycolor] (-7.5,7.5)-- (0,0) -- (8.5,-8.5);
   \fill[pattern color=\vemptycolor!\vemptythickness, pattern=\vemptypattern] (-10.5,5.5)-- (-2,-3) -- (5.5,-10.5) -- (-10.5,-10.5) -- cycle;
   \draw[color=\vemptycolor] (-10.5,5.5)-- (-2,-3) -- (5.5,-10.5);

   \fill[pattern color=\vtwocolor!\vtwothickness, pattern=\vtwopattern]  (-7.5,7.5) -- (1,-1) -- (8.5,-8.5)  --(8.5,7.5)  -- cycle;
   \draw[color=\vtwocolor] (-7.5,7.5) -- (1,-1) -- (8.5,-8.5) ;
   \fill[pattern color=\vtwocolor!\vtwothickness, pattern=\vtwopattern] (-10.5,5.5) -- (-3,-2) -- (5.5,-10.5) -- (-10.5,-10.5) -- cycle;
   \draw[color=\vtwocolor] (-10.5,5.5) -- (-3,-2) -- (5.5,-10.5);

   \fill[pattern color=\vzerocolor!\vzerothickness, pattern=\vzeropattern] (-7.5,7.5) -- (-4,4) -- (8.5, 4) --( 8.5,7.5)  -- cycle;
   \draw[color=\vzerocolor] (-7.5,7.5) -- (-4,4) -- (8.5, 4);
   \fill[pattern color=\vzerocolor!\vzerothickness, pattern=\vzeropattern] (5.5,-10.5) -- (2,-7) -- (-10.5,-7) -- (-10.5,-10.5) -- cycle;
   \draw[color=\vzerocolor] (5.5,-10.5)-- (2,-7) -- (-10.5,-7) ;

   \fill[pattern color=\vfourcolor!\vfourthickness, pattern=\vfourpattern]  (8.5,-8.5)-- (5,-5) -- (5,7.5) -- (8.5,7.5)  -- cycle;
   \draw[color=\vfourcolor] (8.5,-8.5)-- (5,-5) -- (5,7.5);
   \fill[pattern color=\vfourcolor!\vfourthickness, pattern=\vfourpattern] (-10.5,5.5)-- (-7,2) -- (-7,-10.5) -- (-10.5,-10.5) -- cycle;
   \draw[color=\vfourcolor] (-10.5,5.5)-- (-7,2) -- (-7,-10.5);
   
   \fill[pattern color=\vonecolor!\vonethickness, pattern=\vonepattern] (-10.5,4) -- (-7,4) -- (-7,7.5) -- (-10.5,7.5)  -- cycle;
   \draw[color=\vonecolor] (-10.5,4) -- (-7,4) -- (-7,7.5) ;
   \fill[pattern color=\vonecolor!\vonethickness, pattern=\vonepattern]  (5,-10.5) -- (5,-7) -- (8.5,-7) -- (8.5,-10.5) -- cycle;
   \draw[color=\vonecolor] (5,-10.5) -- (5,-7) -- (8.5,-7);

   \filldraw[color=red, pattern color=red, pattern=north east lines] (-6,5) -- (4,-5) -- (4,-8) -- (-6,2) -- cycle;
   \filldraw[color=blue, pattern color=blue, pattern=north west lines] (-7,3) -- (-4,3) -- (5,-6) -- (2,-6) -- cycle;

  %POINTS
  \fill[color=\vemptycolor] (v_K) circle (3 pt);
  \node at (v_K) [nodecountour, text=black, below left, draw=none, align=left] {$\overline{v_{\Sigma(1)}}$};
  \fill[color=\vemptycolor] (v_sig1) circle (3 pt);
  \node at (v_sig1) [nodecountour, text=black, above right, draw=none, align=left] {$\overline{v_{\emptyset}}$};
  \fill[color = \vzerocolor] (v_0) circle (3 pt);
  \node at (v_0) [nodecountour, text=black, below left, draw=none, align=left] {$\overline{v_{0^c}}$};
  \fill[color = \vzerocolor] (v_0c) circle (3 pt);
  \node at (v_0c) [nodecountour, text=black, above right, draw=none, align=left] {$\overline{v_{0}}$};
  \fill[color = \vonecolor] (v_13c) circle (3 pt);
  \node at (v_13c) [nodecountour, text=black, below right, draw=none, align=left] {$\overline{v_{1^c}} = \overline{v_{3}}$};
  \fill[color = \vonecolor] (v_1c3) circle (3 pt);
  \node at (v_1c3) [nodecountour, text=black, above left, draw=none, align=left] {$\overline{v_{1}} = \overline{v_{3^c}}$};
  \fill[color = \vtwocolor] (v_2) circle (3 pt);
  \node at (v_2) [nodecountour, text=black, above right, draw=none, align=left] {$\overline{v_{2^c}}$};
  \fill[color = \vtwocolor] (v_2c) circle (3 pt);
  \node at (v_2c) [nodecountour, text=black, below left, draw=none, align=left] {$\overline{v_{2}}$};
  \fill[color = \vfourcolor] (v_4) circle (3 pt);
  \node at (v_4) [nodecountour, text=black, below left, draw=none, align=left] {$\overline{v_{4^c}}$};
  \fill[color = \vfourcolor] (v_4c) circle (3 pt);
  \node at (v_4c) [nodecountour, text=black, above right, draw=none, align=left] {$\overline{v_{4}}$};

   \node[color=black, thick] at (-7,4) {$B$};
   \node[color=black, thick] at (5,-7) {$b$};

   \foreach \i in {0,...,12} {
      \node[color=black] at (6-\i,-6+\i) {$A$};
   }
   \foreach \i in {0,...,12} {
      \node[color=black] at (-8+\i,3-\i) {$a$};
   }
  % POINTS STYLE
  \definecolor{pointcolor_P}{rgb}{ 1,0,0 }
  \definecolor{pointcolor_V}{rgb}{ 1,0,0 }
  \tikzstyle{pointstyle_P} = [fill=pointcolor_P]
  \tikzstyle{pointstyle_V} = [fill=pointcolor_V]
\end{tikzpicture}

\begin{tabular}{cc|cc}
\toprule
\textbf{Point} & \textbf{Coordinates} &
\textbf{Point} & \textbf{Coordinates}\\
\midrule
$\overline{v_{\Sigma (1)}} = \overline{K}$ & $\twodim{-p_1-p_2+p_4}{p_1-p_3-p_4}$ &
$\overline{v_{\emptyset}}$ & $\twodim{0}{0}$\\
$\overline{v_{0^c}}$ & $\twodim{-p_2+p_4}{-p_3-p_4}$ &
$\overline{v_{0}}$ & $\twodim{-p_1}{p_1}$\\
$\overline{v_{1^c}} = \overline{v_{3}}$ & $\twodim{p_4}{-p_3-p_4}$ &
$\overline{v_{1}} = \overline{v_{3^c}}$ & $\twodim{-p_1-p_2}{p_1}$\\
$\overline{v_{2^c}}$ & $\twodim{-p_1+p_4}{p_1-p_4}$ & 
$\overline{v_{2}}$ & $\twodim{-p_2}{-p_3}$\\
$\overline{v_{4^c}}$ & $\twodim{-p_1-p_2}{p_1-p_3}$ &
$\overline{v_{4}}$ & $\twodim{p_4}{-p_4}$\\
\bottomrule
\end{tabular}
% $\overline{v_{\emptyset}} = (-p_1-p_2+p_4,p_1-p_3-p_4)$
% $\overline{v_{\Sigma (1)}} = (0,0)$
% $\overline{v_{0}} = (-p_2+p_4,-p_3-p_4)$
% $\overline{v_{0^c}} = (-p_1,p_1)$
% $\overline{v_{1}} = \overline{v_{3^c}} = (p_4,-p_3-p_4)$
% $\overline{v_{1^c}} = \overline{v_{3}} = (-p_1-p_2,p_1)$
% $\overline{v_{2}} = (-p_1+p_4,p_1-p_4)$
% $\overline{v_{2^c}} = (-p_2,-p_3)$
% $\overline{v_{4}} = (-p_1-p_2,p_1-p_3)$
% $\overline{v_{4^c}} = (p_4,-p_4)$
\caption{\(p_2,p_3 >1\)}\label{figure:pic3_mr_gen}}
The projected maculate regions to the \((y,\ z)\)-plane for the example 
\((p_1,\ p_2,\ p_3,\ p_4) =(4,\ 3,\ 2,\ 5)\) and a table with the general coordinates of the projected vertices of the maculate regions,
where \(\overline{v_{i}}\) and \(\overline{v_{i^c}}\) denotes the projected vertex of 
the maculate region $\CM_{\R}(\CR)$
for \(\CR = \mathcal{J}_i\) respectively \(\CR = \mathcal{J}_i^c\).
The polyhedra \(P_1\) and \(P_2\) from Definition~\ref{def:parallelograms} are depicted in blue and red.
The letters \(A\) and \(B\) indicate where the line segments of immaculate line bundles are located
in the projection, and the letters \(a,b\) denote the location of their Serre duals.
\vspace{0.15cm}
\end{minipage}
\begin{minipage}{\textwidth}
\centering
\begin{minipage}{.49\textwidth}
 \centering
\begin{tikzpicture}[scale = 0.35,
                    color = {lightgray},
                    font = \tiny]

  \draw[step=1.0,black!20,thin] (-10.5,-10.5) grid (8.5,7.5);
  
  % POINTS STYLE
  \definecolor{pointcolor_P}{rgb}{ 1,0,0 }
  \tikzstyle{pointstyle_P} = [fill=pointcolor_P]

  % DEF POINTS
  \coordinate (v_K) at (0, -3);
  \coordinate (v_sig1) at (0, 0);
  \coordinate (v_0) at (4,-7);
  \coordinate (v_0c) at (-4,4);
  \coordinate (v_13c) at (5,-7);
  \coordinate (v_1c3) at (-5,4);
  \coordinate (v_2) at (1,-1);
  \coordinate (v_2c) at (-1,-2);
  \coordinate (v_4) at (-5,2);
  \coordinate (v_4c) at (5,-5);

   \fill[pattern color=\vemptycolor!\vemptythickness, pattern=\vemptypattern] (-7.5,7.5)-- (0,0) -- (8.5,-8.5) --( 8.5,7.5)  -- cycle;
   \draw[color=\vemptycolor] (-7.5,7.5)-- (0,0) -- (8.5,-8.5);
   \fill[pattern color=\vemptycolor!\vemptythickness, pattern=\vemptypattern] (-10.5,7.5) -- (-1,-2) -- (7.5,-10.5) -- (-10.5,-10.5) -- cycle;
   \draw[color=\vemptycolor] (-10.5,7.5)-- (0,-3) -- (7.5,-10.5);

   \fill[pattern color=\vtwocolor!\vtwothickness, pattern=\vtwopattern]  (-7.5,7.5) -- (1,-1) -- (8.5,-8.5)  --(8.5,7.5)  -- cycle;
   \draw[color=\vtwocolor] (-7.5,7.5) -- (1,-1) -- (8.5,-8.5) ;
   \fill[pattern color=\vtwocolor!\vtwothickness, pattern=\vtwopattern] (-10.5,7.5) -- (-1,-2) -- (7.5,-10.5) -- (-10.5,-10.5) -- cycle;
   \draw[color=\vtwocolor] (-10.5,7.5) -- (-1,-2) -- (7.5,-10.5);

   \fill[pattern color=\vzerocolor!\vzerothickness, pattern=\vzeropattern] (-7.5,7.5) -- (-4,4) -- (8.5, 4) --( 8.5,7.5)  -- cycle;
   \draw[color=\vzerocolor] (-7.5,7.5) -- (-4,4) -- (8.5, 4);
   \fill[pattern color=\vzerocolor!\vzerothickness, pattern=\vzeropattern] (7.5,-10.5) -- (4,-7) -- (-10.5,-7) -- (-10.5,-10.5) -- cycle;
   \draw[color=\vzerocolor] (7.5,-10.5)-- (4,-7) -- (-10.5,-7) ;

   \fill[pattern color=\vfourcolor!\vfourthickness, pattern=\vfourpattern]  (8.5,-8.5)-- (5,-5) -- (5,7.5) -- (8.5,7.5)  -- cycle;
   \draw[color=\vfourcolor] (8.5,-8.5)-- (5,-5) -- (5,7.5);
   \fill[pattern color=\vfourcolor!\vfourthickness, pattern=\vfourpattern] (-10.5,7.5)-- (-5,2) -- (-5,-10.5) -- (-10.5,-10.5) -- cycle;
   \draw[color=\vfourcolor] (-10.5,7.5)-- (-5,2) -- (-5,-10.5);
   
   \fill[pattern color=\vonecolor!\vonethickness, pattern=\vonepattern] (-10.5,4) -- (-5,4) -- (-5,7.5) -- (-10.5,7.5)  -- cycle;
   \draw[color=\vonecolor] (-10.5,4) -- (-5,4) -- (-5,7.5) ;
   \fill[pattern color=\vonecolor!\vonethickness, pattern=\vonepattern]  (5,-10.5) -- (5,-7) -- (8.5,-7) -- (8.5,-10.5) -- cycle;
   \draw[color=\vonecolor] (5,-10.5) -- (5,-7) -- (8.5,-7);

  %POINTS
  \fill[color=\vemptycolor] (v_K) circle (3 pt);
  % \node at (v_K) [text=black, inner sep=0.5pt, above left, draw=none, align=left] {$\overline{v_{\vemptyset}} = (-p_1+p_4-1,p_1-p_3-p_4)$};
  \fill[color=\vemptycolor] (v_sig1) circle (3 pt);
  % \node at (v_sig1) [text=black, inner sep=0.5pt, above right, draw=none, align=left] {$\overline{v_{\Sigma (1)}} = (0,0)$};
  \fill[color = \vzerocolor] (v_0) circle (3 pt);
  % \node at (v_0) [text=black, inner sep=0.5pt, below left, draw=none, align=left] {$\overline{v_{0}} = (p_4-1,-p_3-p_4)$};
  \fill[color = \vzerocolor] (v_0c) circle (3 pt);
  % \node at (v_0c) [text=black, inner sep=0.5pt, above right, draw=none, align=left] {$\overline{v_{0^c}} = (-p_1,p_1)$};
  \fill[color = \vonecolor] (v_13c) circle (3 pt);
  % \node at (v_13c) [text=black, inner sep=0.5pt, below right, draw=none, align=left] {$\overline{v_{1}} = \overline{v_{3^c}} = (p_4,-p_3-p_4)$};
  \fill[color = \vonecolor] (v_1c3) circle (3 pt);
  % \node at (v_1c3) [text=black, inner sep=0.5pt, above left, draw=none, align=left] {$\overline{v_{1^c}} = \overline{v_{3}} = (-p_1-1,p_1)$};
  \fill[color = \vtwocolor] (v_2) circle (3 pt);
  % \node at (v_2) [text=black, inner sep=0.5pt, above right, draw=none, align=left] {$\overline{v_{2}} = (-p_1+p_4,p_1-p_4)$};
  \fill[color = \vtwocolor] (v_2c) circle (3 pt);
  % \node at (v_2c) [text=black, inner sep=0.5pt, above left, draw=none, align=left] {$\overline{v_{2^c}} = (-1,-p_3)$};
  \fill[color = \vfourcolor] (v_4) circle (3 pt);
  % \node at (v_4) [text=black, inner sep=0.5pt, above left, draw=none, align=left] {$\overline{v_{4}} = (-p_1-1,p_1-p_3)$};
  \fill[color = \vfourcolor] (v_4c) circle (3 pt);
  % \node at (v_4c) [text=black, inner sep=0.5pt, below right, draw=none, align=left] {$\overline{v_{4^c}} = (p_4,-p_4)$};

   \filldraw[color=red, pattern color=red, pattern=north east lines] (-4,3) -- (4,-5) -- (4,-6) -- (-4,2) -- cycle;
   \filldraw[color=blue, pattern color=blue, pattern=north west lines] (-5,3) -- (-4,3) -- (5,-6) -- (4,-6) -- cycle;

   \node[color=black, thick] at (-6,4) {$B$};
   \node[color=black, thick] at (-5,4) {$B$};
   \node[color=black, thick] at (5,-7) {$b$};
   \node[color=black, thick] at (6,-7) {$b$};
   
   \foreach \i in {0,...,10} {
      \node[color=black] at (6-\i,-6+\i) {$A$};
   }
   \foreach \i in {0,...,10} {
      \node[color=black] at (-6+\i,3-\i) {$a$};
   }

  % POINTS STYLE
  \definecolor{pointcolor_P}{rgb}{ 1,0,0 }
  \definecolor{pointcolor_V}{rgb}{ 1,0,0 }
  \tikzstyle{pointstyle_P} = [fill=pointcolor_P]
  \tikzstyle{pointstyle_V} = [fill=pointcolor_V]
\end{tikzpicture}
\caption{\(p_2=1\)}\label{figure:pic3_mr_p2}
\((p_1,\ p_2,\ p_3,\ p_4) =(4,\ 1,\ 2,\ 5)\).
\end{minipage}
\begin{minipage}{.49\textwidth}
  \centering
\begin{tikzpicture}[scale = 0.35,
                    color = {lightgray},
                    font=\tiny]

  \draw[step=1.0,black!20,thin] (-10.5,-10.5) grid (8.5,7.5);
  
  % POINTS STYLE
  \definecolor{pointcolor_P}{rgb}{ 1,0,0 }
  \tikzstyle{pointstyle_P} = [fill=pointcolor_P]

  % DEF POINTS
  \coordinate (v_K) at (-2, -2);
  \coordinate (v_sig1) at (0, 0);
  \coordinate (v_0) at (2,-6);
  \coordinate (v_0c) at (-4,4);
  \coordinate (v_13c) at (5,-6);
  \coordinate (v_1c3) at (-7,4);
  \coordinate (v_2) at (1,-1);
  \coordinate (v_2c) at (-3,-1);
  \coordinate (v_4) at (-7,3);
  \coordinate (v_4c) at (5,-5);

   \fill[pattern color=\vemptycolor!\vemptythickness, pattern=\vemptypattern] (-7.5,7.5)-- (0,0) -- (8.5,-8.5) --( 8.5,7.5)  -- cycle;
   \draw[color=\vemptycolor] (-7.5,7.5)-- (0,0) -- (8.5,-8.5);
   \fill[pattern color=\vemptycolor!\vemptythickness, pattern=\vemptypattern] (-10.5,6.5)-- (-2,-2) -- (6.5,-10.5) -- (-10.5,-10.5) -- cycle;
   \draw[color=\vemptycolor] (-10.5,6.5)-- (-2,-2) -- (6.5,-10.5);

   \fill[pattern color=\vtwocolor!\vtwothickness, pattern=\vtwopattern]  (-7.5,7.5) -- (1,-1) -- (8.5,-8.5)  --(8.5,7.5)  -- cycle;
   \draw[color=\vtwocolor] (-7.5,7.5) -- (1,-1) -- (8.5,-8.5) ;
   \fill[pattern color=\vtwocolor!\vtwothickness, pattern=\vtwopattern] (-10.5,6.5) -- (-3,-1) -- (6.5,-10.5) -- (-10.5,-10.5) -- cycle;
   \draw[color=\vtwocolor] (-10.5,6.5) -- (-3,-1) -- (6.5,-10.5);

   \fill[pattern color=\vzerocolor!\vzerothickness, pattern=\vzeropattern] (-7.5,7.5) -- (-4,4) -- (8.5, 4) --( 8.5,7.5)  -- cycle;
   \draw[color=\vzerocolor] (-7.5,7.5) -- (-4,4) -- (8.5, 4);
   \fill[pattern color=\vzerocolor!\vzerothickness, pattern=\vzeropattern] (6.5,-10.5) -- (2,-6) -- (-10.5,-6) -- (-10.5,-10.5) -- cycle;
   \draw[color=\vzerocolor] (6.5,-10.5)-- (2,-6) -- (-10.5,-6) ;

   \fill[pattern color=\vfourcolor!\vfourthickness, pattern=\vfourpattern]  (8.5,-8.5)-- (5,-5) -- (5,7.5) -- (8.5,7.5)  -- cycle;
   \draw[color=\vfourcolor] (8.5,-8.5)-- (5,-5) -- (5,7.5);
   \fill[pattern color=\vfourcolor!\vfourthickness, pattern=\vfourpattern] (-10.5,6.5)-- (-7,3) -- (-7,-10.5) -- (-10.5,-10.5) -- cycle;
   \draw[color=\vfourcolor] (-10.5,6.5)-- (-7,3) -- (-7,-10.5);
   
   \fill[pattern color=\vonecolor!\vonethickness, pattern=\vonepattern] (-10.5,4) -- (-7,4) -- (-7,7.5) -- (-10.5,7.5)  -- cycle;
   \draw[color=\vonecolor] (-10.5,4) -- (-7,4) -- (-7,7.5) ;
   \fill[pattern color=\vonecolor!\vonethickness, pattern=\vonepattern]  (5,-10.5) -- (5,-6) -- (8.5,-6) -- (8.5,-10.5) -- cycle;
   \draw[color=\vonecolor] (5,-10.5) -- (5,-6) -- (8.5,-6);

  %POINTS
  \fill[color=\vemptycolor] (v_K) circle (3 pt);
  % \node at (v_K) [text=black, inner sep=0.5pt, above left, draw=none, align=left] {$\overline{v_{\emptyset}} = (-p_1-p_2+p_4,p_1-p_4-1)$};
  \fill[color=\vemptycolor] (v_sig1) circle (3 pt);
  % \node at (v_sig1) [text=black, inner sep=0.5pt, above right, draw=none, align=left] {$\overline{v_{\Sigma (1)}} = (0,0)$};
  \fill[color = \vzerocolor] (v_0) circle (3 pt);
  % \node at (v_0) [text=black, inner sep=0.5pt, below left, draw=none, align=left] {$\overline{v_{0}} = (-p_2+p_4,-p_4-1)$};
  \fill[color = \vzerocolor] (v_0c) circle (3 pt);
  % \node at (v_0c) [text=black, inner sep=0.5pt, above right, draw=none, align=left] {$\overline{v_{0^c}} = (-p_1,p_1)$};
  \fill[color = \vonecolor] (v_13c) circle (3 pt);
  % \node at (v_13c) [text=black, inner sep=0.5pt, below right, draw=none, align=left] {$\overline{v_{1}} = \overline{v_{3^c}} = (p_4,-p_4-1)$};
  \fill[color = \vonecolor] (v_1c3) circle (3 pt);
  % \node at (v_1c3) [text=black, inner sep=0.5pt, above left, draw=none, align=left] {$\overline{v_{1^c}} = \overline{v_{3}} = (-p_1-p_2,p_1)$};
  \fill[color = \vtwocolor] (v_2) circle (3 pt);
  % \node at (v_2) [text=black, inner sep=0.5pt, above right, draw=none, align=left] {$\overline{v_{2}} = (-p_1+p_4,p_1-p_4)$};
  \fill[color = \vtwocolor] (v_2c) circle (3 pt);
  % \node at (v_2c) [text=black, inner sep=0.5pt, above left, draw=none, align=left] {$\overline{v_{2^c}} = (-p_2,-1)$};
  \fill[color = \vfourcolor] (v_4) circle (3 pt);
  % \node at (v_4) [text=black, inner sep=0.5pt, above left, draw=none, align=left] {$\overline{v_{4}} = (-p_1-p_2,p_1-1)$};
  \fill[color = \vfourcolor] (v_4c) circle (3 pt);
  % \node at (v_4c) [text=black, inner sep=0.5pt, below right, draw=none, align=left] {$\overline{v_{4^c}} = (p_4,-p_4)$};
   \filldraw[color=red, pattern color=red, pattern=north east lines] (-6,5) -- (4,-5) -- (4,-7) -- (-6,3) -- cycle;
   \filldraw[color=blue, pattern color=blue, pattern=north west lines] (-6,3) -- (-4,3) -- (4,-5) -- (2,-5) -- cycle;

   \node[color=black, thick] at (-7,4) {$B$};
   \node[color=black, thick] at (-7,5) {$B$};
   \node[color=black, thick] at (-7,6) {$B$};
   \node[color=black, thick] at (5,-6) {$b$};
   \node[color=black, thick] at (5,-7) {$b$};
   \node[color=black, thick] at (5,-8) {$b$};
   
   \foreach \i in {0,...,11} {
      \node[color=black] at (5-\i,-5+\i) {$A$};
   }
   \foreach \i in {0,...,11} {
      \node[color=black] at (-7+\i,3-\i) {$a$};
   }

  % POINTS STYLE
  \definecolor{pointcolor_P}{rgb}{ 1,0,0 }
  \definecolor{pointcolor_V}{rgb}{ 1,0,0 }
  \tikzstyle{pointstyle_P} = [fill=pointcolor_P]
  \tikzstyle{pointstyle_V} = [fill=pointcolor_V]
\end{tikzpicture}
\caption{\(p_3 =1\)}\label{figure:pic3_mr_p3}
\((p_1,\ p_2,\ p_3,\ p_4) =(4,\ 3,\ 1,\ 5)\).
\end{minipage}
\end{minipage}
\end{figure}

Now we can describe the three classes of our immaculate candidates. 
They consist of entire ``horizontal''
lines or line segments, that are parallel to the $x$-axis:

\begin{itemize}
 \item \textbf{Full horizontal lines (F).}
This class consists of the union of the (infinite) lines $\threedim{*}{y}{z}$        
with $(y,\ z) \in P_1\cup P_2$ (including the boundary).
Note that it does not depend on the values of $\prb$ and $\prc$ 
  and it is self dual with respect to Serre duality: 
here the canonical divisor is 
   \( \threedim{-p_0-p_1+p_3+\poc+\pob}{-p_1-p_2+p_4}{p_1-p_3-p_4} \).
\item \textbf{Line segments of Type (A).}
This class consists of finite horizontal segments $I_y$ (described below) located over the diagonal $\threedim{*}{y}{-y}$.
Denote \(D_{x,y} = \threedim{x}{-y}{y}\), and 
for any \(y \in [-p_3-p_4+1, p_1+p_2-1]\) 
let 
\[
   I_y\ :=\ \{D_{x,y}\ | \ x_0(y) \le x \le x_1(y)\}
\]
be the set of lattice points on the segment with $x$ coordinate varying from $ x_0(y)$ to $x_1(y)$.
The values of $x_0(y),x_1(y)$ and the number of elements of $I_y$ 
is in \autoref{table:isolated_typeA}.
Notice that they do not depend on $\prb$ or $\prc$, as in the case of type (F).

\begin{table}[htb]\centering
\caption{Isolated immaculate line bundles type A.}\label{table:isolated_typeA}
\[
\begin{array} {cc|cc|c}
\toprule
\bullet \le y & y \le \bullet & x_0(y) & x_1(y) & \# I_y \\
\midrule
\textbf{Case $p_1<p_4$}\\
\midrule
-p_3-p_4+1 & -p_4      & -p_0-p_4-y+1 & -y-1 & p_0+p_4-1 \\
-p_4+1     & p_1-p_4   & -p_0-p_1+1 & -y-1 & p_0+p_1-y-1 \\
p_1-p_4+1  & 0         & -p_0-p_4-y+1 & -y-1 & p_0+p_4-1 \\
1          & p_1-1     & -p_0-p_4-y+1 & -1 & p_0+p_4+y-1 \\
p_1        & p_1+p_2-1 & -p_0-p_1+1 & -1 & p_0+p_1-1 \\
\bottomrule
\toprule
\textbf{Case $p_1>p_4$}\\
\midrule
-p_3-p_4+1 & -p_4      & -p_0-p_4-y+1 & -y-1 & p_0+p_4-1 \\
-p_4+1     & 0         & -p_0-p_1+1 & -y-1 & p_0+p_1-y-1 \\
1          & p_1-p_4   & -p_0-p_1+1 & -1 & p_0+p_1-1 \\
p_1-p_4+1  & p_1-1     & -p_0-p_4-y+1 & -1 & p_0+p_4+y-1 \\
p_1        & p_1+p_2-1 & -p_0-p_1+1 & -1 & p_0+p_1-1 \\
\bottomrule
\toprule
\textbf{Case $p_1=p_4$}\\
\midrule
-p_3-p_4+1 & -p_4      & -p_0-p_4-y+1 & -y-1 & p_0+p_4-1 \\
-p_4+1     & 0         & -p_0-p_1+1 & -y-1 & p_0+p_1-y-1 \\
1          & p_1-1     & -p_0-p_4-y+1 & -1 & p_0+p_4+y-1 \\
p_1        & p_1+p_2-1 & -p_0-p_1+1 & -1 & p_0+p_1-1 \\
\bottomrule
\end{array}
\]
\end{table}

\item \textbf{Line segments of Type (B).}
The segments of this type depend on \(p_2\) and \(p_3\) 
via the parallelograms $P_1$ and $P_2$ 
elaborated in Definition~\ref{def:parallelograms}.
\begin{itemize}
\item If \(p_2,p_3 \ge 2\), then this type consist of just one horizontal segment
   whose projection to the \((y,\ z)\)-plane is located left and above the intersection
of the upper edges of the parallelograms \(P_1\) and \(P_2\), see the point marked as $B$ on \autoref{figure:pic3_mr_gen}.
The line segment contains \(p_0-1\) immaculate line bundles with coordinates 
\[\threedim{[-p_0-p_1+\poc+1,-p_1+\poc-1]}{-p_1-p_2}{p_1},\] 
where \(\poc := \sum c_i\).

\item 
If  \(p_2 = 1\), then the points of Type (B) consist of \(p_3\) horizontal line segments,
   each containing \(p_0-1\) immaculate line bundles.
The coordinates are \[\threedim{[-p_0-p_1+1,-p_1-1]}{-p_1-p_2-y}{p_1}\] for \(y \in [0,p_3-1]\).
On \autoref{figure:pic3_mr_p2} their projections onto $(y,z)$ plane are indicated by the letter \(B\).
Roughly speaking, their projections are at each lattice point directly above the upper edge of the parallelogram \(P_1\), 

\item For \(p_3 = 1\), there are \(p_2\) horizontal line segments of Type (B) each containing \(p_0-1\) immaculate line bundles. 
The coordinates are \[\threedim{[-p_0-p_1+\poc-y(\pob+1)+1, -p_1+\poc-y(\pob+1)-1]}{-p_1-p_2}{p_1+y}\]  for \(y \in [0,p_2-1]\).
On \autoref{figure:pic3_mr_p3} their locations in the projection are indicated by the letter \(B\), as for the previous case.
In this case, the projections are located directly left of \(P_2\).
\end{itemize}
\end{itemize}

Our result is that the types (F), (A), (B) are always immaculate. 
Moreover, for sufficiently ``general'' parameters the listed line bundles and their  Serre duals
are all really immaculate line bundles. 
We summarise this discussion in the following proposition, 
  but we only sketch the proof as it consists of working out the combinatorial details.

\begin{proposition} 
\label{prop-mainPicThree}
For $X$ a toric projective variety of Picard number $3$ with $5$ primitive collections we have:
\begin{enumerate}
\item All the line bundles of type (F), (A) and (B) are really immaculate.
\item The coordinates of the line bundles of type (F) and (A) do not depend on $\prb$ and $\prc$, 
the coordinates of the Serre duals of the line bundles of (A),  do depend on $\prb$ and $\prc$.

\item The line bundles of type (A) are the only immaculate line bundles among the \(D_{x,y} = \threedim{x}{-y}{y}\) 
with \(y \in [-p_3 - p_4 + 1, p_1 + p_2 + 1]\) independent of
\(\prb,\prc\).

\item For \(\prb,\prc\) large enough, that is \(\max(b_{p_3}, c_{p_2}) \ge
p_0+p_1+\max(p_2,p_3)+p_4\) and if \(p_2 \not= 1\) additionally \(c_{p_2} \ge p_0-1\) 
and for \(p_3 \not= 1\) the additional condition that \(b_{p_3} - b_1 \ge p_0-1\),
   the only (really) immaculate line bundles are the previously mentioned and
their Serre duals.
\end{enumerate}
\end{proposition}

\begin{proof}[Sketch of proof.]
For proving all of those statements we will consider the ''horizontal'' projection of 
the Picard group and in particular of the twelve maculate regions to the \(y,z\)-plane. 
Figures~\ref{figure:pic3_mr_gen}, \ref{figure:pic3_mr_p2} and \ref{figure:pic3_mr_p3} illustrate the situation, 
   for the cases \(p_2, p_3 >1\)  and \(p_2 =1\), \(p_3=1\) respectively.

\begin{itemize}
\item \textbf{(F) Full horizontal lines.}
If the projected divisor \(\overline{D}\) is not in any of the projected maculate cones $\CM_{\R}(\CR)$,
then all the divisors in the line parallel to the kernel of the projection are immaculate.

\item \textbf{Line segments of Type (A).}
Given a divisor \(D_{x,y}\) we want to know whether it is immaculate. Thus we
want to know whether there is an $\CR$ such that $D_{x,y}\in\CM_{\R}(\CR)$. 
We analyse the projected situation. We know that if \(\overline{D} \notin
\overline{\CM_{\R}(\CR)}\), then \(D \notin \CM_{\R}(\CR)\). This eliminates a large
number of candidate $\CR$'s. Various  situations that can occur are
depicted in the aforementioned figures.

Once we have identified the candidate $\CR$'s we do a case analysis for the $y$-coordinate.

\item \textbf{Line Segments of Type (B).}
The proof essentially proceeds in the same manner as for the isolated points
of type (A), by working out the combinatorial details of the situation.
\end{itemize}
\end{proof}

\section{Computational aspects}\label{sec:computational}
In this section we want to highlight the computational advantages of immaculate
line bundles and maculate regions. All of these objects and conditions give
rise to nice combinatorial algorithms. Throughout the development of this paper
we have implemented these in \polymake and they are available on \url{github
link} as a \polymake(\cite{polymake}) extension. The combinatorial nature of these algorithms
makes them very fast, as opposed to many algorithms from commutative algebra.
This stresses the main computational advantage of working with toric
varieties. We will give a short sketch of the resulting algorithms.

Immaculacy of a line bundle can be checked from its representation as a
difference of nef divisors. Thus we want to check all differences
$\Delta^-\setminus (\Delta^+-\dm)$, for any $\dm\in M$,
   for $\kk$-acyclicity, via
Proposition~\ref{prop-diffPol}. But it is actually enough to check only finitely many
$m$, since both $\Delta^-$ and $\Delta^+$ are compact and thus they only
intersect for finitely many shifts. Using
Proposition~\ref{prop_retract_on_the_boundary_complex}, we just need to consider
$\Delta^-$ as a polytopal complex and remove any face intersecting 
$\Delta^+-\dm$
non-trivially, for those finitely many $\dm$. Homology computation of the resulting polytopal
complex is already built in \polymake and many other software frameworks for
combinatorial software as well.

Next we want to find the tempting $\CR\subseteq\Sigma[1]$. The easiest way is
to brute force this by checking any subset of rays and then compute the
homology. One can also imagine a more sophisticated approach by considering
sub-diagrams of the Hasse diagram of $\Sigma$. So far this has never been a
bottleneck in our examples, though in case this happens, results of 
Subsection~\ref{sect_three_temptations} might be of use.

\begin{table}[htb]\centering
\caption{Lines of immaculate line bundles for the hexagon}\label{table:hexagon_lines}
\newcommand{\tinyscale}{.3}
\begin{tabular}{>{\hspace{.8cm}}cccc<{\hspace{.8cm}}|cccc|cc}
\toprule
\multicolumn{4}{c}{unbounded direction} & \multicolumn{4}{c}{basepoint} & $(\Delta^+$, & $\Delta^-)$\\
\midrule
\multirow{4}{*}{1} & 
\multirow{4}{*}{1} & 
\multirow{4}{*}{0} & 
\multirow{4}{*}{0} & 
     0 & 0 & -1 & -1
     & 
     $ t\cdot\begin{tikzpicture}[scale=\tinyscale]
     \draw[thick] (0,0) -- (1,1);
     \end{tikzpicture}$ &
     $\begin{tikzpicture}[scale=\tinyscale]
     \draw[thick] (0,0) -- (1,1) -- (1,0) -- cycle;
     \end{tikzpicture}
     $
     \\
&&&& 1 & 0 & -1 & 0
     & 
     $ t\cdot\begin{tikzpicture}[scale=\tinyscale]
     \draw[thick] (0,0) -- (1,1);
     \end{tikzpicture}$ &
     $\begin{tikzpicture}[scale=\tinyscale]
     \draw[thick] (0,0) -- (0,1) -- cycle;
     \end{tikzpicture}
     $
     \\
&&&& 0 & 0 & -1 & 0
     & 
     $ t\cdot\begin{tikzpicture}[scale=\tinyscale]
     \draw[thick] (0,0) -- (1,1);
     \end{tikzpicture}$ &
     $\begin{tikzpicture}[scale=\tinyscale]
     \draw[thick] (0,0) -- (1,1) -- (0,1) -- cycle;
     \end{tikzpicture}
     $
     \\
&&&& -1 & 0 & -1 & -1
     & 
     $ t\cdot\begin{tikzpicture}[scale=\tinyscale]
     \draw[thick] (0,0) -- (1,1);
     \end{tikzpicture}$ &
     $\begin{tikzpicture}[scale=\tinyscale]
     \draw[thick] (0,0) -- (1,0) -- cycle;
     \end{tikzpicture}
     $
     \\
\midrule
\multirow{4}{*}{1} & 
\multirow{4}{*}{0} & 
\multirow{4}{*}{1} & 
\multirow{4}{*}{1} 
   & 0 & -1 & -1 & 0
     & 
     $ t\cdot\begin{tikzpicture}[scale=\tinyscale]
     \draw[thick] (0,0) -- (1,0);
     \end{tikzpicture}$ &
     $\begin{tikzpicture}[scale=\tinyscale]
     \draw[thick] (0,0) -- (0,1) -- cycle;
     \end{tikzpicture}
     $
     \\
&&&& 0 & -1 & 0 & 0
     & 
     $ t\cdot\begin{tikzpicture}[scale=\tinyscale]
     \draw[thick] (0,0) -- (1,0);
     \end{tikzpicture}$ &
     $\begin{tikzpicture}[scale=\tinyscale]
     \draw[thick] (0,0) -- (1,1) -- (1,0) -- cycle;
     \end{tikzpicture}
     $
     \\
&&&& -1 & -1 & 0 & 0
     & 
     $ t\cdot\begin{tikzpicture}[scale=\tinyscale]
     \draw[thick] (0,0) -- (1,0);
     \end{tikzpicture}$ &
     $\begin{tikzpicture}[scale=\tinyscale]
     \draw[thick] (0,0) -- (1,1) -- cycle;
     \end{tikzpicture}
     $
     \\
&&&& -1 & -1 & -1 & 0
     & 
     $ t\cdot\begin{tikzpicture}[scale=\tinyscale]
     \draw[thick] (0,0) -- (1,0);
     \end{tikzpicture}$ &
     $\begin{tikzpicture}[scale=\tinyscale]
     \draw[thick] (0,0) -- (0,1) -- (1,1) -- cycle;
     \end{tikzpicture}
     $
     \\
\midrule
\multirow{4}{*}{0} & 
\multirow{4}{*}{1} & 
\multirow{4}{*}{1} & 
\multirow{4}{*}{0} 
   & -1 & 0 & 0 & 0
     & 
     $ t\cdot\begin{tikzpicture}[scale=\tinyscale]
     \draw[thick] (0,0) -- (0,1);
     \end{tikzpicture}$ &
     $\begin{tikzpicture}[scale=\tinyscale]
     \draw[thick] (0,0) -- (0,1) -- (1,1) -- cycle;
     \end{tikzpicture}
     $
     \\
&&&& -1 & 0 & 1 & 0
     & 
     $ t\cdot\begin{tikzpicture}[scale=\tinyscale]
     \draw[thick] (0,0) -- (0,1);
     \end{tikzpicture}$ &
     $\begin{tikzpicture}[scale=\tinyscale]
     \draw[thick] (0,0) -- (1,1) -- cycle;
     \end{tikzpicture}
     $
     \\
&&&& -1 & 0 & 0 & -1
     & 
     $ t\cdot\begin{tikzpicture}[scale=\tinyscale]
     \draw[thick] (0,0) -- (0,1);
     \end{tikzpicture}$ &
     $\begin{tikzpicture}[scale=\tinyscale]
     \draw[thick] (0,0) -- (1,0) -- (1,1) -- cycle;
     \end{tikzpicture}
     $
     \\
&&&& -1 & 0 & -1 & -1
     & 
     $ t\cdot\begin{tikzpicture}[scale=\tinyscale]
     \draw[thick] (0,0) -- (0,1);
     \end{tikzpicture}$ &
     $\begin{tikzpicture}[scale=\tinyscale]
     \draw[thick] (0,0) -- (1,0) -- cycle;
     \end{tikzpicture}
     $
     \\
\bottomrule
\end{tabular}
\end{table}

\begin{table}[htb]\centering
\caption{Isolated immaculate line bundles for the hexagon}\label{table:hexagon_isolated}
\newcommand{\tinyscale}{.3}
\begin{tabular}{>{\hspace{.5cm}}cccc<{\hspace{.5cm}}|cc}
\toprule
\multicolumn{4}{c}{$\Pic(X)$ coordinates} & $(\Delta^+$, & $\Delta^-)$\\
\midrule
-2 & -2 & -2 & -2
   &
   pt
   &
   \begin{tikzpicture}[scale=\tinyscale]
   \draw[step=1, black!20, thin] (-.5,-.5) grid (2.5, 2.5);
   \draw[thick] (0,0) -- (2,0) -- (2,2) -- cycle;
   \fill (0,0) circle (3pt);
   \fill (1,0) circle (3pt);
   \fill (2,0) circle (3pt);
   \fill (1,1) circle (3pt);
   \fill (2,1) circle (3pt);
   \fill (2,2) circle (3pt);
   \end{tikzpicture}
\\
-2 & -2 & -2 & 0
   &
   pt
   &
   \begin{tikzpicture}[scale=\tinyscale]
   \draw[step=1, black!20, thin] (-.5,-.5) grid (2.5, 2.5);
   \draw[thick] (0,0) -- (0,2) -- (2,2) -- cycle;
   \fill (0,0) circle (3pt);
   \fill (0,1) circle (3pt);
   \fill (0,2) circle (3pt);
   \fill (1,1) circle (3pt);
   \fill (1,2) circle (3pt);
   \fill (2,2) circle (3pt);
   \end{tikzpicture}
\\
0 & 0 & 0 & -1
   &
   \begin{tikzpicture}[scale=\tinyscale]
   \draw[thick] (0,0) -- (1,1) -- (0,1) -- cycle;
   \end{tikzpicture}
   &
   \begin{tikzpicture}[scale=\tinyscale]
   \draw[thick] (0,0) -- (1,1) -- (1,0) -- cycle;
   \end{tikzpicture}\\
0 & 0 & 0 & 1
   &
   \begin{tikzpicture}[scale=\tinyscale]
   \draw[thick] (0,0) -- (1,1) -- (1,0) -- cycle;
   \end{tikzpicture}
   &
   \begin{tikzpicture}[scale=\tinyscale]
   \draw[thick] (0,0) -- (1,1) -- (0,1) -- cycle;
   \end{tikzpicture}\\
\bottomrule
\end{tabular}
\end{table}

\begin{table}[htb]\centering
\caption{Exceptional sequences of line bundles for the hexagon}\label{table:hexagon_es}
\newcommand{\tinyscale}{.3}
\[D^0=[0,0,0,0]\]
\begin{tabular}{rrrr|rrrr|rrrr|rrrr|rrrr}
\toprule
\multicolumn{4}{c}{$D^1$}  & 
\multicolumn{4}{c}{$D^2$}  & 
\multicolumn{4}{c}{$D^3$}  & 
\multicolumn{4}{c}{$D^4$}  & 
\multicolumn{4}{c}{$D^5$}\\
\midrule
-2 & -1 & -1 & -1  &  -1 & -2 & -1 & 0  &  -2 & -2 & -1 & -1  &  -2 & -2 & -1 & 0  &  -1 & -1 & -2 & -1\\
-1 & -1 & -1 & -1  &  -2 & -2 & -1 & -1  &  -1 & -1 & -2 & -1  &  -2 & -1 & -2 & -2  &  -1 & -2 & -2 & -1\\
-1 & -1 & -1 & -1  &  -2 & -1 & -1 & -1  &  -2 & -2 & -1 & -1  &  -1 & -1 & -2 & -1  &  -2 & -1 & -2 & -2\\
-1 & -1 & -1 & -1  &  -2 & -1 & -1 & -1  &  -1 & -2 & -1 & 0  &  -2 & -2 & -1 & -1  &  -1 & -1 & -2 & -1\\
 -1 & -1 & -1 & -1  &  -1 & -1 & -1 & 0  &  -2 & -1 & -1 & -1  &  -1 & -2 & -1 & 0  &  -1 & -1  & -2 & -1\\
-1 & -1 & 0 & 0  &  -2 & -1 & -1 & -1  &  -1 & -2 & -1 & 0  &  -2 & -2 & -1 & -1  &  -2 & -2 & -1 & 0\\
-1 & -1 & 0 & 0  &  -1 & -1 & -1 & -1  &  -2 & -1 & -1 & -1  &  -1 & -2 & -1 & 0  &  -2 & -2 & -1 & -1\\
-1 & -1 & 0 & 0  &  -1 & -1 & -1 & -1  &  -1 & -1 & -1 & 0  &  -2 & -1 & -1 & -1  &  -1 & -2 & -1 & 0\\
-1 & -1 & 0 & 0  &  -1 & 0 & -1 & -1  &  -1 & -1 & -1 & -1  &  -1 & -1 & -1 & 0  &  -2 & -1 & -1 & -1\\
-1 & -1 & 0 & 0  &  -1 & 0 & -1 & -1  &  0 & -1 & -1 & 0  &  -1 & -1 & -1 & -1  &  -1 & -1 & -1 & 0\\
-1 & -1 & 0 & 0  &  0 & 0 & -1 & -1  &  -1 & 0 & -1 & -1  &  0 & -1 & -1 & 0  &  -1 & -1 & -1 & -1\\
-1 & -1 & 0 & 0  &  0 & 0 & -1 & -1  &  0 & 0 & -1 & 0  &  -1 & 0 & -1 & -1  &  0 & -1 & -1 & 0\\
-1 & 0 & 0 & -1  &  -1 & -1 & -1 & -1  &  -2 & -1 & -1 & -1  &  -2 & -2 & -1 & -1  &  -2 & -1 & -2 & -2\\
-1 & 0 & 0 & -1  &  -1 & -1 & 0 & 0  &  -1 & -1 & -1 & -1  &  -2 & -1 & -1 & -1  &  -2 & -2 & -1 & -1\\
-1 & 0 & 0 & -1  &  -1 & -1 & 0 & 0  &  -1 & 0 & -1 & -1  &  -1 & -1 & -1 & -1  &  -2 & -1 & -1 & -1\\
-1 & 0 & 0 & -1  &  -1 & -1 & 0 & 0  &  0 & 0 & -1 & -1  &  -1 & 0 & -1 & -1  &  -1 & -1 & -1 & -1\\
-1 & 0 & 0 & -1  &  -1 & 0 & 0 & 0  &  -1 & -1 & 0 & 0  &  -1 & 0 & -1 & -1  &  -2 & -1 & -1 & -1\\
 -1 & 0 & 0 & -1  &  0 & -1 & 0 & 0  &  -1 & -1 & 0 & 0  &  -1 & -1 & -1 & -1  &  -2 & -2 & -1 & -1\\
 -1 & 0 & 0 & -1  &  0 & -1 & 0 & 0  &  -1 & -1 & 0 & 0  &  0 & 0 & -1 & -1  &  -1 & -1 & -1 & -1\\                 
\bottomrule
\end{tabular}

\end{table}

From the collection of all tempting $\CR$ we can finally compute the immaculate
locus $\Imm_{\R}(X)$, or rather the lattice points thereof. We only need to
compute the intersection of all complements of the $\CM_{\R}(\CR)$. It is not
difficult to see that this is a union of polyhedra. Since $\CM_{\R}(\CR)$ is a
rational polyhedral cone, we can write it as a finite intersection of
halfspaces. Taking the complement of this cone means taking the union of the
complementary halfspaces. Since we are only interested in the lattice points of
$\Imm_{\R}(X)$, we just move the bounding hyperplane by one away from
$\CM_{\R}(\CR)$ and do not worry about openness of the complement. 
Now we get the
polyhedra giving the lattice points of $\Imm_{\R}(X)$ by picking one
complementary halfspace for every $\CR$ and then intersecting these. Consider
any possible combination and take the union of the resulting polyhedra.

We now restrict our attention to the hexagon example 
   (see Examples~\ref{ex_hexagon_introduce_coordinates} and~\ref{ex_hexagon_tempting_subsets}).
We immediately see that the main bottleneck of
   the algorithm for $\Imm_{\R}(X)$ is the amount of intersections to compute.
There are $34$ tempting $\CR$'s and if every $\CM_{\R}(\CR)$ was bounded by only two
hyperplanes, we would have to compute $2^{34}$ intersections. In fact, all
$\CM_{\R}(\CR)$ are actually bounded by more than two hyperplanes. This issue can be
overcome by building the intersections step by step and eliminating trivial
intersections in between. We start by building the complementary halfspaces of
$\CM_{\R}(\CR_1)$ and $\CM_{\R}(\CR_2)$, then we consider any intersection. If an
intersection is empty already, we eliminate it. Furthermore, we choose the
inclusion maximal intersections. Then we intersect the resulting polyhedra with
the complementary halfspaces of $\CM_{\R}(\CR_3)$ and so on.

Thus we have computed the immaculate loci $\Imm_{\Z}(X) = \Imm_{\R}(X)$.
They are equal to a union of three unbounded polyhedra and four isolated lattice points 
that are listed in Table~\ref{table:hexagon_isolated}.
Each unbounded polyhedron consists of four parallel lines, that is lattice lines. 
The exact lines, together with their polytopes $(\Delta^+,\Delta^-)$ are depicted in
Table~\ref{table:hexagon_lines}.
Each pair of quadruples of lines intersects in four points.

Now it is easy to compute all exceptional sequences that are contained in the
projection of the cube $\pi([-1,0]^6)\subseteq\Cl(X)$. One just collects the
lattice points in the projected cube and then runs a depth first search. There
are $228$ exceptional sequences of length six in the projected cube. Under the
group action on the hexagon these $228$ exceptional sequences correspond to
$19$ orbits of size $12$. In Table~\ref{table:hexagon_es} we list one
representative from each orbit.
Note that we do not need to use the four isolated points for these exceptional sequences.
This is different than in the case of the splitting fans, for example the Picard rank $2$ case
   (see \cite{CostaMiroRoig}).

\nocite{orlov92}
\addcontentsline{toc}{section}{References}
\bibliographystyle{alpha}
\bibliography{immaculate}

\end{document}